\documentclass[twoside,11pt]{article}

\usepackage{jmlr2e}

\input{style_TS.sty}

\jmlrheading{17}{2016}{1-37}{2/16; Revised 7/16}{9/16}{Tobias Sutter, Arnab Ganguly, and Heinz Koeppl}

\ShortHeadings{Variational inference for hidden diffusion processes}{Sutter, Ganguly and Koeppl}
\firstpageno{1}

\begin{document}

\title{A Variational Approach to Path Estimation and Parameter Inference of Hidden Diffusion Processes}

\author{\name Tobias Sutter \email sutter@control.ee.ethz.ch \\
       \addr Department of Electrical Engineering and Information Technology\\
       ETH Zurich, Switzerland
       \AND
       \name Arnab Ganguly \email aganguly@lsu.edu \\
       \addr Department of Mathematics\\
       Louisiana State University, USA
       \AND
       \name Heinz Koeppl \email heinz.koeppl@bcs.tu-darmstadt.de \\
       \addr Department of Electrical Engineering and Information Technology\\
       TU Darmstadt, Germany}

\editor{Manfred Opper}

\maketitle

\begin{abstract}
We consider a hidden Markov model, where the signal process, given by a diffusion, is only indirectly observed through some noisy measurements. The article develops a variational method for approximating the hidden states of the signal process given the full set of observations.\ This, in particular, leads to  systematic approximations  of the smoothing densities of the signal process.\ The paper then demonstrates how an efficient inference scheme, based on this variational approach to the approximation of the hidden states, can be designed to estimate the unknown parameters of  stochastic differential equations.\ Two examples at the end illustrate the efficacy and the accuracy of the presented method. 
\end{abstract}


\section{Introduction} \label{sec:introduction}
Diffusion processes modeled by stochastic differential equations (SDEs) appear in several disciplines varying from mathematical finance to systems biology. For example, in systems biology  stochastic differential equations are used for efficient modeling of the states of the chemical species in a reaction system when they are present in high abundance \cite{ref:Wilkinson-06}. Oftentimes, the state of the system or the signal process is not directly observed, and inference of the state trajectories and parameter of the system has to be achieved based on noisy partial observations. Typically, in such a scenario, the observation data is conveniently modeled as a function of the hidden state corrupted with independent additive noise. However, generalizations of this basic setup, which, for example, could include stronger coupling between the hidden signal and the observation processes, are often used for modeling more complex phenomena.   

In such a  model optimal filtering theory concerns itself with recurrent estimation of the current state of the hidden signal process given the observation data until the present time. This is particularly useful in tracking problems where the estimation of the current location of an object needs to be constantly updated as new noisy information flows in.  On the other hand, optimal smoothing involves the class of methods which can be used in reconstruction of any past state of the signal process given a set of measurements up to the present time. More specifically, given the signal process $X$ and the observation process $Y$, filtering theory entails computation of the conditional expectations of the form $\Expec{\phi(X_t)| \SC{F}^Y_t}$, where $\{\SC{F}^Y_t\}$ denotes the filtration generated by the process $Y$. The $\sigma$-algebra $\SC{F}^Y_t$ contains all the information about the observation process $Y$ up to the present time $t$. Smoothing, however, involves evaluation of the conditional expectations of the form $\Expec{\phi(X_s)| \SC{F}^Y_t}$, where $s < t$.  The smoothing techniques can also be viewed as tools of estimation of the current state given a data set which includes future observations. This interpretation is particularly relevant in statistics, where such techniques are essentially the means of computing certain posterior conditional densities given the observation set. The present article focusses on a variational approach to this smoothing problem and later employs the method for estimation of parameters of diffusion processes. 

Evaluation of such conditional expectations or densities are quite difficult, since they are often solutions of suitable (stochastic) partial differential equations. These are usually infinite-dimensional problems and analytical solutions are generally impossible.
 Hence, effort has been directed toward developing of a variety of numerical schemes for efficient approximation of these conditional densities. While Markov chain Monte Carlo methods for inference use discretization of the given SDE for writing down an approximate likelihood \cite{ref:Kushner-01, ref:pages-05, ref:Andrieu-10}, particle methods approximate the (posterior) conditional densities by suitably weighted point masses \cite{ref:Crisan-99, ref:DelMoral-01,ref:Bain-09}. However, these methods often rely on a suitable discretization of the problem which is mostly done in an ad-hoc way. Since a theoretical framework for obtaining approximations is not present, the approximation error might be difficult to quantify. 
 
In contrast, the present paper focusses on a variational approach to this estimation problem. The main idea in such a method is to approximate the (posterior) conditional probability distribution of the system's state (given the observed data) by an appropriate Gaussian distribution, where the optimal parameters for the Gaussian distribution are obtained by minimizing the relative entropy (or Kullback-Leibler distance) between the posterior process and a suitable approximating SDE. Earlier works like \cite{ref:Archambeau-07, ref:Archambeau-08, ref:Archambeau-11, ref:Cseke-13, ref:Opper-15} considered the case when the signal process is modeled by an SDE with a constant diffusion term. The advantage of working with a constant diffusion term is that it  implies that the approximating SDE will simply have a linear drift so that  marginals are distributed as Gaussian. This simple expression of the SDE with a linear drift makes the subsequent optimization problem for finding the suitable parameters for this approximating SDE easier. However, since most physical phenomena cannot be realistically modeled by  SDEs with constant diffusion term, there is a pressing need of extending the approach to general SDEs. One natural but naive approach in this regard could be to freeze the diffusion term at an appropriate value, that is, to take the zeroth order expansion of the diffusion coefficient. Although simple to implement, the efficacy of the method is not guaranteed by theoretical results and will vary from case to case, and a reasonable error analysis might require unreasonably restrictive conditions on the model. 

Instead, the  present article delves much deeper in to the problem and develops methods for finding the optimal approximating SDE such that the relative entropy between it and the true posterior process is minimized subject to the condition that the marginals of the former follow Gaussian distributions.  The main obstacle that needs to be overcome in this approach stems from the fact that unlike the previous case, the approximating SDE here cannot be taken to be the one with a linear drift; and a suitable expression of it needs to be found so that the marginals are still Gaussian. This has been achieved in Theorem \ref{thm:brigo_mixture_multi}. 
In fact, our work outlines the most general techniques for approximating the posterior density by any  density from the exponential family or mixture of exponential families.  In this connection we would like to note that the reason for  requiring that the marginals follow a Gaussian distribution or more generally, a distribution from the exponential family because this results in a finite-dimensional smoother which can be used for approximating a wide range of distributions.

 It should be noted that the variational method considered here is different from the so-called extended Kalman filter (EKF) in two ways: first, EKF is employed for filtering problems; but more importantly, EKF starts by linearizing the signal (prior) SDE and then freezing its diffusion term, while the variational approach is concerned with approximation of the posterior SDE. Therefore even though in the constant diffusion term case, the approximating SDE happens to have linear drift and thus resulting in a Gaussian smoother, it is not based on the same philosophy behind the EKF.  And as mentioned before, in the non-constant diffusion term case although our method can be used to obtain a finite-dimensional smoother, in particular, a Gaussian smoother, it completely avoids any form of linearization of the given SDE or subsequent freezing of the diffusion term.

In our paper this variational approximation method has been formulated as an optimal control problem. The advantage of this theoretical framework is that necessary conditions for global optimality are then obtained by employing the Pontryagin maximum principle. This leads to considerable computational advantages of the variational method compared to numerically solving the underlying (stochastic) PDEs, that is highlighted by two examples.

The later part of the paper focusses on the important topic of parameter inference of SDEs. The above scheme of estimating the hidden states and the smoothing densities  is cleverly used in designing an efficient method for estimating parameters of SDEs. In particular, the paper proposes an iterative EM-type algorithm which aims to compute approximate maximum likelihood estimates of the parameters in a tractable way. Two illustrative examples, which are important  in mathematical finance, demonstrate the accuracy and efficiency of the proposed algorithms. Future projects will address more complicated models.

%Layout:
The layout of this article is as follows: In Section~\ref{sec:model:setup} we formally introduce the problem setting. We consider as a running example throughout the manuscript a geometric Brownian motion. The variational approximation idea is motivated in Section~\ref{sec:motivation} leading to a specific class of optimization problems that is addressed in Section~\ref{sec:prescribed_marginal_law}. It is then reformulated in Section~\ref{sec:optimal:control:problem:formulation} as an optimal control problem and necessary conditions for optimality are derived. Section~\ref{sec:parameter:inference} explains how the variational approximation can be used to infer unknown parameters of the model. Section~\ref{sec:discrete:measurements} discusses the presented variational approximation in the context of a discrete time measurement model. The theoretical results are applied in Section~\ref{sec:examples} to two examples: a geometric Brownian motion and to the Cox-Ingersoll-Ross process. We finally conclude with some remarks and directions for
future work in Section~\ref{sec:conclusion}. Certain technical proofs are relegated to the appendix.

%Notation:
\textit{Notation.} Hereafter, $\In$ is the n-dimensional identity matrix and $\Ei$ is the $n\times n$ matrix where the $ii$-th entry is one and zero elsewhere. We let $\sym(n,\R)$ and $\text{GL}(n,\R)$ be respectively the set of symmetric and invertible $n \times n$ matrices with real entries. For matrices $A,B\in\R^{n\times n}$ let $\inprod{A}{B}:=\tr(A\transp B)$ denote the Frobenius inner product. For a vector $b\in\R^n$ and a positive definite matrix A, we employ the norm  $\norm{b}_A := \sqrt{b\transp A^{-1}b}$. We define the standard $n-$simplex as $\Delta_{n}:=\left\{  x\in\R^{n} : x\geq 0, \sum_{i=1}^{n} x_{i}=1\right\}$. Let $\SC{C} := C([0,T], \R^n)$ denote the space of continuous functions on $[0,T]$ taking values in $\R^n.$ 
Let $\S$ be a metric space, equipped with its  Borel $\sigma$-field $\mathcal{B}(S)$. The space of all probability measures on $(\S, \mathcal{B}(\S))$ is denoted by $\mathcal{P}(\S)$. The relative entropy (or Kullback-Leibler divergence) between any two probability measures $\mu, \nu \in \mathcal{P}(\S)$ is defined as
\begin{equation*}
\KL{\mu}{\nu} := \left\{ \begin{array}{ll}
\int \log\left( \frac{\drv\mu}{\drv \nu} \right) \drv\mu, &\text{if } \mu \ll \nu\\
+\infty, & \text{otherwise},
\end{array} \right.
\end{equation*}
where $\ll$ denotes absolute continuity of measures and $\tfrac{\drv\mu}{\drv \nu}$ is the Radon-Nikodym
derivative. By convention \emph{measurable} means \emph{Borel-measurable} in the sequel.
Given an $\S$-valued random variable $X$ with Law($X$) $= \mu \in \mathcal{P}(\S)$, let $\CExpec{X}{\mu}$ denote the expectation of $X$.

%%%%%%%%%%%%%%%%%%%%%%%%%%%%%%%%%%%%%%%%%%%%%%%%%%%%%%%%%%%%%%%%%%%%%%%%%%%%%%%%%%%%%%%%%%%%%%%%%%%%%%%%%%%%
%%%%%%%%%%%%%%%%%%%%%%%%%%%%%%%%%%%%%%%%%%%%%%%%%%%%%%%%%%%%%%%%%%%%%%%%%%%%%%%%%%%%%%%%%%%%%%%%%%%%%%%%%%%%
\section{Model setup} \label{sec:model:setup}
As usual, we will work on a complete probability space  $(\Omega, \SC{F}, \mathbb P)$ equipped with a filtration $\{\SC{F}_t\}$ satisfying the usual conditions, that is, $\{\SC{F}_t\}$ is complete, right continuous and contains all the $\mathbb P$-null sets. 
The basic objects in our study consist of a signal process $X$ and an observation process $Y$, both of which are assumed to be $\{\SC{F}_t\}$- adapted. The unobserved signal process $X$ is modeled by  the following stochastic differential equation describing the state evolution of a dynamical system:
\begin{equation} \label{sec:one:eq1}
	 \drv X_{t} = f(X_{t})\drv t+\sigma(X_{t})\drv W_{t},\quad X_{0}=x_{0}, \quad 0\leq t \leq T, 
	\end{equation}
where  $f:\R^n\to\R^n$, $\sigma:\R^n\to\R^{n\times n}$, and $W$ is an $n$-dimensional Brownian motion independent of $x_{0}$. The observation process $Y$ is modeled as noisy measurements of some function of the  signal process $X$. Mathematically, $Y$ is defined as	\begin{equation}\label{eq:measurement:model:cts}
	Y_{t} = \int_{0}^{t} h(X_{s}) \drv s + B_{t},
	\end{equation}
	where $h:\R^{n}\to\R^{m}$ is called the observation function and $B$ is an $m$-dimensional Brownian motion independent of $x_{0}$ and $W.$
	 
 \begin{assumption} \label{ass:sde_main_ass} \em{
 We stipulate that 
 \begin{enumerate}[(i)]
 \item $f$ and $\sigma$ are globally Lipschitz;
 \item and $h$ is twice continuously differentiable.
 \end{enumerate} }
\end{assumption}
It is known \cite{ref:Kallenberg-02} that under Assumption~\ref{ass:sde_main_ass} there exists a unique strong solution to the SDE~\eqref{sec:one:eq1}. Given the observed data up to some time $T$, $\{Y_s: s\leq T\}$, the goal of the paper is to outline an approximation method for the smoothing density, $\Ps(x,t)$, which is the conditional probability density of $X_t$ given $\{Y_s: s\leq T\}$. In other words,   the smoothing density is defined by the equation:
	\begin{equation} \label{eq:def:smoothing:density}
	\Expec{\phi(X_{t})|\FTY} = \int \phi(x) \Ps(x,t) \ \drv x,
	\end{equation}
	up to a.s. equivalence, where $\phi$ is any bounded measurable function from $\R^n$ to $\R$ and $\{\SC{F}_t^Y\}$ denotes the filtration generated by the process $Y$. 

More generally, we will be interested in approximating the full conditional probability measure on the path space, $\SC{C} \equiv C([0,T], \R^n)$. To describe this mathematically, assume that a regular conditional probability measure $\Prob{\cdot | \SC{F}^Y_T}$ is chosen. Then there exists a measurable probability kernel $y \in \SC{C} \rt \post{(\cdot, y)} \in \mathcal{P}(\SC{C})$ such that for any measurable set $\SC{A} \subset \SC{C}$, 
	$$\Prob{X_{[0,T]} \in \SC{A} | \SC{F}^Y_T} =  \post(A,Y_{[0,T]}).$$
	
	Given the observation process up to time $T$, $Y_{[0,T]}$, we now describe a characterization of the probability measure $\post(\cdot, Y_{[0,T]})$, which will play a pivotal role for our purposes. The probability measure $\post(\cdot, Y_{[0,T]})$ is actually the distribution of a diffusion process $\bar{X}^{T}$ on $\SC{C}$, and the latter is obtained by a modification of the original signal process $X$:
	\begin{align}\label{eq:postdiff}
	d\bar{X}^{T}_t = g(\bar{X}^{T}_t, t) dt + \sigma(\bar{X}^{T}_t) d\bar{W}_t, \quad \bar{X}^T_0=x_{0},
	\end{align} 
where $\bar{W}$ is an $\{\SC{F}_t\}$-adapted Brownian motion that is independent of $Y$. Notice that the diffusion coefficient of the above SDE (which we will henceforth call the posterior SDE or posterior diffusion) is same as that of the original SDE, and the drift of this posterior SDE is time-dependent and is obtained as
\begin{equation}\label{newdrift}
	 g(x,t):=f(x)+a(x)\nabla\log w(x,t),
	\end{equation}
where $a(x):=\sigma(x) \sigma(x)\transp$\!.
We give details about the (random) function $w$ a little later, but the important point to note  here is that the new drift function is the old drift function with an extra additive term, and the observation process $Y_{[0,T]}$ enters into the characterization of $\post(\cdot, Y_{[0,T]})$ only through $w$.

	To see this characterization of $\post(\cdot, Y_{[0,T]})$, we first look at the usual filtering density $\Pf(x,t)$, which is naturally defined by 
	\begin{equation} \label{eq:def:filter:density}
	\Expec{\phi(X_{t})|\FtY} = \int \phi(x) \Pf(x,t) \ \drv x.
	\end{equation}
	Under suitable technical conditions, the filter density $\Pf$ satisfies the Kushner-Stratonovich equation (for example, see \cite{ref:Stratonovich-60,ref:Kushner-67,ref:Bain-09}). For our purposes, however, it is convenient to work with the unnormalized filter density $p(x,t)$, that is, $\Pf(x,t) = p(x,t) ( \int_{\R^{n}} p(x,t) \drv x)^{-1}$, which satisfies the so-called Zakai equation \cite{ref:Zakai-69}
	\begin{equation}\label{eq:Zakai}
	\left\{ \begin{aligned}
	\drv p(x,t) &=  \mathcal{A}^{*} p(x,t) \drv t + p(x,t) h(x)\transp \drv Y_{t} \\
	p(x,0) &= p_{0}(x).
	\end{aligned} \right.
	\end{equation}
      Here $p_{0}$ denotes the density of $x_{0}$ and $\mathcal{A}^{*}$ is the adjoint of the infinitesimal generator of the process $X$ given by
	$ \mathcal{A}\psi(x)=\sum_i f_i(x)\frac{\partial}{\partial x_i}\psi(x) + \frac{1}{2}\sum_{i,j}a_{i,j}(x)\frac{\partial^2}{\partial x_i \partial x_j}\psi(x)$ for $\psi \in \SC{C}_0^2(\R^n,\R)$. We next consider the backward stochastic partial differential equation (SPDE)
	\begin{equation} \label{eq:backward:SPDE:v}
	\left\{ \begin{aligned}
	\drv w(x,t) &=  -\mathcal{A} w(x,t) \drv t - w(x,t) h(x)\transp \drv Y_{t} \\
	w(x,T) &= 1.
	\end{aligned} \right.
	\end{equation}
	Conditions about existence of solutions to \eqref{eq:Zakai} and \eqref{eq:backward:SPDE:v} can be found in \cite{ref:Pardoux-81}. It is well known \cite[Corollary~3.8]{ref:Pardoux-81} that the smoothing density can be expressed as
	\begin{equation} \label{eq:smoothing:density:char}
	\Ps(x,t)=\frac{p(x,t)w(x,t)}{\int_{\R^{n}}p(x,t)w(x,t)\drv x}.
	\end{equation}
Now by using \eqref{eq:Zakai}, \eqref{eq:backward:SPDE:v} and \eqref{eq:smoothing:density:char}, it can be shown\footnote{See Appendix~\ref{app:smoothing:SDE} for a detailed derivation.} that  the smoothing density solves the following Kolmogorov forward equation
\begin{equation} \label{eq:PDE:posterior2}  
\left( \frac{\partial}{\partial t}+\sum_{i} \frac{\partial}{\partial x_{i}} g(x,t)-\frac{1}{2}\sum_{i,j}\frac{\partial^{2}}{\partial x_{i} \partial x_{j}} a_{ij}(x) \right)\Ps(x,t )=0,
	\end{equation}
	with the drift term $g$ defined by \eqref{newdrift}.
%	\begin{equation}\label{newdrift}
%	 g(x,t):=f(x)+a(x)\nabla\log v(x,t).
%	\end{equation}
In other words, the conditional probability measure $\post(\cdot,Y_{[0,T]})$ on $\SC{C}$ is induced by the diffusion process $\bar{X}^{T}$ as defined in \eqref{eq:postdiff}. 

Evaluating $\post(\cdot,Y_{[0,T]})$ is what is known as the path estimation problem. Except for a few simple cases, the SPDEs, that are involved in this estimation of the hidden path,  are  analytically intractable. The variational approach that we undertake in this paper actually has the goal of approximating $\post(\cdot,Y_{[0,T]})$. Toward this end, a natural objective is to approximate $\post(\cdot,Y_{[0,T]})$ by  a probability measure such that the corresponding marginals of the latter come from a known family of distributions (e.g, exponential family). As a result, the marginal of this approximating probability measure at time $t$ approximates the smoothing density $\Ps(x,t)$. The procedure adopted in this article involves finding the optimal parameters of this approximating distribution by minimizing the relative entropy between the posterior distribution and the approximating one.

\subsection{Example: Geometric Brownian Motion} \label{sec:GBM:example:1}
 We present as a running example throughout the article  the geometric Brownian motion that is used to model stock prices in the Black-Scholes model, see \cite{ref:Shiryaev-99}. The system dynamics \eqref{sec:one:eq1} is given by a one-dimensional geometric Brownian motion
	\begin{equation} \label{eq:system:GBM1}
	 	\drv X_{t} = \kappa X_{t}\drv t+\lambda X_{t} \drv W_{t},\quad X_{0}=x_0\sim \log\mathcal{N}(\mu,\sigma), 
	\end{equation}
	for $0\leq t \leq T$, $\lambda, \kappa >0$ and an observation process \eqref{eq:measurement:model:cts} defined by
	\begin{equation}\label{eq:measurement:model:GBM1}
	Y_{t} = \int_{0}^{t} X_{s} \drv s + B_{t}.
	\end{equation}
It is straightforward to see that Assumption~\ref{ass:sde_main_ass} holds in this setting.

%%%%%%%%%%%%%%%%%%%%%%%%%%%%%%%%%%%%%%%%%%%%%%%%%%%%%%%%%%%%%%%%%%%%%%%%%%%%%%%%%%%%%%%%%%%%%%%%%%%%%%%%%%%%
%%%%%%%%%%%%%%%%%%%%%%%%%%%%%%%%%%%%%%%%%%%%%%%%%%%%%%%%%%%%%%%%%%%%%%%%%%%%%%%%%%%%%%%%%%%%%%%%%%%%%%%%%%%%
\section{Variational approximation: Motivation} \label{sec:motivation}
Let $\prior$ denote the distribution of the original signal process $X$ on $\SC{C}$, that is, for a measurable $\SC{A} \subset \SC{C}$, $\prior(A) \equiv \Prob{X_{[0,T]}\in A}$. Define the two terms
\begin{align}
H_{T}(X_{[0,T]},y) &:= -h(X_{T})y_{T} +  \int_{0}^{T} y_{s} \drv h(X_{s}) + \frac{1}{2} \int_{0}^{T} \| h(X_{s}) \|^{2} \drv s  \\
I(H_{T}(\cdot,y)) &:= -\log\left( \int \exp\left( - H_{T}(\cdot,y) \right)\drv\prior \right). \label{eq:information:term}
\end{align}
Let $y$ be a sample path of the observation process $Y$ on the interval $[0,T]$. Then notice that by the pathwise Kallianpur-Striebel formula (or the Bayes formula), we have
\begin{align*}
\frac{\drv \post(\cdot,y)}{\drv \prior}& = \frac{\exp(-H_{T}(\cdot,y))}{\int \exp( -H_{T}(\cdot,y)) \drv\prior  } = \frac{\exp(-H_{T}(\cdot,y))}{L(y)}.
\end{align*}
where $L(y)=\int \exp( -H_{T}(\cdot,y)) \drv\prior.$
Consequently, $L(y)$ can be interpreted naturally  as the likelihood of the path $y$, or equivalently,  $I(H_{T}(\cdot,y))$ is viewed as the negative log-likelihood of the sample path $y$. The term $H_{T}(X_{[0,T]},y)$ can be interpreted as the $X$-conditional information and the information in the observation that $Y=y$, see \cite{ref:Mitter-03} for more details.
Now for any probability measure $Q$\footnote{$Q$ will be called the approximating probability measure in the sequel.} on $C([0,T],\R)$, the relative entropy between $Q$ and $\post(\cdot,y)$ can be expressed by the following lemma.
 
\begin{lemma} \label{Lemma:KL:cost}
$\KL{Q}{\post(\cdot,y)} = - I(H_{T}(\cdot,y)) + \KL{Q}{\prior} + \CExpec{H_{T}(\cdot,y)}{Q}$.  
\end{lemma}
\indent \textit{Proof.}
The proof essentially follows the one in \cite[Lemma~2.2.1]{ref:vanHandel-07}.
Splitting the relative entropy and using the pathwise Kallianpur-Striebel formula yields
\begin{align*}
&\KL{Q}{\prior} 	=	 	\int \left[ \log\left( \frac{\drv Q}{\drv \post(\cdot,y)} \right) + \log\left( \frac{\drv \post(\cdot,y)}{\drv \prior} \right)  \right] \drv Q  \\
					&\hspace{8mm}= 	\KL{Q}{\post(\cdot,y)} + \int \log\left( \frac{\drv \post(\cdot,y)}{\drv \prior} \right) \drv Q  \\
					&\hspace{8mm}= 	\KL{Q}{\post(\cdot,y)} + \int \log \left( \frac{\exp(-H_{T}(\cdot,y))}{\int \exp( -H_{T}(\cdot,y))\drv \prior}  \right) \drv Q \\
					&\hspace{8mm}= 	\KL{Q}{\post(\cdot,y)} - \CExpec{H_{T}(\cdot,y)}{Q} - \log\left( \int \exp( -H_{T}(\cdot,y))\drv \prior \right). \qquad\endproof
\end{align*}
Mitter and Newton \cite{ref:Mitter-03} provide an information-theoretic interpretation to this result. They interpret the term \eqref{eq:information:term} as the \emph{total information} available to the estimator $Q$ through the sample path $y$. On the other hand, they call the quantity $\mathcal{F}(Q,y):=\KL{Q}{\prior} + \CExpec{H_{T}(\cdot,y)}{Q}$ the \emph{apparent information} of the estimator. By non-negativity of the relative entropy $\mathcal{F}(Q,y) \geq I(H_{T}(\cdot,y))$ with equality if and only if $Q=\post(\cdot,y)$. In this sense, a suboptimal estimator appears to have access to more information than is actually available. %Also note that the term $H_{T}(x,y)$ represents likelihood function for $x$, associated with observation $y$, which is the key idea for the parameter inference part, presented in Section~\ref{sec:parameter:inference}.

Since the total information $I(H_{T}(\cdot,y))$ does not depend on $Q$, minimizing the relative entropy between $Q$ and $\post(\cdot,y)$ over a class of probability measures $Q$ is equivalent to minimizing  the apparent information $\mathcal{F}(Q,y)$. This motivates to consider an approximating distribution $Q$ on $\SC{C}$ that is characterized as the solution to the following optimization problem: 
 
\begin{problem} \label{prob:optimization:problem}
Minimize $\KL{Q}{\prior} + \CExpec{H_{T}(\cdot,y)}{Q}$ subject to
\begin{enumerate}[(i)]
\item $Q$ is a probability distribution induced by an SDE of the form \label{eq:problem:i} 
\begin{equation} \label{eq:approx_SDE}
\drv Z_{t}  =u(Z_{t},t)\drv t + \sigma(Z_{t})\drv W_{t},\quad Z_{0}=x_{0},\quad 0\leq t \leq T;
\end{equation}
\item The marginals of $Q$ at time $t$, i.e., the distribution of $Z_{t}$, belong to a chosen family of distributions. \label{eq:problem:ii}
\end{enumerate}  
\end{problem}
We will show in the remainder of this article how Problem~\ref{prob:optimization:problem} can be restated as an optimal control problem, which leads to a standard formulation of necessary optimality conditions in terms of Pontryagin's maximum principle.

Note that the objective function of Problem~\ref{prob:optimization:problem} is known to be strictly convex with respect to $Q$, see \cite{ref:Csiszar-75}. The constraint~\eqref{eq:problem:ii} restricts the feasible set approximating distributions $Q$ to a nonconvex set. Note that such problems (i.e., absence of constraint~\eqref{eq:problem:i} have been studied in the literature \cite{ref:Pinski-15}). In our setting, the set of feasible solutions is also coupled with the first constraint \eqref{eq:problem:i}, that parametrizes the feasible set of distributions in terms of the drift function $u$. This coupling is investigated in Section~\ref{sec:prescribed_marginal_law}, in particular Theorem~\ref{thm:brigo_mixture_multi} characterizes the set of all drift terms $u$ such that the distribution induced by \eqref{eq:approx_SDE} has finite dimensional marginals that belong to a given family of distributions. 
Hence, Problem~\ref{prob:optimization:problem} can alternatively be interpreted as minimizing the objective function over a class of drift functions $u$ that induce $Q$ via \eqref{eq:approx_SDE} and such that $Q$ satisfies constraint~\eqref{eq:problem:ii}.
For example, if the goal is to approximate the posterior distribution $\Pi_{post}$ by a distribution $Q$ whose marginals are normal distributions, then one aims to find a drift term $u$ such that the objective function is minimized and such that the solution $Z_{t}$ to \eqref{eq:approx_SDE} admits a normal distribution.

\begin{remark}
{\rm Notice that the unconstrained optimization of the objective function in Problem \ref{prob:optimization:problem} with respect to $Q$ will simply yield the minimizer $Q$ to be $\post$. Since, as discussed in the beginning of Section \ref{sec:model:setup}, $\post$ is induced by the SDE, \eqref{eq:postdiff}, the constraint $(i)$ in Problem \ref{prob:optimization:problem} is essentially inbuilt. In other words, it is the constraint $(ii)$ which plays the crucial role in the methods outlined in this paper.

}

%\rm{We discuss the behaviour of Problem~\ref{prob:optimization:problem} when some of the constraints are removed. 
%\begin{enumerate} 
%\item If the constraint \eqref{eq:problem:ii} is omitted, i.e., the relaxed optimization problem of minimizing $\KL{Q}{\prior}$ $+ \CExpec{H_{T}(\cdot,y)}{Q}$ subject to $Q$ being a probability distribution induced by \eqref{eq:approx_SDE} is considered, then the optimal choice of $u$ in \eqref{eq:approx_SDE} must coincide with the drift in \eqref{eq:postdiff} that is given by \eqref{newdrift}.
%\item If both constraints \eqref{eq:problem:i} and \eqref{eq:problem:ii} are removed, i.e., we consider the problem of minimizing $\KL{Q}{\prior} + \CExpec{H_{T}(\cdot,y)}{Q}$ subject to $Q$ being a probability distribution on $C([0,T],\R)$, then the unique optimizer is given by $\post$, which follows directly from Lemma~\ref{Lemma:KL:cost}. 
%\end{enumerate} }  
\end{remark}
The objective function in Problem~\ref{prob:optimization:problem}, in particular the relative entropy between the approximating distribution $Q$ and the prior distribution $\prior$ can be simplified, since due to the constraint~\eqref{eq:problem:i} the underlying SDEs \eqref{eq:approx_SDE} and \eqref{sec:one:eq1} share the same diffusion coefficient. 
In view of \eqref{eq:approx_SDE} and \eqref{sec:one:eq1}, consider two SDEs for $0\leq t \leq T$
	\begin{align*}
	  \drv X_{t} &= f(X_{t})\drv t+\sigma(X_{t})\drv W_{t},\quad \drv Z_{t} = u(Z_{t},t)\drv t+\sigma(Z_{t})\drv W_{t},\quad X_{0}=Z_{0} = x_0, 
	\end{align*}
	with $u:\R^n\times\R\to\R^n, f:\R^n\to\R^n, \sigma:\R^n\to\R^{n\times n},$ $W$ an $n$-dimensional Brownian motion independent of $x_{0}$ and both SDEs satisfying Assumption~\ref{ass:sde_main_ass}. Let $(\Omega,\mathcal{F}_T,P)$ be a probability space, where $\mathcal{F}_T$ is the sigma algebra $\sigma(W_{s} : s\leq T)$ and let $\prior$ and $Q$ denote the the laws of $X_{t}$ and $Z_{t}$ with respect to $P$.	
It follows by Girsanov's Theorem \cite{ref:Oksendal-03}, that
	\begin{equation*}
	\CExpec{\log\left( \frac{\drv Q}{\drv \prior} \right)}{Q}  = \frac{1}{2}\CExpec{\int_0^T \varphi(s,\omega)\transp \varphi(s,\omega) \drv s}{Q},
	\end{equation*}
	where
	  $\varphi(s,\omega) := \sigma(Z_s(\omega))^{-1}\left(u(X_s(\omega))-f(X_s(\omega))\right).$
	Therefore, the relative entropy between $Q$ and $\prior$ is
	\begin{equation*} 
	\KL{Q}{\prior}=   \frac{1}{2} \CExpec{\int_0^T \norm{u(X_s,s)-f(X_s)}_{a(X_s)}^2 \drv s}{Q},
	\end{equation*} 
	where $\left\lVert u(x,s)-f(x) \right\rVert^2_{a(x)}:=\left( u(x,s)-f(x) \right)\transp {a(x)}^{-1}\left( u(x,s)-f(x) \right)$. Hence, the objective function in Problem~\ref{prob:optimization:problem} can be expressed as
\begin{equation}\label{eq:obj_fct}
\begin{aligned}
&\KL{Q}{\prior} + \CExpec{H_{T}(\cdot,y)}{Q}\\
&\quad = \int_{0}^{T} \mathbb{E}_{Q}\left[ \frac{1}{2}\norm{u(X_{t},t)-f(X_{t})}_{a(X_{t})}^2 + y_{t}\bigg( u(X_{t},t)\transp \nabla h(X_{t}) \right. \\
&\quad \quad \quad + \frac{1}{2}\sigma(X_{t})\transp \nabla^{2} h(X_{t})\sigma(X_{t}) \bigg) + \left. \frac{1}{2} \| h(X_{t}) \|^{2} \right] \drv t  - y_{T}\CExpec{h(X_{T})}{Q},
\end{aligned}
\end{equation}
where the last equality is due to Fubini's Theorem and It\^o's Lemma. 
The two coupling constraints \eqref{eq:problem:i} and \eqref{eq:problem:ii} in Problem~\ref{prob:optimization:problem} are studied in the next section and will finally allow us to reformulated Problem~\ref{prob:optimization:problem} as an optimal control problem.

%%%%%%%%%%%%%%%%%%%%%%%%%%%%%%%%%%%%%%%%%%%%%%%%%%%%%%%%%%%%%%%%%%%%%%%%%%%%%%%%%%%%%%%%%%%%%%%%%%%%%%%%%%%%
%%%%%%%%%%%%%%%%%%%%%%%%%%%%%%%%%%%%%%%%%%%%%%%%%%%%%%%%%%%%%%%%%%%%%%%%%%%%%%%%%%%%%%%%%%%%%%%%%%%%%%%%%%%%
\section{Multi-dimensional SDE with prescribed marginal law} \label{sec:prescribed_marginal_law}
This section establishes conditions on the drift function in the approximate SDE \eqref{eq:approx_SDE} such that the induced marginal distributions evolve in a given exponential family.
 
	\begin{definition}[Exponential family] \label{def:muli_exp_fam}
	  Let $\Hi_1,\hdots,\Hi_m$ be Hilbert spaces and let $\Hi=\prod_{i=1}^m \Hi_i$ be endowed with the inner product $\inprod{\cdot}{\cdot}$. Let the functions $c_i: \R^n\to \Hi_i$ for $i=1,\hdots,m$ be linearly independent, have at most polynomial growth, be twice continuously differentiable and denote $c(x)=(c_1(x),\hdots,c_m(x))$. Assume that the convex set \\
	  $\Gamma:=\left\{ \Theta\in \Hi \ : \ \psi(\Theta)=\log \int\exp\left( \inprod{\Theta}{c(x)} \right)\drv x<\infty \right\}$
	   has non-empty interior. Then 
	    \begin{equation*}
	     \text{EM}(c)= \{ p(\cdot,\Theta), \Theta \in \Lambda \}, \quad p(x,\Theta):=\exp\left(\inprod{\Theta}{c(x)}-\psi(\Theta) \right),
	    \end{equation*}
	    where $\Lambda\subseteq \Gamma$ is open, is called an \textit{exponential family of probability densities}.  
	  \end{definition}
	  
	  \begin{definition}[Mixture of exponential families]
	  Let $\text{EM}(c^{(i)})$ for $i=1,\hdots,k$ be exponential families according to Definition \ref{def:muli_exp_fam}. Then
	  \begin{equation*}
	  \begin{aligned}
	   \text{EM}(c^{(1)},\hdots,c^{(k)}) &= \bigg\{ \sum_{\ell=1}^k\nu_\ell p_\ell(\cdot,\Theta^{(\ell)}) \ : \ p_\ell(\cdot,\Theta^{(\ell)})\in \text{EM}(c^{(\ell)}), \ \nu \in \Delta_{k}  \bigg\}
	   \end{aligned}
	  \end{equation*}
	   is called a \textit{mixture of $k$ exponential families of probability densities}.  
	 \end{definition} 
	  Consider the stochastic differential equation \eqref{eq:approx_SDE},
	  where  $u:\R^n\times\R\to\R^n$, $\sigma:\R^n\to\R^{n\times d}$ and $W$ is a $d$-dimensional Brownian motion independent of $x_{0}$.
	   
	  \begin{assumption} \label{ass:brigo:thm} \em{ \quad \\ 
	  \vspace{-4mm}\begin{enumerate}
	  \item The SDE~\eqref{eq:approx_SDE} satisfies Assumption \ref{ass:sde_main_ass}.
	   \item The initial condition $x_{0}$ has a density $p_{0}$ that is absolutely continuous with respect to the Lebesgue measure and has finite moments of any order.
	   \item The unique solution $X_{t}$ to \eqref{eq:approx_SDE} admits a density $p(x,t)$ that is absolutely continuous with respect to the Lebesgue measure and that satisfies the Kolmogorov forward equation.
	   \end{enumerate} }  
	  \end{assumption}

	  \begin{problem} \label{problem1}
	   Let $\text{EM}(c^{(1)},\hdots,c^{(k)})$ be a mixture of exponential families, let $p_0$ be a density contained in $\text{EM}(c^{(1)},\hdots,c^{(k)})$, let $\sigma$ be a diffusion term and let $a(\cdot):=\sigma(\cdot)\sigma(\cdot)\transp$. Let $\mathcal{U}(x_0,\sigma)$ denote the set of all drifts $u$ such that $x_0,u,\sigma$ and its related SDE \eqref{eq:approx_SDE} satisfy Assumption~\ref{ass:brigo:thm}. Assume $\mathcal{U}(x_0,\sigma)$ to be non-empty. Then given a curve $t\mapsto p(\cdot,\Theta^{(1)}_t,\hdots,\Theta^{(k)}_t)$ in $\text{EM}(c^{(1)},\hdots,c^{(k)})$, find a drift in $\mathcal{U}(x_0,\sigma)$ whose related SDE has a solution with marginal density $p(\cdot,\Theta^{(1)}_t,\hdots,\Theta^{(k)}_t)$.  
	  \end{problem}
	 A solution to Problem~\ref{problem1} is given by the following theorem.
	 \begin{theorem} \label{thm:brigo_mixture_multi}
	  Given the assumptions and notation of Problem \ref{problem1}. Consider the SDE \eqref{eq:approx_SDE} with drift term
	  \begin{equation*}
	   \begin{aligned}
	    u_i(x,t) &= \frac{1}{2}\sum_{j=1}^n \frac{\partial}{\partial x_j}a_{ij}(x) + \frac{1}{2}\sum_{j=1}^n a_{ij}(x) \frac{\frac{\partial}{\partial x_j}p(x,\Theta^{(1)}_t,\hdots,\Theta^{(k)}_t)}{p(x,\Theta^{(1)}_t,\hdots,\Theta^{(k)}_t)} \\
	    &\hspace{8mm}- \frac{1}{p(x,\Theta^{(1)}_t,\hdots,\Theta^{(k)}_t)}\sum_{\ell=1}^k \nu_\ell p_\ell(x,\Theta^{(\ell)}_{t})\inprod{\dot{\Theta}^{(\ell)}_{t}}{\mathcal{I}_i^{(\ell)}(x)},
	   \end{aligned}
	  \end{equation*}
	    for $i=1,\hdots,n$, where
	    \begin{equation} \label{e:thm:integral}
	     \mathcal{I}_i^{(\ell)}(x) := \int_{-\infty}^{x_i} \varphi_i^{(\ell)}((x_{-i},\xi_i),\Theta^{(\ell)}_{t})\exp\left( \inprod{\Theta^{(\ell)}_{t}}{c^{(\ell)}(x_{-i},\xi_i)-c^{(\ell)}(x)} \right)\drv \xi_i ,
	    \end{equation}
	    $(x_{i-},\xi_i):=(x_1,\hdots,x_{i-1},\xi_i,x_{i+1},\hdots,x_n)\transp$ and the functions $\varphi_i^{(\ell)}:\R^n\times \Hi \to \Hi$ for all $\ell=1,\hdots,k$ satisfy
	   \begin{equation} \label{eq:cond_varphi_thm_multi_brigo}
	   \begin{aligned}
	    \sum_{i=1}^n \left. \inprod{\dot{\Theta}^{(\ell)}_{t}}{\varphi_i^{(\ell)}\left((x_{-i},\xi_i),\Theta^{(\ell)}_{t}\right)}\right|_{\xi_i=x_i}  
	    %&= \sum_{i=1}^n \inprod{\dot{\Theta}^{(\ell)}_{t}}{\varphi_i^{(\ell)}(x,\Theta^{(\ell)}_{t})}\\  
	    &= \inprod{\dot{\Theta}^{(\ell)}_{t}}{c^{(\ell)}(x)-\nabla_{\Theta} \psi_\ell(\Theta^{(\ell)}_{t})}.
	    \end{aligned}
	   \end{equation}
	    If $u\in \mathcal{U}(x_0,\sigma)$, then the SDE \eqref{eq:approx_SDE} solves Problem \ref{problem1}, i.e., $X_{t}$ has a density
	   \begin{equation*}
	    p_{X_{t}}(x)=\sum_{\ell=1}^k \nu_\ell \exp\left(\inprod{\Theta^{(\ell)}_{t}}{c^{(\ell)}(x)}-\psi_\ell(\Theta^{(\ell)}_{t})\right) , \quad \text{for all } t \leq T.
	   \end{equation*}
	 \end{theorem}
	 The proof is provided in Appendix \ref{app:proof}.

	 \begin{remark} \label{rmk:Brigo:extension:thm}
	\rm{  \begin{enumerate}
	   \item For the non-mixture and one-dimensional case $(k=n=1)$, the result is known \cite{ref:Brigo-00} and coincides with Theorem \ref{thm:brigo_mixture_multi}. Furthermore, it can be seen by the proof in \cite{ref:Brigo-00} and by invoking the existence and uniqueness theorem for ODEs, that the drift function $u$ is uniquely determined.
	   \item For the multi-dimensional case $(n>1)$, the drift function is not unique anymore, as there exist multiple choices for $\varphi_i^{(\ell)}$\footnote{For example, $\varphi_i^{(\ell)}(x,\Theta^{(\ell)}_{t}):=\delta_{ij}(c^{(\ell)}(x)-\nabla_{\Theta} \psi_\ell(\Theta^{(\ell)}_{t}))$ for all $j\in \{1,\hdots,n\}$ are feasible choices for $\varphi_i^{(\ell)}$, as they satisfy \eqref{eq:cond_varphi_thm_multi_brigo}.}. This gives rise to a natural question, if there exist a particular choice of $\varphi_i^{(\ell)}$ such that the integral terms $\mathcal{I}_i^{(\ell)}$ in \eqref{e:thm:integral} admit closed-form expressions. In Section~\ref{sec:mixture_normal} (Proposition~\ref{prop:drift:fct}), we derive such functions $\varphi_i^{(\ell)}$ for the mixture of multivariate normal densities. 
	   \item In a non-mixture setting $(k=1)$, the drift function simplifies to
	    \begin{equation*} 
	   \begin{aligned}
	   u_i(x,t) &=  \frac{1}{2}\sum_{j=1}^n \frac{\partial}{\partial x_j} a_{i,j}(x) + \frac{1}{2}\sum_{j=1}^n a_{i,j}(x) \inprod{\Theta_{t}}{\frac{\partial c(x)}{\partial x_j}} \\
	   &\hspace{8mm}-\inprod{\dot{\Theta}_{t}}{\int_{-\infty}^{x_i} \varphi_i((x_{-i},\xi_i),\Theta_{t})\exp\left[ \inprod{\Theta_{t}}{c(x_{-i},\xi_i)-c(x)} \right]\drv \xi_i},
	  \end{aligned}
	  \end{equation*}
	  where the functions $\varphi_i$ have to satisfy \eqref{eq:cond_varphi_thm_multi_brigo}.
	%  \item We use the abbreviations $\Theta:= \Theta(t)$ and $\dot{\Theta}:=\frac{\drv \Theta (t)}{\drv t}$ in oder to keep the notation simple.
	  \end{enumerate} }  
	 \end{remark}
	 As remarked, the drift term proposed in Theorem~\ref{thm:brigo_mixture_multi} consists of the integral terms \eqref{e:thm:integral}, that depend on the particular exponential families considered. In the following, we restrict ourselves to the mixture of multivariate normal densities and show that these integral terms, and hence the drift function, admit a closed-form expression.
	 
	 %****************************************************************************************************************************************************************
	 \subsection{Mixture of multivariate normal densities} \label{sec:mixture_normal}
	 Consider the family of multivariate Gaussian distributions with mean $m\in\R^n$ and covariance matrix $S\in\sym(n,\R)$, that can be expressed in terms of Definition~\ref{def:muli_exp_fam} as follows.
	 Let the Hilbert space $\Hi=\R^n\times\R^{n\times n}$ be endowed with the inner product $\inprod{(a,A)}{(b,B)} =  a\transp b + \tr(A\transp B)$ and define
	  \begin{equation} \label{eq:multivar:normal:parameters}
	   \begin{aligned}
	    &\Theta =(\eta,\theta):=\left(S^{-1}m,-\frac{1}{2}S^{-1}\right) \in \Hi, \quad c: \R^n \to \Hi,\quad c(x) = (x,xx\transp) \\
	    &\psi: \Hi\to \R, \quad \psi(\Theta)=-\frac{1}{4}\tr(\eta \eta\transp \theta^{-1}) + \frac{1}{2}\log \det\left(- \frac{1}{2}\theta^{-1} \right) + \frac{n}{2}\log(2\pi).
	   \end{aligned}
	  \end{equation}
	  A direct computation, using $\tr(\eta \eta\transp \theta^{-1}) = \eta\transp \theta^{-1}\eta$, leads to
	  \begin{align*}
	    p(x,\Theta) &= \exp\left( \inprod{c(x)}{\Theta}-\psi(\Theta) \right) 
		%	&= \exp \left( x\transp S^{-1}m + \tr(-xx\transp \frac{1}{2}S^{-1}) \right) \exp\left( -\frac{1}{2}m\transp S^{-1}m \right)\\ 
		%	&\hspace{12mm}\exp \left( -\frac{1}{2}\log \det\left( S \right) - \frac{n}{2}\log(2\pi) \right) \\
		%	&=  \exp \left( x\transp S^{-1}m -\frac{1}{2}x\transp S^{-1}x  -\frac{1}{2}m\transp S^{-1}m \right) \frac{1}{(\det S)^{\frac{1}{2}}(2\pi)^{\frac{n}{2}}} \\
			= \frac{1}{(2\pi)^{\frac{n}{2}}(\det S)^{\frac{1}{2}}}\exp\!\left(\! -\frac{1}{2}(x-m)\transp S^{-1} (x-m)\! \right).
	  \end{align*}
	  We point out again that for the proposed variational method, it is favourable if the approximating SDE \eqref{eq:approx_SDE} has a drift function that admits a closed-form expression. Furthermore, since the drift function is not unique (cf. Remark~\ref{rmk:Brigo:extension:thm}), among all feasible solutions characterized by the $\varphi_i^{(\ell)}$ functions, we want to find one that can be computed analytically. The latter turns out to be a difficult task and depending heavily on the specific exponential familiy chosen. From now on, we consider the exponential family of the multivariate normal probability densities that is given by \eqref{eq:multivar:normal:parameters}. In this setting, it is possible to find functions $\varphi_i^{(\ell)}$ such that the integral terms \eqref{e:thm:integral}, and therefore the drift function, can be computed in closed form.
	   
	  \begin{proposition} \label{prop:drift:fct}
	  For the mixture of multivariate normal densities, one possible choice for the drift function proposed by Theorem \ref{thm:brigo_mixture_multi} is
	  \begin{align*} 
	    u(x,t) &= \frac{1}{2}\diverg{a(x)} \!+ \! \frac{\sum_{\ell=1}^k \nu_\ell p_\ell(x,\Theta^{(\ell)}_{t}) }{p(x,\Theta^{(1)}_t,\hdots,\Theta^{(k)}_t)}\bigg( \frac{1}{4}{\theta^{(\ell)}_{t}}^{-1}\dot{\theta}^{(\ell)}_{t} {\theta^{(\ell)}_{t}}^{-1}\eta^{(\ell)}_{t} - \frac{1}{2}{\theta^{(\ell)}_{t}}^{-1}\dot{\eta}_{t}^{(\ell)} {} \\ 
&\qquad - \frac{1}{2}{\theta^{(\ell)}_{t}}^{-1}\dot{\theta}^{(\ell)}_{t} x + a(x) \left( \frac{1}{2}\eta^{(\ell)}_{t} + \theta^{(\ell)}_{t} x \right) \!\!\! \bigg). 
	  \end{align*}  
	  \end{proposition}
	  The proof is provided in Appendix~\ref{app:proof_drift_fct}.

	  \begin{remark} \rm{
 For the non-mixture setting the drift term simplifies to
	    \begin{equation*}
	      u(x,t) = \frac{1}{2}\diverg{a(x)} + \frac{1}{4}\theta_{t}^{-1}\dot{\theta}_{t} \theta_{t}^{-1}\eta_{t} - \frac{1}{2}\theta_{t}^{-1}\dot{\eta}_{t} - \frac{1}{2}\theta_{t}^{-1}\dot{\theta}_{t} x + a(x) \left( \frac{1}{2}\eta_{t} + \theta_{t} x \right), 
	    \end{equation*}
	    that in the special case of a constant diffusion term is a linear function, as one would expect. }   
	  \end{remark}  
	   We introduce the following ansatz for the drift function
	   \begin{equation} \label{eq:ansatz_multi_mix}
	   \begin{aligned}
	   u(x,t) \!&=\! \frac{1}{2}\diverg{a(x)}\! +\! \frac{\sum_{\ell=1}^k\! \nu_\ell p_\ell(x,\Theta^{(\ell)}_{t})\!  \left( A^{(\ell)}_{t} \!+\! B^{(\ell)}_{t} x \! +\! a(x)\! \left(C^{(\ell)}_{t}\! +\! D^{(\ell)}_{t} x\! \right)\! \right)}{p(x,\Theta^{(1)}_t,\hdots,\Theta^{(k)}_t)},
	   \end{aligned}
	  \end{equation}
	   where $B^{(\ell)}_{t},D^{(\ell)}_{t}\in\R^{n\times n}$ and $A^{(\ell)}_{t},C^{(\ell)}_{t}\in\R^n$ for all $\ell =1,\hdots k$. The coefficients $A^{(\ell)}_{t}$, $B^{(\ell)}_{t}$, $C^{(\ell)}_{t}$ and $D^{(\ell)}_{t}$ cannot be chosen arbitrarily. They are coupled according to Proposition~\ref{prop:drift:fct}.  By comparing the coefficients of Proposition~\ref{prop:drift:fct} and \eqref{eq:ansatz_multi_mix} one gets
	  \begin{equation*} 
	   A^{(\ell)}_{t} \!= \! \frac{1}{4}{\theta^{(\ell)}_{t}}^{-1}\!\dot{\theta}^{(\ell)}_{t} {\theta^{(\ell)}_{t}}^{-1}\!\eta^{(\ell)}_{t} - \frac{1}{2}{\theta^{(\ell)}_{t}}^{-1}\dot{\eta}_{t}^{(\ell)} {}, \
	   B^{(\ell)}_{t} \!= \! -\frac{1}{2}{\theta^{(\ell)}_{t}}^{-1}\dot{\theta}^{(\ell)}_{t}, \
	   C^{(\ell)}_{t} \!= \! \frac{1}{2}\eta^{(\ell)}_{t}, \
	   D^{(\ell)}_{t} \!= \! \theta^{(\ell)}_{t}\!.
	  \end{equation*}
	   Hence, one directly sees that the four parameters $A^{(\ell)}_{t}$, $B^{(\ell)}_{t}$, $C^{(\ell)}_{t}$ and $D^{(\ell)}_{t}$ for all $\ell=1,\hdots,k$ are coupled via the two ODEs
	  \begin{equation}
	   \frac{\drv C^{(\ell)}_{t}}{\drv t}= -D^{(\ell)}_{t}A^{(\ell)}_{t} - {B^{(\ell)}_{t}}\transp C^{(\ell)}_{t},  \qquad  \frac{\drv D^{(\ell)}_{t}}{\drv t}= -2D^{(\ell)}_{t}B^{(\ell)}_{t}. \label{eq:ode_for_C}\\
	  \end{equation}
	  Note that the parametrization introduced in \eqref{eq:ansatz_multi_mix} provides relatively simple expression for the mean and variance of the variational approximation derived in the next section (Section~\ref{sec:eq_mean_var}). In the authors' opinion this parametrization therefore helps to keep the notation simple.

	 \subsection{Equations for mean and variance} \label{sec:eq_mean_var}
	 Theorem~\ref{thm:brigo_mixture_multi} provides an explicit formula for the drift term in the approximating SDE~\eqref{eq:approx_SDE}, that simplifies to \eqref{eq:ansatz_multi_mix} in the case of multi-normal marginal densities. Therefore, the mean and variance of the approximating SDE~\eqref{eq:approx_SDE} are characterized via the following two ODEs.

	  \begin{theorem} \label{thm:mean_variance}
	   Consider the SDE~\eqref{eq:approx_SDE} with drift term $u$ given by \eqref{eq:ansatz_multi_mix}, such that the solution $X_{t}$ has a  marginal density $p(x,\Theta^{(1)}_t,\hdots,\Theta^{(k)}_t) \in \text{EM}(c_1,\hdots,c_k)$ that is an arbitrary convex combination of densities $p_\ell(x,\Theta^{(\ell)}_{t})\in \text{EM}(c_\ell)$ for $\ell =1,\hdots, k$.
Let $m^{(\ell)}_{t}$ and $S^{(\ell)}_{t}$ denote the mean and variance of $X_{t}$ with respect to $p_\ell(x,\Theta^{(\ell)}_{t})$. Then,
	   \begin{equation} \label{e:thm:mean}
	    \begin{aligned}
	     \frac{\drv m^{(\ell)}_{t}}{\drv t} \!\!&=\! \frac{1}{2}\CExpec{\diverg{a(X)}\!\!}{p_\ell} \!+\! A^{(\ell)}_{t}\! \!+\! B^{(\ell)}_{t} m^{(\ell)}_{t} \!\!+\! \CExpec{a(X)}{p_\ell}\!C^{(\ell)}_{t}\!\! +\! \CExpec{a(X)D^{(\ell)}_{t} X}{p_\ell}
	    \end{aligned}
	   \end{equation}
	   and
	   \begin{align} 
	     \frac{\drv S^{(\ell)}_{t}}{\drv t} &= \frac{1}{2}\CExpec{X\diverg{a(X)}\transp}{p_\ell}+ \frac{1}{2}\CExpec{\diverg{a(X)} X\transp}{p_\ell} - \frac{1}{2}m^{(\ell)}_{t} \CExpec{\diverg{a(X)}}{p_\ell}\transp \nonumber\\ 
	     &\hspace{8mm}- \frac{1}{2}\CExpec{\diverg{a(X)}}{p_\ell}{m^{(\ell)}_{t}}\transp + \CExpec{a(X)}{p_\ell} + S^{(\ell)}_{t} {B^{(\ell)}_{t}}\transp + B^{(\ell)}_{t} S^{(\ell)}_{t} \nonumber \\
	     &\hspace{8mm}+ \CExpec{X{C^{(\ell)}_{t}}\transp a(X)}{p_\ell}+ \CExpec{a(X)C^{(\ell)}_{t} X\transp}{p_\ell} - m^{(\ell)}_{t} {C^{(\ell)}_{t}}\transp \CExpec{a(X)}{p_\ell} \label{e:thm:var}\\
	     &\hspace{8mm}- \CExpec{a(X)}{p_\ell}C^{(\ell)}_{t} {m^{(\ell)}_{t}}\transp+\CExpec{XX\transp D^{(\ell)}_{t} a(X)}{p_\ell} + \CExpec{a(X) D^{(\ell)}_{t} XX\transp}{p_\ell} \nonumber \\ 
	     &\hspace{8mm}- m^{(\ell)}_{t} \CExpec{X\transp D^{(\ell)}_{t} a(X)}{p_\ell} -\CExpec{a(X)D^{(\ell)}_{t} X}{p_\ell}{m^{(\ell)}_{t}}\transp.  \nonumber
	   \end{align}  
	  \end{theorem}
	  The proof is provided in Appendix~\ref{app:proof_thm_mS}.
	  Note that given $m^{(\ell)}_{t}$ and $S^{(\ell)}_{t}$ the mean and variance of $X_{t}$ can be expressed as $m_{t} = \sum_{\ell=1}^k \nu_\ell m^{(\ell)}_{t}$ and $S_{t} = \sum_{\ell=1}^k \nu_\ell S^{(\ell)}_{t} + \sum_{\ell=1}^k\nu_\ell m^{(\ell)}_{t} {m^{(\ell)}_{t}} \transp - ( \sum_{\ell =1}^k \nu_\ell  m^{(\ell)}_{t} )( \sum_{\ell =1}^k \nu_\ell  m^{(\ell)}_{t} )\transp$, respectively.	 
	 
	  \begin{remark} \rm{
	  If the coefficients $\nu_{i}$  in the convex combination of the marginal density \\ $p(x,\Theta^{(1)}_t,\hdots,\Theta^{(k)}_t)$ in Theorem~\ref{thm:mean_variance} are fixed a priori, the ODEs \eqref{e:thm:mean} and \eqref{e:thm:var} are only sufficient for describing $m^{(\ell)}_{t}$ and $S^{(\ell)}_{t}$. Oftentimes, however, one is interested in choosing those coefficients a posteriori, for example by solving an auxiliary optimization problem. In such a setting the ODEs given by Theorem~\ref{thm:mean_variance} are necessary and sufficient.}  
	  \end{remark}

We have studied how to reformulate the constraints \eqref{eq:problem:i} and \eqref{eq:problem:ii} of Problem~\ref{prob:optimization:problem} by deriving an expression for the drift term to the approximating SDE \eqref{eq:approx_SDE}. In the case that the marginals in \eqref{eq:problem:ii} are restricted to a mixture of multivariate normal densities this reformulation reduces to the ODEs \eqref{eq:ode_for_C}, \eqref{e:thm:mean} and \eqref{e:thm:var}.

\subsection{Example: Geometric Brownian Motion} \label{sec:GBM:example:2}
We continue the geometric Brownian motion example started in Section~\ref{sec:GBM:example:1}. The goal is to approximate the smoothing density by a normal density. Therefore, according to Proposition~\ref{prop:drift:fct}, the drift function for the approximating SDE~\eqref{eq:approx_SDE} has to be chosen as
	\begin{equation} \label{eq:GBM:drift}
	u(x,t) = A_{t} + (\lambda^2 + B_{t})x + \lambda^2 x^{2}( C_{t} +D_{t}x),
	\end{equation}
where the coefficients $A_t, B_t, C_t, D_t$ are coupled via the two ODEs \eqref{eq:ode_for_C}. This choice of drift function leads to ODEs for the mean and the variance of the posterior process, according to Theorem~\ref{thm:mean_variance}
\begin{align}
\frac{\drv m_t}{\drv t} &= \lambda^2 m_t + A_t + B_t m_t +  \lambda^2C_t(m_t^2 + S_t) + \lambda^2 D_t(m_t^3 + 3m_t S_t) \label{ODE:mean:GBM}\\
\frac{\drv S_t}{\drv t} &= \lambda^2(m_t^2 + 3 S_t) + 2B_t S_t + 4\lambda^2 C_t m_t S_t + 6\lambda^2 D_t(m_t^2 S_t + S_t^2). \label{ODE:var:GBM}
\end{align}

%%%%%%%%%%%%%%%%%%%%%%%%%%%%%%%%%%%%%%%%%%%%%%%%%%%%%%%%%%%%%%%%%%%%%%%%%%%%%%%%%%%%%%%%%%%%%%%%%%%
%%%%%%%%%%%%%%%%%%%%%%%%%%%%%%%%%%%%%%%%%%%%%%%%%%%%%%%%%%%%%%%%%%%%%%%%%%%%%%%%%%%%%%%%%%%%%%%%%%%%%%%%%%%%
 \section{Optimal control problem formulation} \label{sec:optimal:control:problem:formulation}
	 In this section, we show that the optimization problem~\ref{prob:optimization:problem}, using the results derived from Theorem~\ref{thm:brigo_mixture_multi}, can be reformulated as a standard optimal control problem (OCP), which conceptually is similar to \cite{ref:Mitter-03}\footnote{Note that \cite{ref:Mitter-03} addresses a related problem, whose main difference, when compared to the presented method, is that the variational characterization considered there is exact.}. Therefore, the presented variational approximation method to the path estimation problem for SDEs can be expressed as an OCP and as such leads to a standard formulation of necessary global optimality conditions in terms of Pontryagin's maximum principle.
	Consider the vector spaces $\hat{\mathcal{V}}:=\R^n\times\R^{n\times n}$, $\hat{\mathcal{Z}}:=\R^n \times \sym(n,\R)\times \R^n \times \sym(n,\R)$
	and define the trajectories
	\begin{align*}
	[0,T]\ni t \mapsto v^{(\ell)}(t) &:= (A^{(\ell)}_{t},B^{(\ell)}_{t})\in\hat{\mathcal{V}} \\
	[0,T]\ni t \mapsto z^{(\ell)}(t) &:= (m^{(\ell)}_{t}, S^{(\ell)}_{t}, C^{(\ell)}_{t},D^{(\ell)}_{t})\in\hat{\mathcal{Z}},
	\end{align*}
     for $\ell=1,\hdots,k$. We introduce the state variable $z(t):=\left(z^{(1)}(t),\hdots,z^{(k)}(t)\right)\in\prod_{\ell=1}^k \hat{\mathcal{Z}}=:\mathcal{Z}$ and the control variable $v(t):=\left(v^{(1)}(t),\hdots,v^{(k)}(t)\right)\in\prod_{\ell=1}^k \hat{\mathcal{V}}=:\mathcal{V}$ for $t\in[0,T]$.   
	As a first step, in view of the cost functional \eqref{eq:obj_fct} of Problem~\ref{prob:optimization:problem}, the so-called Lagrangian 
	\begin{equation} \label{eq:Lagrangian}
	\begin{aligned}
	&\mathbb{E}_{Q}\left[\frac{1}{2}\norm{u(X_{t},t)-f(X_{t})}_{a(X_{t})}^2  \right.\\ 
	&\hspace{10mm} \left. + \ y_{t}\left( u(X_{t},t)\transp \nabla h(X_{t}) + \frac{1}{2}\sigma(X_{t})\transp \nabla^{2} h(X_{t})\sigma(X_{t}) \right) +  \frac{1}{2} \| h(X_{t}) \|^{2} \right]
	\end{aligned}
	\end{equation}
	is expressed as a function of only $z(t),v(t)$ and $t$. This step, while being exact in some cases, may require an approximation. In the case that the marginals of $Q$ are mixtures of normal densities, the expectation of any polynomial in $X_{t}$ can be expressed as a function of its mean and variance. If the diffusion term $\sigma$ is a polynomial, and no mixture is considered $(k=1)$, the drift function $u$, according to \eqref{eq:ansatz_multi_mix}, is a polynomial. We refer to Section~\ref{sec:examples} to see how the Lagrangian can be derived for two concrete examples.
Consider a Lagrangian 	
	\begin{align*}
	 L : [0,T]\times\mathcal{Z} \times \mathcal{V} \to  \R, \quad L(t,z(t),v(t)) \approx  \eqref{eq:Lagrangian},
	\end{align*}
	where $\approx$ indicates that in order to express the term \eqref{eq:Lagrangian} by the state and control variables only, an approximation might be needed, as explained above.
	Similarly to the Lagrangian, in view of the cost functional \eqref{eq:obj_fct}, we introduce a terminal cost $F:\mathcal{Z}\to \R$ by
	\begin{align*}
	 F(z(T)) \approx - y_{T}\CExpec{h(X_{T})}{Q}. 
	\end{align*}			
	Under the assumption that the drift term $\sigma$ is a polynomial, the ODEs derived in the previous section can be expressed in standard form. We define the function $H : \mathcal{Z} \times  \mathcal{V}\to \mathcal{Z}$ by
	\begin{align*}
	 H(z(t),v(t)) \! =& \! \left(\!H_1^{(1)}\!(z(t),v(t)),\hdots,H_4^{(1)}\!(z(t),v(t)), \hdots, H_1^{(k)}\!(z(t),v(t)),\hdots,H_4^{(k)}\!(z(t),v(t))\!\right)\!,
	\end{align*}
	where
	\begin{align*}
	 \frac{ \drv m^{(\ell)}_{t}}{\drv t} &= \frac{\drv z_1^{(\ell)}}{\drv t}(t) = H_1^{(\ell)}\left(z(t),v(t) \right), \quad && \frac{ \drv C^{(\ell)}_{t}}{\drv t} = \frac{\drv z_3^{(\ell)}}{\drv t}(t) = H_3^{(\ell)}\left(z(t),v(t)\right), \\
	 \frac{ \drv S^{(\ell)}_{t}}{\drv t} &= \frac{\drv z_2^{(\ell)}}{\drv t}(t) = H_2^{(\ell)}\left(z(t),v(t)\right), \quad && \frac{ \drv D^{(\ell)}_{t}}{\drv t} = \frac{\drv z_4^{(\ell)}}{\drv t}(t) = H_4^{(\ell)}\left(z(t),v(t)\right), 
	\end{align*}
	for $\ell=1,\hdots,k$ are given by \eqref{e:thm:mean}, \eqref{e:thm:var} and \eqref{eq:ode_for_C}. Thus, we have shown so far in this article that Problem~\ref{prob:optimization:problem} can be reformulated as the following optimal control problem
	 \begin{equation} \label{eq:optimal:control:problem} \left\{
	 \setlength\arraycolsep{3pt}
	 \begin{array}{cllll}
	 \underset{v\in\mathcal{M}([0,T],\mathcal{V})}{\text{minimize}}& & J(v) &=  &\int_{0}^{T} L(t,z(t),v(t)) \drv t + F(z(T)) \\
	 \text{subject to} \hspace{10mm}& &\dot{z}(t) &= &H(z(t),v(t)),  \quad t\in[0,T] \text{ a.e.}\\
	 									 &&  z(0) &= &z_{0},
	\end{array} \right.
	\end{equation}
where $\mathcal{M}([0,T],\mathcal{V})$ denotes the space of measurable functions from $[0,T]$ to $\mathcal{V}$. It remains to discuss how to find the initial condition $z_0$ in the OCP~\eqref{eq:optimal:control:problem}. A straightforward, however, clearly not efficient, method for that is solving the Pardoux equation~\eqref{eq:backward:SPDE:v}, which according to \eqref{eq:smoothing:density:char} provides the smoothing density at initial time as $\Ps(x,0)=\frac{p_0(x)w(x,0)}{\int_{\R^{n}}p_0(x)w(x,0)\drv x}$, from where $z_0$ can be derived.

	\subsection{Maximum principle} \label{sec:Pontryagin}
We derive necessary conditions for global optimality of the optimization problem \eqref{eq:optimal:control:problem} that are provided by the Pontryagin maximum principle (PMP). Since the control set $\mathcal{V}$ is unbounded, we need an extended setting of the standard PMP, see \cite[Section~22.4]{ref:Clarke-13} for a comprehensive survey.
It requires some further assumptions.
 
\begin{assumption}\label{ass:PMP} \em{
Let the process $(z^{\star}(t),v^{\star}(t))_{t\in[0,T]}$ be a local minimizer for the OCP \eqref{eq:optimal:control:problem}, that satisfies
\begin{enumerate}[(i)]
\item The function $F$ is continuously differentiable; \label{ass:PMP:i}
\item The functions $H$ and $L$ are continuous and admit derivatives relative to $z$ which are themselves continuous in all variables $(t,z,v)$; \label{ass:PMP:ii}
\item \label{ass:PMP:iii} There exist $\varepsilon >0$, a constant $c$, and a summable function $d$ such that for almost every $t\in[0,T]$, we have
\begin{equation*}
|z-z^{\star}(t)|\leq \varepsilon \Rightarrow |\nabla_{z} (H,L)(t,z,v^{\star}(t))| \leq c |(H,L)(t,z,v^{\star}(t))| + d(t).
\end{equation*}
\end{enumerate} }  
\end{assumption}
\noindent Note that Assumption~\ref{ass:PMP}\eqref{ass:PMP:iii} is implied if 
\begin{equation*}
|\nabla_{z} H(t,z,v)| + |\nabla_{z} L(t,z,v)| \leq c \left( |H(t,z,v)| + |L(t,z,v)| \right) + d(t)
\end{equation*}
holds for all $v\in\mathcal{V}$ when $z$ is restricted to a bounded set, which is satisfied by many systems. Moreover, the condition automatically holds if $v^{\star}$ happens to be bounded.
 
\begin{lemma}[PMP {\cite[Theorem~22.2]{ref:Clarke-13}}] \label{lem:maximum:principle}
Given Assumption~\ref{ass:PMP}, let the process $(z^{\star}(t),v^{\star}(t))_{t\in[0,T]}$ be a local minimizer for the problem \eqref{eq:optimal:control:problem}. Then there exists an absolutely continuous function $p:[0,T]\to   \mathcal{Z}$ satisfying
\begin{enumerate}
\item the adjoint equation $\dot{p}(t)=-\nabla_z\inprod{p(t)}{H(z^{\star}(t),v^{\star}(t))} - \nabla_z L(t,z^{\star}(t),v^{\star}(t))$ for almost every $t\in[0,T];$ 
\item the transversality condition $p(T) = \nabla_{z}F(z(T));$
\item \label{item:maximum:condition} the maximum condition \\$ \inprod{p(t)}{H(z^{\star}(t),v^{\star}(t))} +  L(t,z^{\star}(t),v^{\star}(t)) = \underset{v\in\mathcal{V}}{\inf} \inprod{p(t)}{H(z^{\star}(t),v)} +  L(t,z^{\star}(t),v)$ for almost every $t\in[0,T]$.
\end{enumerate}  
\end{lemma}
\begin{remark} \label{rem:PMP:cts} \rm{
\begin{enumerate}
\item Given that an optimal process $(z^{\star},v^{\star})$ exists\footnote{Existence of an optimal process can be assured by standard existence results, see for example \cite[Theorem~23.11]{ref:Clarke-13}.}, the maximum condition \ref{item:maximum:condition} can be used to derive a feedback law 
\begin{equation*}
v^{\star}(t)\in \arg \min_{v\in\mathcal{V}} \inprod{p(t)}{H(z^{\star}(t),v)} +  L(t,z^{\star}(t),v).
\end{equation*}
\item Lemma~\ref{lem:maximum:principle}, basically leads to a boundary value problem with initial conditions for the states and terminal conditions for the adjoint states, that provides necessary conditions for global optimality of Problem~\ref{prob:optimization:problem}.
\end{enumerate} }  
\end{remark}

We summarize the described method to approximate the smoothing density introduced so far. It basically consists of the following three steps, that provide a solution to Problem~\ref{prob:optimization:problem}:
\begin{enumerate}
\item[\textbf{Step 1}] Fix a mixture of exponential families of probability densities, e.g., the mixture of multivariate normal densities. Theorem~\ref{thm:brigo_mixture_multi}, that simplifies to Proposition~\ref{prop:drift:fct} for the multivariate normal densities, characterizes the approximate posterior SDE \eqref{eq:approx_SDE} whose solution admits marginal densities evolving in the chosen mixture of exponential families. \vspace{1mm}
\item[\textbf{Step 2}] Given the approximate posterior SDE \eqref{eq:approx_SDE}, we derive an optimal control formulation of Problem~\ref{prob:optimization:problem}. For the mixture of multivariate normal densities, this derivation is presented in Sections~\ref{sec:prescribed_marginal_law} and \ref{sec:optimal:control:problem:formulation} and finally leads to the OCP~\eqref{eq:optimal:control:problem}.\vspace{1mm}
\item[\textbf{Step 3}] Necessary conditions for optimality of the OCP~\eqref{eq:optimal:control:problem}, and hence for Problem~\ref{prob:optimization:problem}, can be derived from Pontryagin's maximum principle and result in a structured boundary value problem. 
%One possible method, for estimating the smoothing density at terminal time is to solve the Zakai equation~\eqref{eq:Zakai}.
\end{enumerate}

\begin{remark}
{\rm	It is important to note that the presented method chooses the best approximating SDE in a desired class using an objective distance measure between the corresponding probability distributions. One crucial advantage of this approach is that this distance could be quantified and numerically calculated (note that the first term in Lemma~\ref{Lemma:KL:cost} can be directly computed and the remaining two terms form the objective function of the optimal control problem considered), and hence the user gets an excellent estimate on the necessary approximating error. For instance, Figure~\ref{fig:GBM:KL} and Figure~\ref{fig:CIR:KL} demonstrated the accuracy of corresponding approximating SDEs by plotting  the relative entropies between the approximate models and the exact ones for the two examples considered in the paper.}
\end{remark}

\subsection{Computational complexity} \label{sec:computational:complexity}
If the initial condition to the OCP~\eqref{eq:optimal:control:problem} is known, the PMP, Lemma~\ref{lem:maximum:principle}, reduces to a boundary value problem, that can usually be solved numerically more efficiently than (S)PDEs by using numerical methods specifically tailored to these problems, such as the shooting method, see \cite{ref:Stoer-02}. Therefore, the major computational difficulty of the presented variational approach lies in estimating the initial condition to the OCP~\eqref{eq:optimal:control:problem}, for example via estimating the smoothing density at initial time. A straightforward, however clearly not efficient, method for that is solving the Pardoux equation~\eqref{eq:backward:SPDE:v}, as explained in Section~\ref{sec:model:setup}, which we used in the numerical examples in Section~\ref{sec:examples}. As such, whereas the standard PDE approach for computing a smoothing density requires solving a Zakai equation and the Pardoux equation~\eqref{eq:backward:SPDE:v}, the presented variational approach relies on only a Pardoux equation and the mentioned boundary value problem. This can usually be seen as a reduction in terms of computational effort required and is demonstrated by two numerical examples in Section~\ref{sec:examples}, Table~\ref{tab:runtime}.
Moreover, for future work, we aim to study the derivation of an estimator for the marginal smoothing density at terminal time without solving a Pardoux equation, that would then allow us to apply the proposed variational approximation method to high-dimensional problems, see Section~\ref{sec:conclusion} for more details.
Another idea to circumvent the estimation of this mentioned terminal condition is to use an alternative approach to the PMP, for characterizing a solution to the OCP~\eqref{eq:optimal:control:problem} that is briefly described in the following remark.
 
\begin{remark}[Semidefinite programming] \label{rem:SDP}\rm{
Solutions to the OCP~\eqref{eq:optimal:control:problem} can be characterized via the so-called weak formulation which consists of an infinite-dimensional linear program, see \cite[Chapter~10]{ref:Las-10} for details. Therefore, numerical approximation schemes to such infinite-dimensional linear programs, that have been studied in the literature, can be employed to solve Problem~\ref{prob:optimization:problem}. This approach seems particularly promising when the data of the OCP (dynamics and costs) are described by polynomials, as then the seminal Lasserre hierarchy based on solving a sequence of semidefinite programs, is applicable \cite{ref:Lasserre-01,ref:Las-10}.
}\end{remark}

\subsection{Example: Geometric Brownian Motion} \label{sec:GBM:example:3}
We continue the geometric Brownian motion example started in Sections~\ref{sec:GBM:example:1} and \ref{sec:GBM:example:2} and formulate the corresponding optimal control problem \eqref{eq:optimal:control:problem}. Recall that the state variable is defined as $z(t):=(m_t, S_t, C_t, D_t)$ and the control variable as $v(t):=(A_t, B_t)$. The ODEs for the state variables are given by \eqref{eq:ode_for_C}, \eqref{ODE:mean:GBM} and \eqref{ODE:var:GBM}. The objective function of the optimal control problem \eqref{eq:optimal:control:problem} can be expressed as $F(x(T)) = -y_Tm_T$ and
\begin{align*}
L(t,z(t),v(t)) &= \frac{A_t^2}{2\lambda^2(m_t^2+S_t)} + \frac{A_t(\lambda^2 + B_t - \kappa)}{\lambda^2 m_t} + \frac{(\lambda^2 + B_t - \kappa)^2}{2\lambda^2} + A_tC_t + y_t A_t \\
& \quad +m_t \left( C_t(\lambda^2 + B_t - \kappa) + A_t D_t + y_t(\lambda^2 + B_t) \right) \\
& \quad + (m_t^2 + S_t) \left( \frac{1}{2}\lambda^2 C_t^2 + D_t (\lambda^2 + B_t - \kappa)+\frac{1}{2} + \lambda^2 y_t C_t \right) \\
& \quad + (m_t^3 + 3m_tS_t) \left( \lambda^2C_tD_t + \lambda^2 y_t D_t \right) + (m_t^4+6m_t^2S_t + 3S_t^2)\frac{\lambda^2}{2}D_t^2, 
\end{align*}
where, in order to derive the cost function above, the first two inverse moments of $X_{t}$ with respect to $Q$ have been approximated. Due to the non-negativity of the GBM, we use the approximation $\CExpec{X_{t}^{-1}}{Q}\approx\CExpec{X_{t}}{Q}^{-1}\!=m_{t}^{-1}$ and $\CExpec{X_{t}^{-2}}{Q}\approx \CExpec{X_{t}}{Q}^{-2}\!=(S_{t}+m_{t}^{2})^{-1}$, whose accuracy has been investigated in \cite{ref:Garcia-01}.

%%%%%%%%%%%%%%%%%%%%%%%%%%%%%%%%%%%%%%%%%%%%%%%%%%%%%%%%%%%%%%%%%%%%%%%%%%%%%%%%%%%%%%%%%%%%%%%%%%%%%%%%%%%%
%%%%%%%%%%%%%%%%%%%%%%%%%%%%%%%%%%%%%%%%%%%%%%%%%%%%%%%%%%%%%%%%%%%%%%%%%%%%%%%%%%%%%%%%%%%%%%%%%%%%%%%%%%%%
\section{Parameter inference} \label{sec:parameter:inference} \
The goal of this section is to outline the use of the techniques, developed so far for path estimation, for inference of parameters in a hidden Markov model. We consider a class of dynamical systems
\begin{equation}\label{eq:system:unknown:paramters}
\drv X^\kappa_{t} = f(X^\kappa_{t},\kappa) \drv t + \sigma(X^\kappa_{t},\kappa) \drv W_{t}, \quad X^\kappa_{0}=x_{0}, \quad 0\leq t\leq T,
\end{equation}
parametrized by $\kappa$. The observation process can be modeled by \eqref{eq:measurement:model:cts}, but as discussed in the next section, the approach discussed below can also be used with necessary modifications for a discrete observation process.

%In the estimation problems considered up to now, we presumed that the underlying hidden Markov model \eqref{sec:one:eq1} is already known. However, in many applications it is unknown, or consists of unknown parameters, which motivates the task of designing (or learning) a hidden Markov model from the given observed data \eqref{eq:measurement:model:cts}.
%We consider a family of models for the underlying system of the form
%
%that is assumed to satisfy Assumption~\ref{ass:sde_main_ass}. 
 As a natural notation, for each parameter $\kappa$, the probability distribution of $X^\kappa_{[0,T]}$ on $\SC{C}$ will be denoted by $\priork$. Given a sample path $\{y_t: 0\leq t\leq T\}$ of the observation process $Y_{[0,T]}$, the objective is to select an optimal $\kappa^{\star}\in\R^{d}$ such that the observation process $(Y_{t})_{t\in[0,T]}$ in \eqref{eq:measurement:model:cts} has a high probability of reproducing the given data $y$. This is basically the inference scheme based on classical maximum likelihood estimation, and we propose an algorithm similar to the  
lines of  \emph{expectation maximization (EM) algorithm} (see \cite{ref:Cappe-05} for a comprehensive survey), which aims to obtain the optimal $\kappa^\star$ through multiple iterations.
 %to compute the maximum likelihood estimate in a tractable manner, which is widely used in statistical inference problems in hidden Markov models. 
 Recalling  \eqref{eq:information:term}, for each $\kappa$, we define
$I^{\kappa}(H_{T}(\cdot,y)) := -\log\left( \int \exp\left( - H_{T}(\cdot,y) \right)\drv \priork  \right)$. As already noted  in Section \ref{sec:motivation},  for each parameter $\kappa$, the term $I^{\kappa}(H_{T}(\cdot,y))$ provides the total information available through the sample path $y$, and  can be interpreted as the negative log-likelihood of $y$ given the parameter $\kappa$. 
However, minimizing this negative log-likelihood function, even if numerical evaluation of it can be done, usually is a hard problem. 
But, as mentioned in Section \ref{sec:motivation}, Lemma~\ref{Lemma:KL:cost} and non-negativity of the relative entropy together imply that an upper bound to this negative log-likelihood term is given by the apparent information,  $\mathcal{F}(Q,\kappa):=\KL{Q}{\priork} + \CExpec{H_{T}(\cdot,y)}{Q}$. The advantage of this observation is that this upper bound to the negative log-likelihood function is also the objective function in Problem~\ref{prob:optimization:problem}, for which the program for finding the minimizer $Q$ is by now well-established.  Therefore instead of minimizing the actual negative log-likelihood, we minimize an upper bound of it. The path to find the {\em right} parameter $\kappa$ corresponding to the sample path $y$ is now quite standard in statistics. After initialization of the parameter $\kappa$, we find the optimal $Q$ by solving the Problem \ref{prob:optimization:problem}, and then in the subsequent step, for this $Q$ we obtain the optimal parameter $\kappa$ by minimizing $\mathcal{F}(Q,\kappa)$. This yields an iterative EM-type algorithm whose details are given below. 

 \begin{table}[!htb]
\centering 
\begin{tabular}{c} 
\hspace{113mm} \text{ }\\
  \Xhline{3\arrayrulewidth} \vspace{-3mm}\\ 
\hspace{21.2mm}{\bf{\hypertarget{algo:1}{EM-type algorithm} }}  \hspace{34.2mm} \\ \vspace{-3mm} \\ \hline \vspace{-0.5mm}
\end{tabular} \\
\vspace{-1mm}
  \begin{tabular}{l l}
{\bf initialize} & $i=0,\ \kappa_{i}:=\hat{\kappa}_{0}$\\
{\bf while} 	& $i\leq M$ \\
{\bf Step 1: } & compute $Q_{i}$ by solving Problem~\ref{prob:optimization:problem} with parameter $\kappa_{i}$\\
{\bf Step 2: } & update parameter as $\kappa_{i+1} \in\arg\min\limits_{\kappa}\mathcal{F}(Q_{i},\kappa)$\\
{\bf Step 3: } & set $i\rightarrow i+1$
\vspace{-7mm}
  \end{tabular}
\begin{tabular}{c}
\hspace{-2mm} \phantom{ {\bf{Algorithm:}} Optimal Scheme for Smooth Optimization}\hspace{29.5mm} \\ \vspace{-1.0mm} \\\Xhline{3\arrayrulewidth}
\end{tabular}
\end{table}

\begin{remark}{\rm
Analyzing convergence of the above algorithm and consistency of the above corresponding estimator is the next important step and will be addressed in our future projects.}  
\end{remark}

We refer to Section~\ref{sec:examples} for a numerical visualization of this variational parameter inference method applied to two examples and to Section~\ref{sec:conclusion} for a discussion about convergence and consistency of the estimator as a topic of further research.

%%%%%%%%%%%%%%%%%%%%%%%%%%%%%%%%%%%%%%%%%%%%%%%%%%%%%%%%%%%%%%%%%%%%%%%%%%%%%%%%%%%%%%%%%%%%%%%%%%%%%%%%%%%%
\section{Discrete time measurement model} \label{sec:discrete:measurements}
%So far, we have considered a continuous time measurement model \eqref{eq:measurement:model:cts}.
 In most practical examples, the measurements of physical quantities are processed by computers, and as such the data available are obtained only at discrete times, potentially restricted to a low number. The goal of this section is to outline how the discussed variational approximation scheme adapts naturally to such cases with obvious modifications.  

In this case the signal process \eqref{sec:one:eq1} is observed through noisy measured data $y :=\{y_{k}\}_{k=1}^{N}$ at discrete times $t_1\leq t_2\leq\hdots\leq t_N\leq T$. The canonical model for the observation process is thus given by 
	\begin{equation}\label{sec:one:eq:observation}
	 Y_k=h(X_{k},t_{k})+\rho_k,\quad \text{for }k=1,\hdots,N,
	\end{equation}
	where $X_{k}:=X_{t_{k}}$, $h:\R^{n}\times \R\to\R^{n}$ is a measurable function, the $\rho_k$ are $\R^n$-valued i.i.d. Gaussian random variables with zero mean and covariance $R_k$, and they are independent of $x_{0}$ and $\sigma(W_{s}: s\leq T)$. 
	%We will show in this section that the variational approximation scheme introduced in this article, with minor modifications, supports such a discrete time measurement model as well.
We consider $m$ such that $t_{m}\leq t < t_{m+1}$ and similarly to Section~\ref{sec:model:setup} define the filter density $p$ and smoothing density $\Ps$ by
\begin{align}
\Expec{\phi(X(t))|Y_1,\hdots,Y_m,x_0} &= \int \phi(x) p(x,t) \ \drv x \label{eq:def:filter:density:discrete} \\
\Expec{\phi(X(t))|Y_1,\hdots,Y_N,x_0} &= \int \phi(x) \Ps(x,t) \ \drv x,\label{eq:def:smoothing:density:discrete}
\end{align}
where $\phi$ is any measurable function from $\R^n$ to $\R$. It is well known (see \cite[Appendix]{ref:Eyink-00} for a derivation) that the smoothing can be expressed as
\begin{equation}\label{eq:smoothing:density:discrete} 
	\Ps(x,t)=\frac{p(x,t)w(x,t)}{\int_{\R^{n}}p(x,t)w(x,t)\drv x},
\end{equation}
where $p(x,t)$ and $w(x,t)$ in between the observation times are the solutions to the
	\begin{equation}\label{eq:filter:PDE:discrete}
	\text{Kolmogorov forward equation:} \hspace{10mm}
	\left\{ \begin{aligned}
	\drv p(x,t) &=  \mathcal{A}^{*} p(x,t) \drv t \\
	p(x,0) &= p_{0}(x),
	\end{aligned} \right.
	\end{equation}
	\begin{equation}\label{eq:v:PDE:discrete}
	\text{Kolmogorov backward equation:} \hspace{10mm}
	\left\{ \begin{aligned}
	\drv w(x,t) &=  -\mathcal{A} w(x,t) \drv t \\
	w(x,T) &= 1,
	\end{aligned} \right.
	\end{equation}
punctuated by jumps at the data points $t_{k}$ for $k=1,\hdots, N$
\begin{align}
p(x,t_{k}^{+}) &\propto p(x,t_{k}) \exp\left( y_{k}\transp R_{k}^{-1} h(x,t_{k}) - \frac{1}{2} h(x,t_{k})\transp R_{k}^{-1} h(x,t_{k}) \right) \label{eq:jump:cond:p}\\
w(x,t_{k}) &\propto w(x,t_{k}^{+}) \exp\left( y_{k}\transp R_{k}^{-1} h(x,t_{k}) - \frac{1}{2} h(x,t_{k})\transp R_{k}^{-1} h(x,t_{k}) \right) \label{eq:jump:cond:v}.
\end{align}
Similar to the continuous time measurement model, it can be shown that the smoothing density solves the Kolmogorov forward equation given by \eqref{eq:PDE:posterior2}, with drift function $g(x,t):=f(x)+a(x)\nabla\log w(x,t)$, where $w$ is the solution to \eqref{eq:v:PDE:discrete}. As before, we denote the prior probability measure by $\prior(A) = \Prob{X_{[0,T]}\in A}$ and the posterior probability measure, induced by the solution to \eqref{eq:PDE:posterior2}, by $\post(A,Y) = \Prob{X_{[0,T]}\in A|\FTY}$, where $\FTY=\sigma(x_{0},Y_{1},\hdots, Y_{N})$. 
%We denote the process $X_{[0,T]}$ by $x$ and $\{Y_{i}\}_{i=1}^{N}$ by $y$.
Let $y_k$ denote a realization of the observation process at the time $t_k$.
The variational approximation derived in Section~\ref{sec:motivation}, and, in particular, Problem~\ref{prob:optimization:problem} carries over to the discrete time observation setting considered here. As before, the path to the objective function starts from Lemma \ref{Lemma:KL:cost}, which holds in this case with 
\begin{equation} \label{eq:KL:cost}
H_{T}(X,y) := \sum_{i=1}^{N} \left( \frac{1}{2}\| R_k^{-1}h(X_{i},t_{i}) \|^{2} - y_{i}\transp R_k^{-1} h(X_{i},t_{i}) \right).
\end{equation}
One way to see this is to recast the discrete model in the traditional setup of Section \ref{sec:model:setup}, and then use the Kallianpur-Striebel theorem.  To do this, first assume that without loss of generality $R_k = I$.  Define
%  $Y_{(k)}$ by 
%$$Y_{(k)} = \sum_{j=1}^k Y_j.$$
%and the continuous process $\bar{Y}$ by linear interpolation between the points $\{Y_{(k)}\}$. In other words, for $t_k\leq t< t_{k+1}$,
%$$\bar{Y}(t) = Y_{(k-1)}+ \f{t-t_k}{t_{k+1}-t_k}(Y_{(k)} - Y_{(k-1)}) =  Y_{(k-1)}+ \f{t-t_k}{t_{k+1}-t_k}Y_k,$$
 the function $\bar{h}: \SC{C}\times [0,T] \rt \R^n$ by
$$\bar{h}(x,t) = \sum_{k} (t_{k+1} - t_k)^{-1/2}h(x\circ \eta(t), \eta(t)) 1_{\{t_k\leq t<t_{k+1}\}},$$
where $\eta:[0,T] \rt [0,T]$ is defined as 
\begin{align*}
\eta(t) = t_k, \quad \mbox {if}\  t_k \leq t<t_{k+1}.
\end{align*}
Define the observation model $\tilde{Y}_t =\int_0^t \bar{h}(X_{s},s)\ ds + B_t$.
Notice that for each $k$, 
$$\tilde{Y}_{k+1} - \tilde{Y_k} = (t_{k+1}-t_k)^{1/2} h(X_k,t_k) +(B(t_{k+1}) - B(t_k)),$$
and hence 
$$(t_{k+1}-t_k)^{-1/2}(\tilde{Y}_{k+1} - \tilde{Y_k}) = h(X_k,t_k) + \tilde{\rho}_k,$$
where $\tilde{\rho_k} \stackrel{Law} = \rho_k \sim \SC{N}(0,I).$ In other words, $(t_{k+1}-t_k)^{-1/2}(\tilde{Y}_{k+1} - \tilde{Y_k}) \stackrel{Law}= Y_k$, and in this sense the discrete measurement model can be subsumed in the observation model given by $\tilde{Y}_t =\int_0^t \bar{h}(X_{s},s)\ ds + B_t$.

Notice that by the definitions of $\tilde{Y}$ and $\bar{h}$, the exponent in Kallianpur-Striebel formula is given by
\begin{align*}
\int_0^T \!\! \f{1}{2}\|\bar{h}(X,s)\|^2\drv s  - \! \int_0^T \! \! \bar{h}(X,s)\drv \tilde{Y}(s)\!&  = \sum_{k=1}^{N} \! \left( \! \frac{1}{2}\| h(X_{k},t_{k}) \|^{2} - \f{(\tilde{Y}_{k+1}\! - \tilde{Y_k})\transp}{(t_{k+1}\!-t_k)^{1/2}} h(X_{k},t_{k}) \!\! \right)\\
& \stackrel{Law}= \sum_{k=1}^{N} \left( \frac{1}{2}\| h(X_{k},t_{k}) \|^{2} -  Y_i\transp h(X_{k},t_{k}) \right),
\end{align*}
which leads to \eqref{eq:KL:cost}.
Therefore, in this case the objective function in Problem~\ref{prob:optimization:problem} can be expressed as
\begin{equation}\label{eq:obj_fct:discrete}
\KL{Q}{\prior} \!+ \CExpec{H_{T}(\cdot,y)}{Q} \!=\! \int_{0}^{T}\! \CExpec{\frac{1}{2}\norm{u(X_{t},t)-f(X_{t})}_{a(X_{t})}^2 + \iota(X_{t},t)}{Q}\! \drv t,
\end{equation}
where
\begin{equation} \label{eq:itoa}
	  \iota(X_{t},t) = \sum_{i=1}^{N} \left( y_{i}\transp R_k^{-1} h(X_{i},t_{i}) -\frac{1}{2}\| R_k^{-1}h(X_{i},t_{i}) \|^{2} \right) \delta(t-t_i).
\end{equation}
Section~\ref{sec:prescribed_marginal_law} is independent of the considered measurement model, and by following Section~\ref{sec:optimal:control:problem:formulation} we arrive at an optimal control problem \eqref{eq:optimal:control:problem}, where the cost functional is replaced by \eqref{eq:obj_fct:discrete}. The derivation of necessary conditions for global optimality of the optimization problem \eqref{eq:optimal:control:problem}, compared to the continuous time measurement model, here is somewhat nonstandard, due to the Dirac delta terms \eqref{eq:itoa} involved in the Lagrangian. However, the problem can be seen as an OCP with so-called \emph{intermediate constraints}, for which an extension of the PMP is available \cite{ref:Dmitruk-08}.
 
\begin{assumption}\label{ass:PMP:discrete} \em{
Let the process $(z^{\star}(t),v^{\star}(t))_{t\in[0,T]}$ be a local minimizer for the optimal control problem \eqref{eq:optimal:control:problem}, that satisfies
\begin{enumerate}[(i)]
\item Assumptions~\ref{ass:PMP}\eqref{ass:PMP:i} and \eqref{ass:PMP:ii};
\item $v^{\star}$ is measurable and essentially bounded.
\end{enumerate} }  
\end{assumption}
 
\begin{lemma}[Extended PMP] \label{lem:maximum:principle:discrete}
Let the process $(z^{\star}(t),v^{\star}(t))_{t\in[0,T]}$ be a local minimizer for the problem \eqref{eq:optimal:control:problem}. Given Assumption~\ref{ass:PMP:discrete}, then there exists an absolutely continuous function $p:[0,T]\to   \mathcal{Z}$ satisfying
\begin{enumerate}
\item the adjoint equation $\dot{p}(t)=-\nabla_z\inprod{p(t)}{H(z^{\star}(t),v^{\star}(t))} - \nabla_z L(t,z^{\star}(t),v^{\star}(t))$ for almost all $t\in[0,T];$
\item the transversality conditions $p(t_{i}) = p(t_{i}^{-}) - \nabla_{z}\CExpec{\iota(X,t_{i})}{Q}$ for $i=1,\hdots,N$ and $p(T)=0;$
\item the maximum condition 
\begin{align*}
&\inprod{p(t)}{H(z^{\star}(t),v^{\star}(t))} +  L(t,z^{\star}(t),v^{\star}(t))  \\ 
&\quad = \sup_{v\in\mathcal{M}([0,T],\mathcal{V})} \inprod{p(t)}{H(z^{\star}(t),v(t))} +  L(t,z^{\star}(t),v(t)) \text{ for almost all } t\in[0,T].
\end{align*}
\end{enumerate}  
\end{lemma}
\begin{proof}
Follows directly from \cite{ref:Dmitruk-08}, when transforming problem \eqref{eq:optimal:control:problem} into an OCP with intermediate constraints.
\end{proof}
 
\begin{remark} \rm{
\begin{enumerate}
\item Note that the data (measurements) enter the expression through the cost function, namely the term \eqref{eq:itoa}, which is nonzero only at measurement times $\{t_{i}\}_{i=1}^{N}$ and leads to jumps in the adjoint state. 
\item Lemma~\ref{lem:maximum:principle:discrete}, basically leads to a boundary value problem, that provides necessary conditions for optimality of Problem~\ref{prob:optimization:problem}. See Section~\ref{sec:computational:complexity} for a discussion about how to numerically solve it. We refer to the numerical examples in Section~\ref{sec:examples} for the performance of such a solution. 
\end{enumerate} }
\end{remark}

%%%%%%%%%%%%%%%%%%%%%%%%%%%%%%%%%%%%%%%%%%%%%%%%%%%%%%%%%%%%%%%%%%%%%%%%%%%%%%%%%%%%%%%%%%%%%%%%%%%%%%%%%%%%
%%%%%%%%%%%%%%%%%%%%%%%%%%%%%%%%%%%%%%%%%%%%%%%%%%%%%%%%%%%%%%%%%%%%%%%%%%%%%%%%%%%%%%%%%%%%%%%%%%%%%%%%%%%%  
	  \section{Simulation results} \label{sec:examples} \
	  In this section, we present two examples to illustrate the performance of the variational approximation method introduced. Both examples have important applications in mathematical finance. 
	  As a first example, we consider the geometric Brownian motion that we introduced as a running example in Sections~\ref{sec:GBM:example:1}, \ref{sec:GBM:example:2} and \ref{sec:GBM:example:3}. The second example is concerned with the Cox-Ingersoll-Ross process, that is often used for describing the evolution of interest rates \cite{ref:Cox-85}.

	 \subsection{Geometric Brownian motion} \label{sec:ex1:GBM}
%===============================================================================
  	As presented in Sections~\ref{sec:GBM:example:1}, \ref{sec:GBM:example:2} and \ref{sec:GBM:example:3} we consider a one-dimensional geometric Brownian motion (GBM) \eqref{eq:system:GBM1} and assume that the available data are noisy observations $\{y_{k}\}_{k=1}^{N}$ at time $t_{k}$, modeled by the observation process
	\begin{equation*}
		 Y_{k}=X_{t_{k}}+\rho_{k},
	\end{equation*}  
	 where $\{\rho_{k}\}_{k=1}^{N}$ are i.i.d.~normal random variables with zero mean, standard deviation $R$ and $t_{N}=T$. \vspace{1mm}

\np
\textit{PDE approach.} 
	As explained in Section~\ref{sec:discrete:measurements}, the smoothing density can be characterized by \eqref{eq:smoothing:density:discrete} that is the (normalized) product of two densities $w$ and $p$.
	The first density satisfies equation \eqref{eq:v:PDE:discrete} with jump conditions \eqref{eq:jump:cond:v} at the measurement times
%	\begin{equation*}
%	\frac{\partial}{\partial t}v(x,t) = -\eta x \frac{\partial}{\partial x}v(x,t) - \frac{\lambda^{2}x^{2}}{2} \frac{\partial^{2}}{\partial x^{2}}v(x,t),
%	\end{equation*}
	and terminal condition $w(x,T)=\tfrac{1}{\sqrt{2\pi}R}\exp\left( \tfrac{-(x-y_{N})^{2}}{2R^{2}} \right)$. Its marginals are shown in Figure~\ref{fig:GBM:likeli}.
	The second density, called the filter density, is given by equation \eqref{eq:filter:PDE:discrete} with jump conditions \eqref{eq:jump:cond:p}
%	\begin{equation*}
%	\frac{\partial}{\partial t}p(x,t) = (-\eta+\lambda^{2})p(x,t)+(-\eta x +2\lambda^{2}x)\frac{\partial}{\partial x}p(x,t) + \frac{1}{2}\lambda^{2}x^{2}\frac{\partial^{2}}{\partial x^{2}}p(x,t),
%	\end{equation*}
and initial condition $p(x,0) = \tfrac{1}{\sqrt{2\pi}x\sigma}\exp\left( \tfrac{-(\log x - \mu)^{2}}{2\sigma^{2}} \right)$ that is given by \eqref{eq:system:GBM1}. Its marginals are shown in Figure~\ref{fig:GBM:filter}. The smoothing density is depicted in Figure~\ref{fig:GBM:smoothing} as the solid line. \vspace{1mm}

\np
\textit{Variational approximation.}
Following Section~\ref{sec:GBM:example:2}, the drift function for the approximating SDE~\eqref{eq:approx_SDE} has to be chosen as \eqref{eq:GBM:drift}. The optimal control problem can be formulated along the lines of Section~\ref{sec:GBM:example:3}, choosing the discrete-time measurement setting presented in Section~\ref{sec:discrete:measurements}.
Note that Assumption~\ref{ass:PMP:discrete} can be easily verified to hold, if we restrict the optimizers in \eqref{eq:optimal:control:problem} to bounded controls. We solve the boundary value problem obtained from Lemma~\ref{lem:maximum:principle:discrete} under the assumption that the smoothing density at initial time is available, see Section~\ref{sec:computational:complexity} for a discussion about this assumption. The solution is depicted in Figure~\ref{fig:GBM:smoothing} as the dashed line. Finally, Figure~\ref{fig:GBM:KL} shows the relative entropy between the marginals of the smoothing density obtained by the PDE approach and the variational method, and hence reflects the accuracy of the variational approximation. \vspace{1mm}

\np
\textit{Parameter inference.}
We consider the case where the drift parameter $\kappa$ in \eqref{eq:system:GBM1} is assumed to be unknown. Figure~\ref{fig:GBM:parameter} shows the performance of the \hyperlink{algo:1}{EM}-Algorithm introduced in Section~\ref{sec:parameter:inference} for an initial guess $\hat{\kappa}_{0}=4$ of the unknown parameter. It can be seen that the estimator $\hat{\kappa}$ is close to the true value of $\kappa=1$ indicating the efficacy of our algorithm. Also, the algorithm converges quite rapidly.

	\begin{figure}[!htb]
	   \centering
	   		\begin{subfigure}[b]{0.3\textwidth}
			\newlength\figureheight 
			\newlength\figurewidth 
			\setlength\figureheight{3cm} 
			\setlength\figurewidth{4.5cm} 
    		        \includegraphics[scale = 0.65]{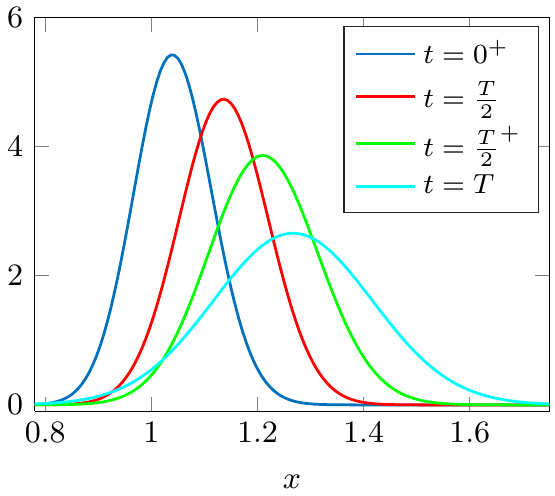} 
			\caption{Density $w(x,t)$} \label{fig:GBM:likeli}
		\end{subfigure}
	%	\vspace{-3mm}
		\begin{subfigure}[b]{0.3\textwidth}
		        \setlength\figureheight{3cm} 
			\setlength\figurewidth{4.5cm} 
		        \includegraphics[scale = 0.65]{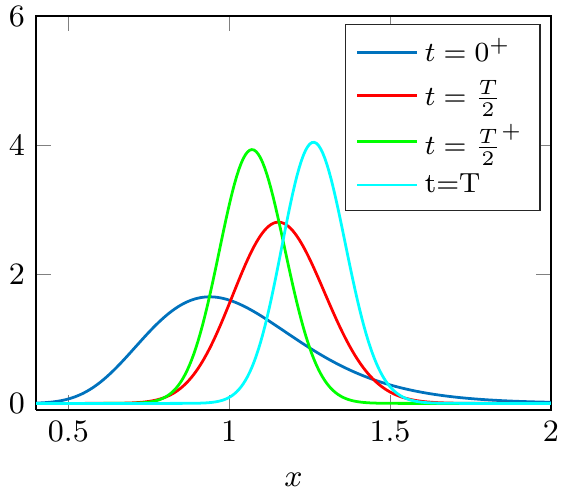} 
			\caption{Filter density $p(x,t)$} \label{fig:GBM:filter}
		\end{subfigure}
				\begin{subfigure}[b]{0.35\textwidth}
    		        \setlength\figureheight{3cm} 
			\setlength\figurewidth{4.5cm} 
		         \includegraphics[scale = 0.65]{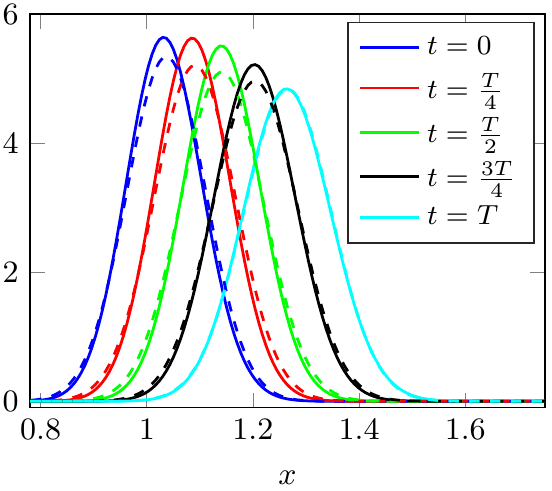} 
			\caption{Smooting \! density \!$\Ps(x,t)$} \label{fig:GBM:smoothing}
		\end{subfigure}
	%	\vspace{-3mm}
		\begin{subfigure}[b]{0.4\textwidth}
		        \setlength\figureheight{3cm} 
			\setlength\figurewidth{4.5cm} 
		    \includegraphics[scale = 0.65]{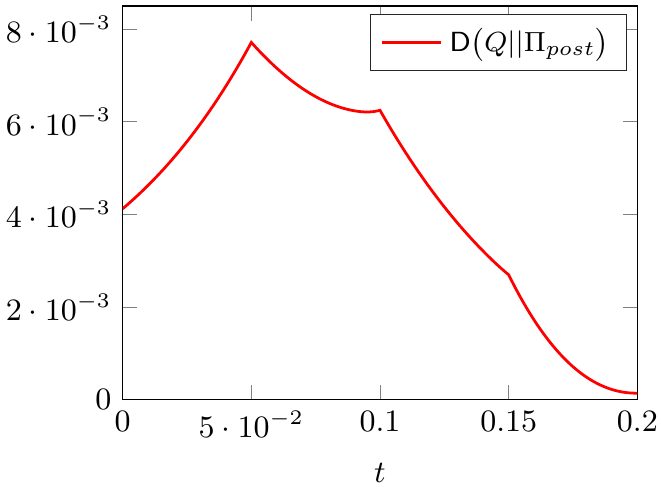} 
			\caption{Relative entropy} \label{fig:GBM:KL}
		\end{subfigure}
		\begin{subfigure}[b]{0.4\textwidth} 
		       \setlength\figureheight{3cm} 
			\setlength\figurewidth{4.5cm} 
		        \includegraphics[scale = 0.65]{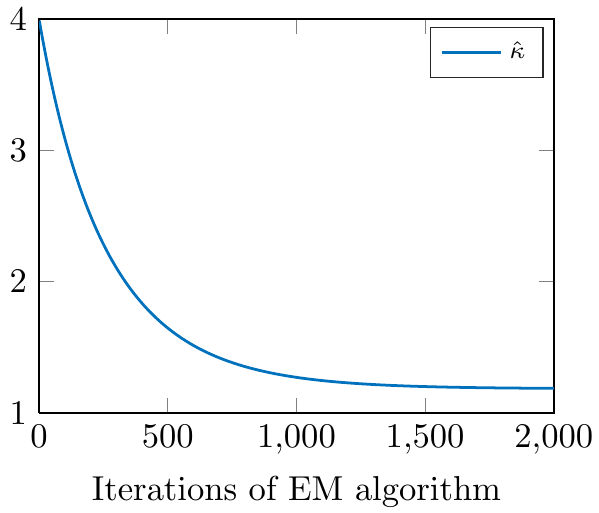} 
			\caption{Parameter inference, $\hat{\kappa}=1.1867$} \label{fig:GBM:parameter}
		\end{subfigure}
	%	\vspace{-3mm}
    \caption{Geometric Brownian motion: Comparison of the PDE solution (solid) versus the variational approach (dashed). The considered numerical values are: $\kappa =1$, $\lambda=0.1$, $R=0.15$, $T=0.2s$, $\mu=0$, $\sigma=0.25$, $N=4$, $t_{1}=T/4$, $t_{2}=T/2$, $t_{3}=3T/4$ and $t_{4}=T$.}
    \label{fig:GBM:variational}
\end{figure}

%===============================================================================
 \subsection{Cox-Ingersoll-Ross} \label{sec:ex3:CIR} 
 Consider as underlying system a Cox-Ingersoll-Ross (CIR) process
	\begin{equation} \label{eq:system:CIR}
	 	\drv X_{t} = \kappa (b- X_{t})\drv t+\lambda \sqrt{X_{t}} \drv W_{t},\quad X_{0}=x_0\sim \mathcal{N}(\mu,\sigma), 
	\end{equation}
	for $0\leq t \leq T$ and assume that the available data are noisy observations $\{y_{k}\}_{k=1}^{N}$ at time $t_{k}$, modeled by an observation process
	\begin{equation*} 
		 Y_{k}=X_{t_{k}}+\rho_{k},
	\end{equation*}  
	 where $\rho_{k}$ are i.i.d.~normal random variables with zero mean, standard deviation $R$ and $t_{N}=T$. \vspace{1mm}
	 
	 \np
	\textit{PDE approach.} 
	As in the GBM example \ref{sec:ex1:GBM}, we solve the underlying PDEs introduced in Section~\ref{sec:discrete:measurements}, to characterize the smoothing density as the (normalized) product of two densities $v$ and $p$. Figure~\ref{fig:CIR:likeli} shows the marginals of the first density $w$, the filter density $p$ is depicted in Figure~\ref{fig:CIR:filter} and the smoothing density in Figure~\ref{fig:CIR:smoothing} as the solid line. \vspace{1mm}

%	As explained in Section~\ref{sec:discrete:measurements}, the smoothing density can be characterized as the (normalized) product of two densities $v$ and $p$.
%	The first density satisfies equation \eqref{eq:v:PDE:discrete} with jump conditions \eqref{eq:jump:cond:v} at the measurements
%%	\begin{equation*}
%%	\frac{\partial}{\partial t}v(x,t) = -a(b-x) \frac{\partial}{\partial x}v(x,t) - \frac{\lambda^{2}x}{2} \frac{\partial^{2}}{\partial x^{2}}v(x,t),
%%	\end{equation*}
%	and terminal condition $v(x,T)=\tfrac{1}{\sqrt{2\pi}R}\exp\left( \tfrac{-(x-y_{N})^{2}}{2R^{2}} \right)$. Its marginals are shown in Figure~\ref{fig:CIR:likeli}.
%	The second density, called the filter density, is given by equation \eqref{eq:filter:PDE:discrete} with jump conditions \eqref{eq:jump:cond:p}
%%	\begin{equation*}
%%	\frac{\partial}{\partial t}p(x,t) = a p(x,t)+(-a(b-x)+\lambda^2)\frac{\partial}{\partial x}p(x,t) + \frac{1}{2}\lambda^{2}x\frac{\partial^{2}}{\partial x^{2}}p(x,t),
%%	\end{equation*}
%and initial condition $p(x,0) = \tfrac{1}{\sqrt{2\pi}\sigma}\exp\left( \tfrac{-(x - \mu)^{2}}{2\sigma^{2}} \right)$ that is given by \eqref{eq:system:CIR}. Its marginals are shown in Figure~\ref{fig:CIR:filter}.
%The smoothing density is depicted in Figure~\ref{fig:CIR:smoothing} as the solid line.\vspace{1mm}

\np	
\textit{Variational approximation.} 
The variational approximation is derived similarly to the GBM example \ref{sec:ex1:GBM}, where we chose a drift function for the approximating SDE~\eqref{eq:approx_SDE}, according to Theorem~\ref{thm:brigo_mixture_multi}, as
	\begin{equation} \label{eq:CIR:drift}
	u(x,t) = \frac{1}{2}\lambda^2 + A(t) + B(t)x + \lambda^2x(C(t)+D(t)x).
	\end{equation}
%	To express the Lagrangian in \eqref{eq:obj_fct:discrete} as a function of only the state variables and control inputs, one can see directly from \eqref{eq:system:GBM} and \eqref{eq:GBM:drift}, that the first inverse moment of $X_{t}$ with respect to $Q$ needs to be approximated. Due to the non-negativity of the CIR, we use $\CExpec{X_{t}^{-1}}{Q}\approx\CExpec{X_{t}}{Q}^{-1}\!=m_{t}^{-1}$, whose approximation quality has been studied in \cite{ref:Garcia-01}. 
%Assumption~\ref{ass:PMP:discrete} can be easily verified to hold, if we restrict the optimizers in \eqref{eq:optimal:control:problem} to bounded controls. We solve the ODE system obtained from Lemma~\ref{lem:maximum:principle:discrete}, assuming that we have access to the smoothing density at terminal time $T$, see Section~\ref{sec:computational:complexity} for a discussion about that assumption. 
The variational approximation to the smoothing density is depicted in Figure~\ref{fig:CIR:smoothing} as the dashed line. Finally, Figure~\ref{fig:CIR:KL} shows the relative entropy between the marginals of the smoothing density obtained by the PDE solution and the variational approximation. 	\vspace{1mm}

\begin{figure}[!htb]
  \centering
		\begin{subfigure}[b]{0.3\textwidth}
			% \newlength\figureheight 
			%\newlength\figurewidth 
			\setlength\figureheight{3cm} 
			\setlength\figurewidth{4.5cm} 
                        \includegraphics[scale = 0.65]{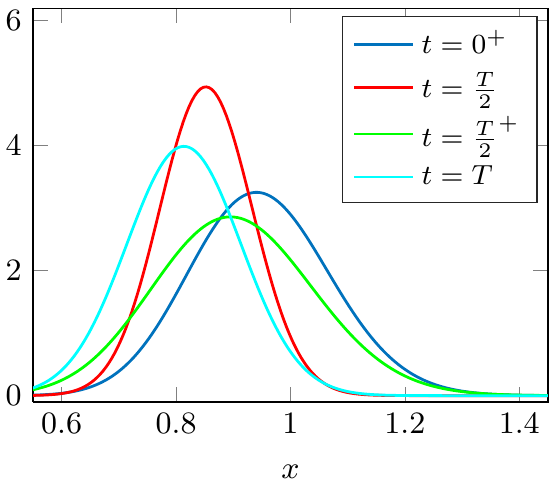} 
		        \caption{Density $w(x,t)$} \label{fig:CIR:likeli}
		\end{subfigure}
	%	\vspace{-3mm}
		\begin{subfigure}[b]{0.3\textwidth}
		        %\newlength\figureheight 
			%\newlength\figurewidth 
			\setlength\figureheight{3cm} 
			\setlength\figurewidth{4.5cm} 
		       \includegraphics[scale = 0.65]{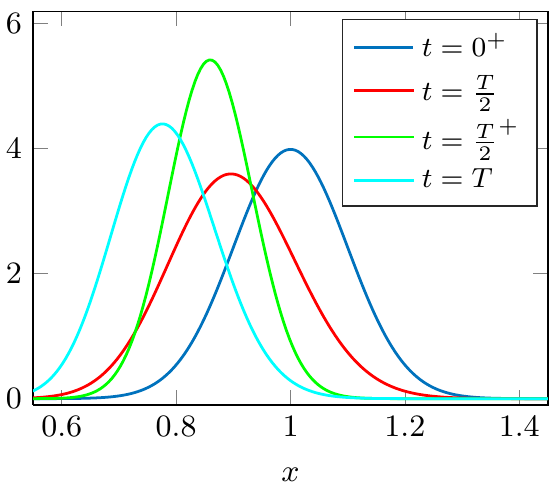}
			\caption{Filter density $p(x,t)$} \label{fig:CIR:filter}
		\end{subfigure} 
		%\vspace{-3mm}
		\begin{subfigure}[b]{0.35\textwidth}
		      % \newlength\figureheight 
		      % \newlength\figurewidth 
			\setlength\figureheight{3cm} 
			\setlength\figurewidth{4.5cm} 
		        \includegraphics[scale = 0.65]{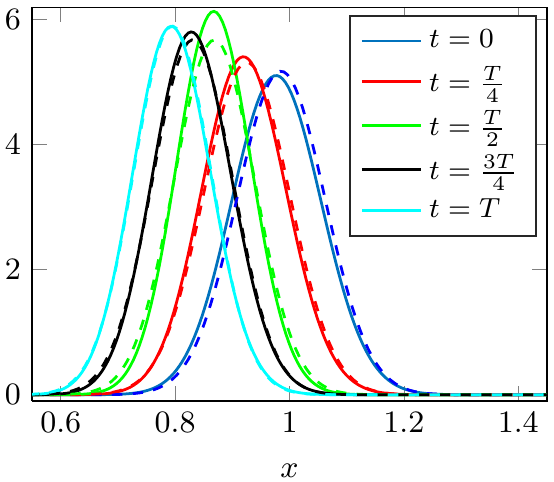} 
			\caption{Smooting density $\Ps(x,t)$} \label{fig:CIR:smoothing}
		\end{subfigure}
		\begin{subfigure}[b]{0.4\textwidth}
		       \setlength\figureheight{3cm} 
			\setlength\figurewidth{4.5cm} 
		       \includegraphics[scale = 0.65]{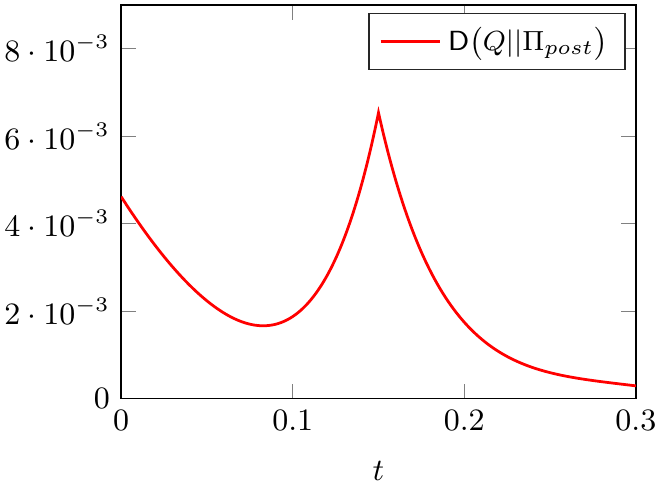} 
			\caption{Relative entropy} \label{fig:CIR:KL}
		\end{subfigure}
		\begin{subfigure}[b]{0.4\textwidth}
		         \setlength\figureheight{3cm} 
			\setlength\figurewidth{4.5cm} 
		  \includegraphics[scale = 0.65]{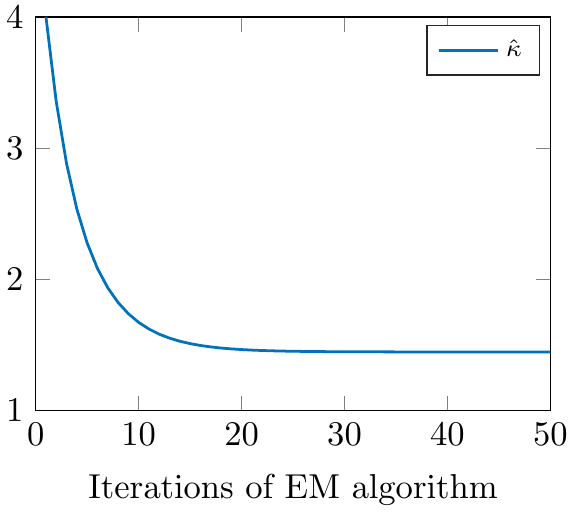} 
			\caption{Parameter inference, $\hat{\kappa}=1.4469$} \label{fig:CIR:parameter}
		\end{subfigure}	
		%\vspace{-3mm}	
    \caption{Cox-Ingersoll-Ross: Comparison of the PDE solution (solid) versus the variational approach (dashed). The considered numerical values are: $\lambda = 0.2$, $\kappa=1$, $b=0.3$, $\mu=1$, $\sigma=0.1$, $R=0.1$, $T=0.3$s, $N=2$, $t_{1}=T/2$ and $t_{2}=T$.}
    \label{fig:CIR:PDE}
\end{figure} 

\np
\textit{Parameter inference.}
We consider the case where the parameter $\kappa$ in the drift term of \eqref{eq:system:CIR} is unknown. Figure~\ref{fig:CIR:parameter} shows the \hyperlink{algo:1}{EM}-Algorithm introduced in Section~\ref{sec:parameter:inference} for an initial guess $\hat{\kappa}_{0}=4$ of the unknown parameter.	
 It can be seen that the estimator $\hat{\kappa}$ is close to the true value of $\kappa=1$ indicating the efficacy of our algorithm. Also, the algorithm converges very fast.

 \begin{table}[!htb]
\centering 
\caption{ {\bf Runtime comparison.} The presented variational approach for approximating the smoothing density is compared with the standard PDE approach for the two examples in Sections~\ref{sec:ex1:GBM} and \ref{sec:ex3:CIR}. All simulations are performed on a 2.3 GHz Intel Core i7 processor with 8 GB RAM using Matlab.}
\label{tab:runtime}

  \begin{tabular}{c @{\hskip 3mm} | @{\hskip 3mm} c @{\hskip 2mm} c }
    & Geometric Brownian motion & Cox-Ingersoll-Ross  \\
    \hline
    Forward PDE \eqref{eq:filter:PDE:discrete} & 2.02 s & 1.33 s \\
    Backward PDE \eqref{eq:v:PDE:discrete} & 2.87 s & 2.38 s \\    
    Boundary value problem & 0.10 s & 0.23 s \\   
    \hline
 PDE approach\protect\footnotemark & 4.89 s & 3.70 s  \\
  Variational approach\protect\footnotemark & 2.97 s & 2.61 s  \\
  \end{tabular}
\end{table}
 \addtocounter{footnote}{-1}
 \footnotetext{consists of solving the two PDEs \eqref{eq:filter:PDE:discrete} and \eqref{eq:v:PDE:discrete}}
 \addtocounter{footnote}{1}
 \footnotetext{consists of the PDE \eqref{eq:v:PDE:discrete} and the boundary value problem in order to solve Problem~\ref{prob:optimization:problem}}

 Table~\ref{tab:runtime} summarizes the runtimes of the two numerical examples above. It can be seen that the boundary value problems provided by the maximum principle can be solved by roughly one magnitude faster than the backward PDE~\eqref{eq:v:PDE:discrete}, that is the reason for the speedup of the variational approach compared to the PDE approach. Moreover, it is highlighted that the main computational effort in the variational approach is needed to estimate the marginal smoothing density at initial time, which is done by solving the backward PDE \eqref{eq:v:PDE:discrete}. If, however, the backward density at initial time could be estimated in a more efficient way, e.g., by using an MCMC method, the proposed variational approximation method could be applicable to high-dimensional problems.

%%%%%%%%%%%%%%%%%%%%%%%%%%%%%%%%%%%%%%%%%%%%%%%%%%%%%%%%%%%%
\section{Conclusion} \label{sec:conclusion}
The paper is devoted to a variational method for estimating paths of a signal process in a hidden Markov model. In particular, this leads to approximations of smoothing density which can be used to reconstruct any past state of the signal process given a full set of observations. A crucial fact that plays an important role in our method is that the smoothing distribution is induced by a posterior SDE which itself is a modification of the original signal process.  
%
%In this article, we developed a method based on a variational principle to approximate the conditional density, called the smoothing density, of a fairly general hidden Markov model described by SDEs. We exploit the fact that the mentioned conditional density is known to be induced by a particular posterior SDE that, however, is numerically difficult to solve as its coefficients the solutions to a backward (S)PDE. 
The presented variational approach proposes an approximate SDE which minimizes the relative entropy between the posterior SDE and  a class of SDEs whose marginals belong to a chosen mixture of exponential families.
In the simplest case of normal marginals and a posterior SDE with constant diffusion term, the approximating SDE consists of a linear drift and constant diffusion term, which is well known. It is shown that the prescribed approximation scheme can be formulated as an optimal control problem, and necessary conditions for global optimality are obtained by the Pontryagin maximum principle. The resulting numerical methods have considerable computational advantages over numerically solving the underlying (S)PDEs, that is highlighted by two examples. The developed approximation scheme is then used for designing an efficient method for parameter inference for SDEs.  \vspace{1mm}

For future work, as mentioned in Section~\ref{sec:computational:complexity}, we aim to study how to efficiently estimate the backward density at initial time. Then, the presented variational approximation method reduces to solving a PMP-shooting-type boundary value problem that is tractable even in relatively large dimensions, compared to PDEs. Additionally, it would be interesting to study numerical methods specifically tailored to the boundary value problems resulting from the maximum principle, such as the shooting method, see \cite{ref:Stoer-02} for a comprehensive summary, as well as the approach of solving the optimal control problem via its weak formulation as pointed out in Remark~\ref{rem:SDP}.\\
Our future projects will also delve into analyzing the convergence of the EM-type algorithm used for parameter inference as well as the properties of the obtained estimators. We will also focus on refining the basic inference algorithm to get better efficiency and speed. One promising path to take in this direction would be designing of suitable adaptive EM-type algorithms. It is  also conceivable that the ideas mentioned in the paper can be combined with suitable MCMC schemes or techniques known as Assumed Density Filtering, see \cite{ref:Yuval-15}, to get better accuracy and efficiency in high-dimensional models.

% Acknowledgements should go at the end, before appendices and references

\acks{This work was supported by the ETH grant (ETH-15 12-2). The authors are grateful to Debasish Chatterjee, John Lygeros, and Peyman Mohajerin Esfahani for helpful discussions and pointers to references. H.K. acknowledges funding from the LOEWE research priority program CompuGene.}


\newpage

\appendix

%%%%%%%%%%%%%%%%%%%%%%%%%%%%%%%%%%%%%%%%%%%%%%%%%%%%%%%%%%%%%%%%%%%%%%%%%%%%%%%%%%%%%%%%%%%%%%%%%%%%%%%%%%%%
%%%%%%%%%%%%%%%%%%%%%%%%%%%%%%%%%%%%%%%%%%%%%%%%%%%%%%%%%%%%%%%%%%%%%%%%%%%%%%%%%%%%%%%%%%%%%%%%%%%%%%%%%%%%
%%%%%%%%%%%%%%%%%%%%%%%%%%%%%%%%%%%%%%%%%%%%%%%%%%%%%%%%%%%%%%%%%%%%%%%%%%%%%%%%%%%%%%%%%%%%%%%%%%%%%%%%%%%%%%%%%%%%%%%%%%%%%%%%%%%%%%%%%%%%%%%%%%%%%%%%%%%%%%%%%%%%%%%%%%%%%%%%%%%%%%%%%%%%%%%%%%%%%%%%%%%%%%%%%%%%%%%%%%%%%%%%%%%%%%%%%%%%%%%%		
\begin{appendix}
 \section{Derivation of Equation~\eqref{eq:PDE:posterior2}} \label{app:smoothing:SDE}
	We consider the one-dimensional case;  an extension to the multi-dimensional case is straightforward. According to \eqref{eq:smoothing:density:char} the smoothing density is given by $\Ps(x,t) = K(t) p(x,t) w(x,t)$, where $K(t):=\left(\int_{R^{n}}p(x,t)w(x,t)\drv x\right)^{-1}$. The main idea is to recall that the process $(K(t))_{t\in[0,T]}$ is known to be almost surely constant \cite[Theorem~3.2]{ref:Pardoux-81}. Therefore
	\begin{equation*}
	 \begin{aligned}
	  \frac{\partial}{\partial t}\Ps(x,t) &= K(t)p(x,t) \frac{\partial}{\partial t} w(x,t) + K(t) w(x,t) \frac{\partial}{\partial t} p(x,t) \\
					      &= \frac{\Ps(x,t)}{w(x,t)}\left( -f(x)w^{\prime}(x,t)-\frac{1}{2}a(x)w^{\prime \prime} (x,t) - w(x,t) h(x)\transp \drv Y_{t} \right)\\
				              &+  w(x,t) \left( -\left( \frac{f(x)\Ps(x,t)}{w(x,t)} \right)^{\prime} + \frac{1}{2} \left( \frac{a(x)\Ps(x)}{w(x,t)}\right)^{\prime\prime} +\frac{\Ps(x,t)}{w(x,t)} h(x)\transp \drv Y_{t} \right).
	 \end{aligned}
	\end{equation*}
	The proof follows by a straightforward computation. We compute in a preliminary step
	\begin{equation*}
	\begin{aligned}
	 \left( \frac{f(x)\Ps(x,t)}{w(x,t)} \right)^{\prime} \!\!\!&=\!\frac{1}{w^2(x,t)}\!\left( \left( f^\prime(x)\Ps(x,t)\!+\!f(x)\Ps^\prime(x,t) \right)w(x,t)\right.\left.\!-\!f(x)\Ps(x,t)v^\prime(x,t) \right)\!,
	\end{aligned}
	\end{equation*}
	and
	\begin{align*}
	 &\left( \frac{a(x)\Ps(x,t)}{w(x,t)}\right)^{\prime\prime}
%	  &= \frac{1}{\Pf^2(x,t)}\left\{\left( a^\prime(x)\Pst(x,t)+a(x)\Pst^\prime(x,t) \right) \Pl(x,t)\right. \\
%								     &\hspace{8mm}\left.-a(x)\Pst(x,t)\Pl^\prime(x,t) \right\}^\prime \\
      = \frac{1}{w(x,t)} \left( a''(x)\Ps(x,t)+2a'(x)\Ps'(x,t)+a(x)\Ps''(x,t) \right) \\
	 &\quad\hspace{8mm} - \frac{1}{w^2(x,t)}\left( 2w'(x,t)a'(x)\Ps(x,t) + 2w'(x,t)a(x)\Ps'(x,t)\right.\left. +a(x)\Ps(x,t)w''(x,t) \right) \\
	 &\quad\hspace{8mm} +\frac{1}{w^3(x,t)}\left( 2a(x)\Ps(x,t)w'(x,t)^2 \right).
	\end{align*}
	Using this two preliminaries, we get
	\begin{align*}
	  &\frac{\partial}{\partial t}\Ps(x,t) \\
	  &=-f'(x)\Ps(x,t)-f(x)\Ps'(x,t)+a'(x)\Ps'(x,t)+ \frac{1}{2}(a''(x)\Ps(x,t)+a(x)\Ps''(x,t))\\ 
	  &\quad -\frac{1}{w(x,t)}\left(a'(x)w'(x,t)\Ps(x,t) +a(x)w''(x,t)\Ps(x,t)+a(x)w'(x,t)\Ps'(x,t)\right) \\
	  &\quad +\frac{1}{w^2(x,t)}a(x)w'(x,t)^2\Ps(x,t) \\
	      &= -\bigg( f'(x)\Ps(x,t)+f(x)\Ps'(x,t)+\frac{1}{w(x,t)}\big( a'(x)w'(x,t)\Ps(x,t) \\
	      &\quad +a(x)w''(x,t)\Ps(x,t)+a(x)w'(x,t)\Ps'(x,t)\big)-\frac{1}{w^2(x,t)}(a(x)w'(x,t)^2\Ps(x,t))\bigg)\\
	      &\quad +\frac{1}{2}\big(a'(x)\Ps(x,t)+a(x)\Ps'(x,t)\big)'\\
	      &=-\bigg( \left( f(x)+a(x)\frac{w'(x,t)}{w(x,t)}\right)\Ps(x,t)\bigg)'+\frac{1}{2}\big( a(x)\Ps(x,t) \big)'' \\
	      &=-\bigg( \left( f(x)+a(x)\big(\log w(x,t) \big)'\right)\Ps(x,t)\bigg)'+\frac{1}{2}\big( a(x)\Ps(x,t) \big)'',
	\end{align*}
and as such \eqref{eq:PDE:posterior2} holds.

%%%%%%%%%%%%%%%%%%%%%%%%%%%%%%%%%%%%%%%%%%%%%%%%%%%%%%%%%%%%%%%%%%%%%%%%%%%%%%%%%%%%%%%%%%%%%%%%%%%%%%%%%%%%%%%%%%%%%%%%%%%%%%%%%%%%%%%%%%%%%%%%%%%%%%%%%%%%%%%%%%%%%%%%%%%%%%%%%%%%%%%%%%%%%%%%%%%%%%%%%%%%%%%%%%%%%%%%%%%%%%%%%%%%%%%%%%%%%%%%		

	 \section{Proof of Theorem~\ref{thm:brigo_mixture_multi}} \label{app:proof}
	    Consider an arbitrary curve $t\mapsto p(\cdot,\Theta^{(1)}_t,\hdots,\Theta^{(k)}_t)$ evolving in $\text{EM}(c^{(1)},\hdots,c^{(k)})$. Define a diffusion
	   \begin{equation*} 
	   \drv Z_{t} = u(Z_{t},t)\drv t + \sigma(Z_{t}) \drv B_{t} , \quad Z_0=x_0,
	   \end{equation*} 
	   with the given diffusion coefficient $a(\cdot)=\sigma(\cdot)\sigma(\cdot)\transp$. Clearly the density of $Z_{t}$ coincides with $p(\cdot,\Theta^{(1)}_t,\hdots,\Theta^{(k)}_t)$ if and only if $p(\cdot,\Theta^{(1)}_t,\hdots,\Theta^{(k)}_t)$ satisfies the Kolmogorov forward equation for $Z_{t}$, i.e., 
	  \begin{equation} \label{e:app:proof:FPE}
	  \begin{aligned}
	   \frac{\partial p(x,\Theta^{(1)}_t,\hdots,\Theta^{(k)}_t)}{\partial t} &=-\sum_{i=1}^n\frac{\partial}{\partial x_i}\left(u_i(x,t)p(x,\Theta^{(1)}_t,\hdots,\Theta^{(k)}_t)\right)\\
	   &\hspace{8mm}+\frac{1}{2}\sum_{i=1}^n\sum_{j=1}^n\frac{\partial^2}{\partial x_i\partial x_j}\left( a_{ij}(x)p(x,\Theta^{(1)}_t,\hdots,\Theta^{(k)}_t)\right).
	   \end{aligned}
	  \end{equation}
	  We will show this in two steps that \eqref{e:app:proof:FPE} holds for the proposed drift term. Consider the decomposition $u_i(x,t)=g_i(x,t)+\gamma_i(x,t)$ for all $i=1,\hdots, n,$ where
	  \begin{equation} \label{e:app:def:g}
	    g_i(x,t):=\frac{1}{2}\sum_{j=1}^n\frac{\partial}{\partial x_j}a_{ij}(x) + \frac{1}{2}\sum_{j=1}^n a_{ij}(x) \frac{\frac{\partial}{\partial x_j}p(x,\Theta^{(1)}_t,\hdots,\Theta^{(k)}_t)}{p(x,\Theta^{(1)}_t,\hdots,\Theta^{(k)}_t)}  \vspace{-2mm}
	  \end{equation}
	  and  \vspace{-2mm}
	  \begin{equation} \label{e:app:def:gamma}
	   \begin{aligned}
	    \gamma_i(x,t) :=& \frac{-1}{p(x,\Theta^{(1)}_t,\hdots,\Theta^{(k)}_t)} \sum_{\ell=1}^k \nu_\ell p_\ell(x,\Theta^{(\ell)}_{t})\inprod{\dot{\Theta}^{(\ell)}_{t}}{\mathcal{I}_i^{(\ell)}(x)}.
	   \end{aligned}
	  \end{equation}
We use the shorthand notation $p(x,\Theta^{(1:k)}_t):=p(x,\Theta^{(1)}_t,\hdots,\Theta^{(k)}_t)$.
	  \begin{claim}\label{claim:1}
	  The functions $g_i$ defined in \eqref{e:app:def:g} for all $i=1,\hdots,n$ satisfy
	   \begin{equation*}
	   \sum_{i=1}^n\frac{\partial}{\partial x_i}\left( g_i(x,t)p(x,\Theta^{(1:k)}_t)\right)\!=\!\frac{1}{2}\sum_{i=1}^n\sum_{j=1}^n\frac{\partial^2}{\partial x_i\partial x_j}\left( a_{ij}(x)p(x,\Theta^{(1:k)}_t)\right).
	   \end{equation*}
	  \end{claim}Claim~\ref{claim:1} follows from a straightforward computation, see Appendix~B in the extended version \cite{ref:SutterVar-15} for a detailed derivation.
	  \begin{claim}
	   The functions $\gamma_i$ defined in \eqref{e:app:def:gamma} for all $i=1,\hdots,n$ satisfy
	   \begin{equation*}
	    \frac{\partial}{\partial t}p(x,\Theta^{(1:k)}_t)= -\sum_{i=1}^n \frac{\partial}{\partial x_i}\bigg( \gamma_i(x,t)p(x,\Theta^{(1:k)}_t) \bigg).
	   \end{equation*}
	  \end{claim}
	  \indent\textit{Proof.} \vspace{-3mm}
	  \begin{align*}
	   \frac{\partial}{\partial t}p(x,\Theta^{(1:k)}_t) &= \sum_{\ell=1}^k \nu_\ell \frac{\partial}{\partial t}p_\ell(x,\Theta^{(\ell)}_{t})=\sum_{\ell=1}^k \nu_\ell \inprod{\dot{\Theta}^{(\ell)}_{t}}{c^{(\ell)}(x)-\nabla_{\theta}\psi_\ell(\Theta^{(\ell)}_{t})}p_\ell(x,\Theta^{(\ell)}_{t}).
	  \end{align*}
	  Moreover, \vspace{-2mm}
	  \begin{align*}
	   &\frac{\partial}{\partial x_i}\gamma_i(x,t) = \frac{1}{p(x,\Theta^{(1:k)}_t)^2}\left(\frac{\partial}{\partial x_i}p(x,\Theta^{(1:k)}_t)\right)  \sum_{\ell=1}^k \nu_\ell p_\ell(x,\Theta^{(\ell)}_{t})\inprod{\dot{\Theta}^{(\ell)}_{t}}{\mathcal{I}_i^{(\ell)}(x)} \\
	   &\hspace{8mm}- \frac{1}{p(x,\Theta^{(1:k)}_t)}\sum_{\ell=1}^k \nu_\ell \inprod{\dot{\Theta}^{(\ell)}_{t}}{\varphi_i^{(\ell)}(x,\Theta^{(\ell)}_{t})\exp[\inprod{\Theta^{(\ell)}_{t}}{c^{(\ell)}(x)}-\psi_\ell(\Theta^{(\ell)}_{t})]},
	  \end{align*}
	  where we used
	  \begin{align*}
	   &p_\ell(x,\Theta^{(\ell)}_{t}) \inprod{\dot{\Theta}^{(\ell)}_{t}}{\mathcal{I}_i^{(\ell)}} \\
%	   = \inprod{\dot{\Theta}^{(\ell)}_{t}}{p_\ell(x,\Theta^{(\ell)}_{t})\mathcal{I}_i^{(\ell)}}\\
%	   &\hspace{8mm}=\bigg\langle \dot{\Theta}^{(\ell)}_{t}, \exp[\inprod{\Theta^{(\ell)}_{t}}{c^{(\ell)}(x)}-\psi_\ell(\Theta^{(\ell)}_{t})]  \\
%	    &\hspace{16mm}\int_{-\infty}^{x_i}\varphi_i^{(\ell)}((x_{-i},\xi_i),\Theta^{(\ell)}_{t})\exp[\inprod{\Theta^{(\ell)}_{t}}{c^{(\ell)}(x_{-i},\xi_i)-c^{(\ell)}(x)}]\drv \xi_i\bigg\rangle \\
	    &\hspace{8mm}=\inprod{\dot{\Theta}^{(\ell)}_{t}}{\int_{-\infty}^{x_i}\varphi_i^{(\ell)}((x_{-i},\xi_i),\Theta^{(\ell)}_{t})\exp[\inprod{\Theta^{(\ell)}_{t}}{c^{(\ell)}(x_{-i},\xi_i)}-\psi_\ell(\Theta^{(\ell)}_{t})]\drv \xi_i}.
	  \end{align*}
	  Therefore, \vspace{-2mm}
	  \begin{align*}
	    \frac{\partial}{\partial x_i}\bigg( \gamma_i(x,t)p(x,\Theta^{(1:k)}_t) \bigg) &= \left( \frac{\partial}{\partial x_i}\gamma_i(x,t)\right)p(x,\Theta^{(1:k)}_t)  + \gamma_i(x,t) \left( \frac{\partial}{\partial x_i}p(x,\Theta^{(1:k)}_t)\right) \\
%	   &=\frac{1}{p(x,\Theta^{(1:k)}_t)}\left( \sum_{\ell=1}^k \nu_k \inprod{\Theta^{(\ell)}_{t}}{\frac{\partial c^{(\ell)}(x)}{\partial x_i}}p_\ell(x,\Theta^{(\ell)}_{t}) \right)\left(\sum_{\ell=1}^k \nu_\ell p_\ell(x,\Theta^{(\ell)}_{t}) \inprod{\dot{\Theta}^{(\ell)}_{t}}{\mathcal{I}_i^{(\ell)}(x)} \right)\\
%	   &\hspace{2mm}-\sum_{\ell=1}^k \nu_\ell \inprod{\dot{\Theta}^{(\ell)}_{t}}{\varphi_i^{(\ell)}(x,\Theta^{(\ell)}_{t})}\exp[\inprod{\Theta^{(\ell)}_{t}}{c^{(\ell)}(x)}-\psi_\ell(\Theta^{(\ell)}_{t})]\\
%	   &\hspace{2mm}-\frac{1}{p(x,\Theta^{(1:k)}_t)}\left( \sum_{\ell=1}^k \nu_k \inprod{\Theta^{(\ell)}_{t}}{\frac{\partial c^{(\ell)}(x)}{\partial x_i}}p_\ell(x,\Theta^{(\ell)}_{t}) \right)\left(\sum_{\ell=1}^k \nu_\ell p_\ell(x,\Theta^{(\ell)}_{t}) \inprod{\dot{\Theta}^{(\ell)}_{t}}{\mathcal{I}_i^{(\ell)}(x)} \right)\\
	    &=-\sum_{\ell=1}^k \nu_\ell \inprod{\dot{\Theta}^{(\ell)}_{t}}{\varphi_i^{(\ell)}(x,\Theta^{(\ell)}_{t})}p_\ell(x,\Theta^{(\ell)}_{t}), 
	  \end{align*}
	  and \vspace{-2mm}
	  \begin{align*}
	   &-\sum_{i=1}^n \frac{\partial}{\partial x_i}\bigg( \gamma_i(x,t)p(x,\Theta^{(1:k)}_t) \bigg) 
	   %&= \sum_{i=1}^n \sum_{\ell=1}^k \nu_\ell \inprod{\dot{\Theta}^{(\ell)}_{t}}{\varphi_i^{(\ell)}(x,\Theta^{(\ell)}_{t})}p_\ell(x,\Theta^{(\ell)}_{t}) \\
	   = \sum_{\ell=1}^k \nu_\ell p_\ell(x,\Theta^{(\ell)}_{t}) \left( \sum_{i=1}^n \inprod{\dot{\Theta}^{(\ell)}_{t}}{\varphi_i^{(\ell)}(x,\Theta^{(\ell)}_{t})} \right)\\
	   &\qquad = \sum_{\ell=1}^k \nu_\ell p_\ell(x,\Theta^{(\ell)}_{t}) \inprod{\dot{\Theta}^{(\ell)}_{t}}{c^{(\ell)}(x)-\nabla_{\theta} \psi_\ell(\Theta^{(\ell)}_{t})}
	   = \frac{\partial}{\partial t}p(x,\Theta^{(1:k)}_t),
	  \end{align*}
	  where we used \eqref{eq:cond_varphi_thm_multi_brigo}. \qquad\endproof \\
	 The two claims imply \eqref{e:app:proof:FPE} and hence complete the proof.\qquad\endproof 

%%%%%%%%%%%%%%%%%%%%%%%%%%%%%%%%%%%%%%%%%%%%%%%%%%%%%%%%
	  %%%%%%%%%%%%%%%%%%%%%%%%%%%%%%%%%%%%%%%%%%%%%%%%%%%%%%%%
	  \section{Proof of Proposition~\ref{prop:drift:fct}} \label{app:proof_drift_fct}
	  The proof basically requires Theorem~\ref{thm:brigo_mixture_multi} and two additional lemmas. We first propose in Lemma~\ref{lemma:choose_varphi} a choice of functions $\varphi_i^{(\ell)}$ that satisfy \eqref{eq:cond_varphi_thm_multi_brigo} in Theorem~\ref{thm:brigo_mixture_multi}. Then we show in a second step, in Lemma~\ref{lemma:integrals_multi_gauss}, that for this choice the integral terms \eqref{e:thm:integral} admit a closed form expression.
	   We start with a few preparatory results that are needed to prove Proposition~\ref{prop:drift:fct}.
	   Note that $\nabla_\Theta \psi(\Theta)=(\nabla_{\eta} \psi(\Theta),\nabla_{\theta} \psi(\Theta))\in \Hi$ and recall that according to \cite[p.631]{ref:Ber-09} for $A\in\R^{n\times m}$, $B\in\R^{m \times n}$ and $X\in \text{GL}(n,\R)$
	   $\frac{\drv}{\drv X}\tr \left( AX^{-1}B \right) =-X^{-1}BAX^{-1}$ and $\frac{\drv}{\drv X}\log \det \left( AX^{-1}B \right) =  -X^{-1}B(AX^{-1}B)^{-1}AX^{-1}$
	  and therefore
	  \begin{align}
	   \nabla_{\eta} \psi(\Theta) = -\frac{1}{2}\theta^{-1}\eta, \qquad   \nabla_{\theta} \psi(\Theta) = \theta^{-1}\left( \frac{1}{4}\eta \eta\transp \theta^{-1}-\frac{1}{2} \In \right). \label{eq:derivative:psi:theta}
	  \end{align}

%Lemma1
 \begin{lemma} \label{lemma:choose_varphi}
	 Consider the functions $\varphi_i^{(\ell)}: \R^n\times \Hi \to \Hi$, where $\varphi_i^{(\ell)}=\left(\varphi_{1,i}^{(\ell)},\varphi_{2,i}^{(\ell)}\right)$ with $\varphi_{1,i}^{(\ell)}:\R^n\times \Hi \to \R^n$ and $\varphi_{2,i}^{(\ell)}:\R^n\times \Hi \to \R^{n\times n}$ for $\ell=1,\hdots,k$. Let $\varphi_{1,i}^{(\ell)}$ and $\varphi_{2,i}^{(\ell)}$ be defined as
	 \begin{align*}
	   \varphi_{1,i}^{(\ell)}((x_{-i},\xi_i),\Theta^{(\ell)}_{t}) &:= {\theta^{(\ell)}_{t}}^{-1}E_i\theta^{(\ell)}_{t} ( c^{(\ell)}_{1}(x_{-i},\xi_i)-\nabla_{\eta} \psi_\ell(\Theta^{(\ell)}_{t}) )\\
	   \varphi_{2,i}^{(\ell)}((x_{-i},\xi_i),\Theta^{(\ell)}_{t}) &:= {\theta^{(\ell)}_{t}}^{-1}E_i \bigg( \theta^{(\ell)}_{t} (c^{(\ell)}_{2}(x_{-i},\xi_i)-\nabla_{\theta} \psi_\ell(\Theta^{(\ell)}_{t}) ) \Theta^{(\ell)}_{t} \\ 
	   &\hspace{24mm}- \frac{1}{2}\theta^{(\ell)}_{t} (x_{-i},\xi_i){\eta^{(\ell)}_{t}}\transp+ \frac{1}{2}\eta^{(\ell)}_{t} (x_{-i},\xi_i)\transp \theta^{(\ell)}_{t} \bigg) {\theta^{(\ell)}_{t}}^{-1}. 
	  \end{align*}
	 	 Then, \eqref{eq:cond_varphi_thm_multi_brigo} holds for all $\ell=1,\hdots,k$.
%	 \begin{equation*}
%	    \sum_{i=1}^n \left. \inprod{\dot{\Theta}^{(\ell)}_{t}}{\varphi_i^{(\ell)}\left((x_{-i},\xi_i),\Theta^{(\ell)}_{t}\right)}\right|_{\xi_i=x_i} = \inprod{\dot{\Theta}^{(\ell)}_{t}}{c^{(\ell)}(x)-\nabla_{\theta} \psi_\ell(\Theta^{(\ell)}_{t})}.
%	 \end{equation*}
	 \end{lemma}

%Proof of Lemma1
 {\em Proof of Lemma~\ref{lemma:choose_varphi}.}
	  According to \eqref{eq:multivar:normal:parameters} we have
	  \begin{equation*}
	   \sum_{i=1}^n \inprod{\dot{\Theta}^{(\ell)}_{t}}{\varphi_i^{(\ell)}(x,\Theta^{(\ell)}_{t}) }= \sum_{i=1}^n \left( \inprod{\dot{\eta}_{t}^{(\ell)} {}}{\varphi_{1,i}^{(\ell)}(x,\Theta^{(\ell)}_{t})} + \inprod{\dot{\theta}^{(\ell)}_{t}}{\varphi_{2,i}^{(\ell)}(x,\Theta^{(\ell)}_{t})} \right),  \vspace{-2mm}
	  \end{equation*}
	  consisting of the two components  \vspace{-2mm}
	  \begin{align*}
	   &\sum_{i=1}^n  \inprod{\dot{\eta}_{t}^{(\ell)} {}}{\varphi_{1,i}^{(\ell)}(x,\Theta^{(\ell)}_{t})} = \sum_{i=1}^n  \inprod{\dot{\eta}_{t}^{(\ell)} {}}{{\theta^{(\ell)}_{t}}^{-1}E_i\theta^{(\ell)}_{t} ( c_{1}^{(\ell)}(x)-\nabla_{\eta} \psi_\ell(\Theta^{(\ell)}_{t}) )}\\
	    &\qquad = \inprod{\dot{\eta}_{t}^{(\ell)} {}}{\sum_{i=1}^n {\theta^{(\ell)}_{t}}^{-1}E_i\theta^{(\ell)}_{t} ( c_{1}^{(\ell)}(x)-\nabla_{\eta} \psi_\ell(\Theta^{(\ell)}_{t}) )} 
	    = \inprod{\dot{\eta}_{t}^{(\ell)} {}}{c_{1}^{(\ell)}(x)-\nabla_{\eta} \psi_\ell(\Theta^{(\ell)}_{t})}
	  \end{align*}
	  and
	  \begin{align*}
	    &\sum_{i=1}^n  \inprod{\dot{\theta}^{(\ell)}_{t}}{\varphi_{2,i}^{(\ell)}(x,\Theta^{(\ell)}_{t})}
	   % &\hspace{8mm}= \sum_{i=1}^n \inprod{\dot{\theta}^{(\ell)}_{t}}{ {\theta^{(\ell)}_{t}}^{-1}E_i \bigg( \theta^{(\ell)}_{t} (c_{2}^{(\ell)}(x)-\nabla_{\theta} \psi_\ell(\Theta^{(\ell)}_{t}) ) \theta^{(\ell)}_{t} - \frac{1}{2}\theta^{(\ell)}_{t} x{\eta^{(\ell)}_{t}}\transp + \frac{1}{2}{\eta^{(\ell)}_{t}} x\transp \theta^{(\ell)}_{t} \bigg) {\theta^{(\ell)}_{t}}^{-1}} \\
	   %   &\hspace{8mm}= \sum_{i=1}^n \inprod{\dot{\theta}^{(\ell)}_{t}}{ {\theta^{(\ell)}_{t}}^{-1}E_i  \theta^{(\ell)}_{t} (c_{2}^{(\ell)}(x)-\nabla_{\theta} \psi_\ell(\Theta^{(\ell)}_{t}) )  - \frac{1}{2}{\theta^{(\ell)}_{t}}^{-1}E_i\theta^{(\ell)}_{t} x{\eta^{(\ell)}_{t}}\transp{\theta^{(\ell)}_{t}}^{-1} + \frac{1}{2}{\theta^{(\ell)}_{t}}^{-1}E_i{\eta^{(\ell)}_{t}} x\transp} \\
= \bigg\langle \dot{\theta}^{(\ell)}_{t}, \sum_{i=1}^n \bigg( {\theta^{(\ell)}_{t}}^{-1}E_i  \theta^{(\ell)}_{t} (c_{2}^{(\ell)}(x)-\nabla_{\theta} \psi_\ell(\Theta^{(\ell)}_{t}) ) \\
	     &\hspace{16mm} - \frac{1}{2}{\theta^{(\ell)}_{t}}^{-1}E_i\theta^{(\ell)}_{t} x{\eta^{(\ell)}_{t}}\transp{\theta^{(\ell)}_{t}}^{-1} + \frac{1}{2}{\theta^{(\ell)}_{t}}^{-1}E_i{\eta^{(\ell)}_{t}} x\transp\bigg)\bigg\rangle \\
	     &\hspace{8mm}=\inprod{\dot{\theta}^{(\ell)}_{t}}{c_{2}^{(\ell)}(x)-\nabla_{\theta} \psi_\ell(\Theta^{(\ell)}_{t})}-\frac{1}{2}\inprod{\dot{\theta}^{(\ell)}_{t}}{x{\eta^{(\ell)}_{t}}\transp {\theta^{(\ell)}_{t}}^{-1}} + \frac{1}{2}\inprod{\dot{\theta}^{(\ell)}_{t}}{{\theta^{(\ell)}_{t}}^{-1}{\eta^{(\ell)}_{t}} x\transp}\\
	     &\hspace{8mm}=\inprod{\dot{\theta}^{(\ell)}_{t}}{c_{2}^{(\ell)}(x)-\nabla_{\theta} \psi_\ell(\Theta^{(\ell)}_{t})},
	  \end{align*}
	  where  we have used in the last step that
	     	     $\inprod{\dot{\theta}^{(\ell)}_{t}}{x{\eta^{(\ell)}_{t}}\transp {\theta^{(\ell)}_{t}}^{-1}} =\inprod{\dot{\theta}^{(\ell)}_{t}}{{\theta^{(\ell)}_{t}}^{-1}{\eta^{(\ell)}_{t}} x\transp}$,
		    % &=\tr\left(\dot{\theta}^{(\ell)}_{t}(x{\eta^{(\ell)}_{t}}\transp {\theta^{(\ell)}_{t}}^{-1})\transp\right)=\tr\left(\dot{\theta}^{(\ell)}_{t} x{\eta^{(\ell)}_{t}}\transp {\theta^{(\ell)}_{t}}^{-1}\right) \\
		    % &=\tr\left(\dot{\theta}^{(\ell)}_{t}({\theta^{(\ell)}_{t}}^{-1}{\eta^{(\ell)}_{t}} x\transp)\transp\right)
	    since for $A\in\sym(n,\R)$ and $B\in \R^{n\times n}$
	    $ \tr(AB\transp)=\tr(AB).$ 
	 \qquad\endproof
%Lemma 2
	  \begin{lemma} \label{lemma:integrals_multi_gauss}
	   For $i=1,\hdots,n,$ $j=1,2$ and $\ell=1,\hdots,k$ consider
	    \begin{equation*}
	   \mathcal{I}_{j,i}^{(\ell)}(s_i,x):=\int_{-\infty}^{s_i} \varphi_{j,i}^{(\ell)}( (x_{-i},\xi_i),\Theta^{(\ell)}_{t} )\exp\left( \inprod{\Theta^{(\ell)}_{t}}{c^{(\ell)}(x_{-i},\xi_i)-c^{(\ell)}(x)} \right)\drv \xi_i,
	  \end{equation*}
	  where the functions $\varphi_{j,i}^{(\ell)}$ are chosen according to Lemma \ref{lemma:choose_varphi}. Then,
	  	   \begin{align*}
	    \mathcal{I}_{1,i}^{(\ell)}(s_i,x) &= \frac{1}{2}{\theta^{(\ell)}_{t}}^{-1} e_i\exp\left( \inprod{\Theta^{(\ell)}_{t}}{c^{(\ell)}(x_{-i},s_i)-c^{(\ell)}(x)} \right) \\
	    \mathcal{I}_{2,i}^{(\ell)}(x_i,x) &= \frac{1}{4}{\theta^{(\ell)}_{t}}^{-1}e_i(2\theta^{(\ell)}_{t} x-\eta^{(\ell)}_{t})\transp{\theta^{(\ell)}_{t}}^{-1}.
	   \end{align*}
	  \end{lemma}
 Note that $\mathcal{I}_{i}^{(\ell)}(x)=(\mathcal{I}_{1,i}^{(\ell)}(x_i,x),\mathcal{I}_{2,i}^{(\ell)}(x_i,x))$, where $\mathcal{I}_{i}^{(\ell)}(x)$ is the function defined in \eqref{e:thm:integral} and $e_{i}$ denote the canonical basis vectors of $R^{n}$.\\
%Proof of Lemma2
	  {\em Proof of Lemma \ref{lemma:integrals_multi_gauss}.}
	  \begin{align*}
	    &\mathcal{I}_{1,i}^{(\ell)}(s_i,x) = \int_{-\infty}^{s_i} \varphi_{1,i}^{(\ell)}( (x_{-i},\xi_i),\Theta^{(\ell)}_{t} )\exp\left( \inprod{\Theta^{(\ell)}_{t}}{c^{(\ell)}(x_{-i},\xi_i)-c^{(\ell)}(x)} \right)\drv \xi_i \\
	    % &= \int_{-\infty}^{s_i} {\theta^{(\ell)}_{t}}^{-1}E_i\theta^{(\ell)}_{t} ( c_{1}^{(\ell)}(x_{-i},\xi_i)-\nabla_{\eta} \psi_\ell(\Theta^{(\ell)}_{t}) )\exp\left( \inprod{\Theta^{(\ell)}_{t}}{c^{(\ell)}(x_{-i},\xi_i)-c^{(\ell)}(x)} \right)\drv \xi_i \\
	     &= \frac{1}{2}{\theta^{(\ell)}_{t}}^{-1} \int_{-\infty}^{s_i} E_i2\theta^{(\ell)}_{t} ( c_{1}^{(\ell)}(x_{-i},\xi_i)-\nabla_{\eta} \psi_\ell(\Theta^{(\ell)}_{t}) )\exp\left( \inprod{\Theta^{(\ell)}_{t}}{c^{(\ell)}(x_{-i},\xi_i)-c^{(\ell)}(x)} \right)\drv \xi_i \\
	     &= \frac{1}{2}{\theta^{(\ell)}_{t}}^{-1}  \int_{-\infty}^{s_i}E_i2\theta^{(\ell)}_{t} \left( (x_{-i},\xi_i)+\frac{1}{2}{\theta^{(\ell)}_{t}}^{-1}{\eta^{(\ell)}_{t}} \right)\exp\left( \inprod{\Theta^{(\ell)}_{t}}{c^{(\ell)}(x_{-i},\xi_i)-c^{(\ell)}(x)} \right)\drv \xi_i \\
		&=\frac{1}{2}{\theta^{(\ell)}_{t}}^{-1}  \int_{-\infty}^{s_i}E_i\left( 2\theta^{(\ell)}_{t}  (x_{-i},\xi_i)+{\eta^{(\ell)}_{t}} \right)\exp\left( \inprod{\Theta^{(\ell)}_{t}}{c^{(\ell)}(x_{-i},\xi_i)-c^{(\ell)}(x)} \right)\drv \xi_i,
	  \end{align*}
	  where \eqref{eq:derivative:psi:theta} was used. Consider the substitution $z:= \inprod{\Theta^{(\ell)}_{t}}{c^{(\ell)}(x_{-i},\xi_i)-c^{(\ell)}(x)}$
	  that leads to
	   \begin{align*}
	   &\mathcal{I}_{1,i}^{(\ell)}(s_i,x)  
	        %   &=  \frac{1}{2}{\theta^{(\ell)}_{t}}^{-1}e_i \int_{-\infty}^{s_i}e_i\transp\left( 2\theta^{(\ell)}_{t}  (x_{-i},\xi_i)+{\eta^{(\ell)}_{t}} \right)\exp\left( \inprod{\Theta^{(\ell)}_{t}}{c^{(\ell)}(x_{-i},\xi_i)-c^{(\ell)}(x)} \right)\drv \xi_i \\
	     %      &=  \frac{1}{2}{\theta^{(\ell)}_{t}}^{-1}e_i \int_{-\infty}^{\inprod{\Theta^{(\ell)}_{t}}{c^{(\ell)}(x_{-i},s_i)-c^{(\ell)}(x)}} \exp(z)\drv z\\
	           = \frac{1}{2}{\theta^{(\ell)}_{t}}^{-1} e_i\exp\left( \inprod{\Theta^{(\ell)}_{t}}{c^{(\ell)}(x_{-i},s_i)-c^{(\ell)}(x)} \right),
	   \end{align*}
	  where we used that $\theta^{(\ell)}_{t} \prec 0$, since $\theta^{(\ell)}_{t}=-\frac{1}{2}{S^{(\ell)}_{t}}^{-1}$ and the invese of a negative definite matrix is negative definite.
	  For the second integral term
	   \begin{align*}
	    \mathcal{I}_{2,i}^{(\ell)}(x) &= \int_{-\infty}^{x_i} \varphi_{2,i}^{(\ell)}( (x_{-i},\xi_i),\Theta^{(\ell)}_{t} )\exp\left( \inprod{\Theta^{(\ell)}_{t}}{c^{(\ell)}(x_{-i},\xi_i)-c^{(\ell)}(x)} \right)\drv \xi_i \\
	 %   &=  \int_{-\infty}^{x_i}{\theta^{(\ell)}_{t}}^{-1}E_i \bigg( \theta^{(\ell)}_{t} (c_{2}^{(\ell)}(x_{-i},\xi_i)-\nabla_{\theta} \psi_\ell(\Theta^{(\ell)}_{t}) ) \theta^{(\ell)}_{t} - \frac{1}{2}\theta^{(\ell)}_{t} (x_{-i},\xi_i){\eta^{(\ell)}_{t}}\transp  \\
	  %  &\hspace{8mm} + \frac{1}{2}{\eta^{(\ell)}_{t}} (x_{-i},\xi_i)\transp \theta^{(\ell)}_{t} \bigg) {\theta^{(\ell)}_{t}}^{-1} \exp\left( \inprod{\Theta^{(\ell)}_{t}}{c^{(\ell)}(x_{-i},\xi_i)-c^{(\ell)}(x)} \right)\drv \xi_i \\
	    &= \frac{1}{4}{\theta^{(\ell)}_{t}}^{-1} \int_{-\infty}^{x_i}E_i \bigg( 4\theta^{(\ell)}_{t} (c_{2}^{(\ell)}(x_{-i},\xi_i)-\nabla_{\theta} \psi_\ell(\Theta^{(\ell)}_{t}) ) \theta^{(\ell)}_{t} - 2\theta^{(\ell)}_{t} (x_{-i},\xi_i){\eta^{(\ell)}_{t}}\transp \\
	    &\hspace{8mm} + 2{\eta^{(\ell)}_{t}} (x_{-i},\xi_i)\transp \theta^{(\ell)}_{t} \bigg) {\theta^{(\ell)}_{t}}^{-1} \exp\left( \inprod{\Theta^{(\ell)}_{t}}{c^{(\ell)}(x_{-i},\xi_i)-c^{(\ell)}(x)} \right)\drv \xi_i \\
	 %   &= \frac{1}{4}{\theta^{(\ell)}_{t}}^{-1} \int_{-\infty}^{x_i}E_i \bigg( 4\theta^{(\ell)}_{t} \left((x_{-i},\xi_i)(x_{-i},\xi_i)\transp -{\theta^{(\ell)}_{t}}^{-1}\left( \frac{1}{4}{\eta^{(\ell)}_{t}}{\eta^{(\ell)}_{t}}\transp {\theta^{(\ell)}_{t}}^{-1}-\frac{1}{2} I_n\right) \right) \theta^{(\ell)}_{t}\\ 
	   % &\hspace{0mm}- 2\theta^{(\ell)}_{t} (x_{-i},\xi_i){\eta^{(\ell)}_{t}}\transp 
	   % + 2{\eta^{(\ell)}_{t}} (x_{-i},\xi_i)\transp \theta^{(\ell)}_{t} \bigg) {\theta^{(\ell)}_{t}}^{-1} \exp\left( \inprod{\Theta^{(\ell)}_{t}}{c^{(\ell)}(x_{-i},\xi_i)-c^{(\ell)}(x)} \right)\drv \xi_i \\
	    &=\frac{1}{4}{\theta^{(\ell)}_{t}}^{-1} \int_{-\infty}^{x_i}E_i\bigg( 2\theta^{(\ell)}_{t}  (x_{-i},\xi_i)(x_{-i},\xi_i)\transp 2\theta^{(\ell)}_{t}  -{\eta^{(\ell)}_{t}}{\eta^{(\ell)}_{t}}\transp +2\theta^{(\ell)}_{t} \\ 
	         &\hspace{0mm} - 2\theta^{(\ell)}_{t} (x_{-i},\xi_i) {\eta^{(\ell)}_{t}}\transp + 2 {\eta^{(\ell)}_{t}} (x_{-i},\xi_i)\transp \theta^{(\ell)}_{t} \bigg) {\theta^{(\ell)}_{t}}^{-1} \exp\left( \inprod{\Theta^{(\ell)}_{t}}{c^{(\ell)}(x_{-i},\xi_i)-c^{(\ell)}(x)} \right)\drv \xi_i,
	   \end{align*}
	   where we have used \eqref{eq:derivative:psi:theta}.
	  By expanding terms and using integration by parts together with the first assertion of this lemma
	  \begin{align*}
	    \mathcal{I}_{2,i}^{(\ell)}(x)   
	    %&= \frac{1}{4}{\theta^{(\ell)}_{t}}^{-1}\int_{-\infty}^{x_i} E_i \bigg(  (2\theta^{(\ell)}_{t}(x_{-i},\xi_i)+{\eta^{(\ell)}_{t}})(2\theta^{(\ell)}_{t}(x_{-i},\xi_i)-{\eta^{(\ell)}_{t}})\transp + 2\theta^{(\ell)}_{t} \bigg) {\theta^{(\ell)}_{t}}^{-1}\\ 
	 %   &\hspace{8mm} \exp\left( \inprod{\Theta^{(\ell)}_{t}}{c^{(\ell)}(x_{-i},\xi_i)-c^{(\ell)}(x)} \right)\drv \xi_i\\
	      &=\frac{1}{2} \int_{-\infty}^{x_i} {\frac{1}{2}{\theta^{(\ell)}_{t}}^{-1}E_i  (2\theta^{(\ell)}_{t}(x_{-i},\xi_i)+{\eta^{(\ell)}_{t}})}{(2\theta^{(\ell)}_{t}(x_{-i},\xi_i)-{\eta^{(\ell)}_{t}})\transp {\theta^{(\ell)}_{t}}^{-1}} \\
	      &\hspace{8mm} \exp\left( \inprod{\Theta^{(\ell)}_{t}}{c^{(\ell)}(x_{-i},\xi_i)-c^{(\ell)}(x)} \right)\drv \xi_i\\
	      &\hspace{8mm}+ \frac{1}{2}\int_{-\infty}^{x_i}{\theta^{(\ell)}_{t}}^{-1}E_i \exp\left( \inprod{\Theta^{(\ell)}_{t}}{c^{(\ell)}(x_{-i},\xi_i)-c^{(\ell)}(x)} \right)\drv \xi_i\\
	      &=\frac{1}{2}\mathcal{I}_{1,i}^{(\ell)}(x_i,x) (2\theta^{(\ell)}_{t} x-{\eta^{(\ell)}_{t}})\transp {\theta^{(\ell)}_{t}}^{-1} - \frac{1}{2}\int_{-\infty}^{x_i}\mathcal{I}_{1,i}^{(\ell)}(\xi_i,x)(2\theta^{(\ell)}_{t} e_i)\transp {\theta^{(\ell)}_{t}}^{-1} \drv \xi_i\\
	      &\hspace{8mm}+\frac{1}{2}\int_{-\infty}^{x_i}{\theta^{(\ell)}_{t}}^{-1}E_i \exp\left( \inprod{\Theta^{(\ell)}_{t}}{c^{(\ell)}(x_{-i},\xi_i)-c^{(\ell)}(x)} \right)\drv \xi_i\\     
	      &=\frac{1}{4}{\theta^{(\ell)}_{t}}^{-1}e_i (2\theta^{(\ell)}_{t} x-{\eta^{(\ell)}_{t}})\transp {\theta^{(\ell)}_{t}}^{-1} \\
	      &\hspace{8mm}- \frac{1}{2}\int_{-\infty}^{x_i}\frac{1}{2}{\theta^{(\ell)}_{t}}^{-1}e_i \exp\left( \inprod{\Theta^{(\ell)}_{t}}{c^{(\ell)}(x_{-i},\xi_i)-c^{(\ell)}(x)} \right) e_i\transp 2\theta^{(\ell)}_{t} {\theta^{(\ell)}_{t}}^{-1} \drv \xi_i \\
	      &\hspace{8mm}+\frac{1}{2}\int_{-\infty}^{x_i}{\theta^{(\ell)}_{t}}^{-1}E_i \exp\left( \inprod{\Theta^{(\ell)}_{t}}{c^{(\ell)}(x_{-i},\xi_i)-c^{(\ell)}(x)} \right)\drv \xi_i\\ 
	%      &=\frac{1}{4}{\theta^{(\ell)}_{t}}^{-1}e_i (2\theta^{(\ell)}_{t} x-{\eta^{(\ell)}_{t}})\transp {\theta^{(\ell)}_{t}}^{-1} \\
	 %    &\hspace{8mm} - \frac{1}{2}\int_{-\infty}^{x_i}{\theta^{(\ell)}_{t}}^{-1}E_i \exp\left( \inprod{\Theta^{(\ell)}_{t}}{c^{(\ell)}(x_{-i},\xi_i)-c^{(\ell)}(x)} \right)\drv \xi_i\\
	 %     &\hspace{8mm}+\frac{1}{2}\int_{-\infty}^{x_i}{\theta^{(\ell)}_{t}}^{-1}E_i \exp\left( \inprod{\Theta^{(\ell)}_{t}}{c^{(\ell)}(x_{-i},\xi_i)-c^{(\ell)}(x)} \right)\drv \xi_i\\ 
	      &=\frac{1}{4}{\theta^{(\ell)}_{t}}^{-1}e_i (2\theta^{(\ell)}_{t} x-{\eta^{(\ell)}_{t}})\transp {\theta^{(\ell)}_{t}}^{-1}.\qquad\endproof
	  \end{align*}

% Proof of Proposition	  
	 {\em Proof of Proposition~\ref{prop:drift:fct}}
	   We decompose the function $u_i(x,t)$, given by Theorem \ref{thm:brigo_mixture_multi} into $u_i(x,t)=g_i(x,t)+\gamma_i(x,t)$ for all $i=1,\hdots,n,$ where
	   \begin{align*} 
	    g_i(x,t):=&\frac{1}{2}\sum_{j=1}^n\frac{\partial}{\partial x_j}a_{ij}(x) + \frac{1}{2}\sum_{j=1}^n a_{ij}(x) \frac{\frac{\partial}{\partial x_j}p(x,\Theta^{(1)}_t,\hdots,\Theta^{(k)}_t)}{p(x,\Theta^{(1)}_t,\hdots,\Theta^{(k)}_t)} \\
	    \gamma_i(x,t) :=& \frac{-1}{p(x,\Theta^{(1)}_t,\hdots,\Theta^{(k)}_t)} \sum_{\ell=1}^k \nu_\ell p_\ell(x,\Theta^{(\ell)}_{t})\inprod{\dot{\Theta}^{(\ell)}_{t}}{\mathcal{I}_i^{(\ell)}(x)}.
	   \end{align*}
	   As a preliminary step by invoking Lemma \ref{lemma:integrals_multi_gauss}
	 \begin{align*}
	    &\inprod{\dot{\Theta}^{(\ell)}_{t}}{\mathcal{I}_i^{(\ell)}(x)} =  \inprod{\dot{\eta}_{t}^{(\ell)} {}}{\mathcal{I}_{1,i}^{(\ell)}(x)} + \inprod{\dot{\theta}^{(\ell)}_{t} }{\mathcal{I}_{2,i}^{(\ell)}(x)}\\
	   %  &= \inprod{\dot{\eta}_{t}^{(\ell)} {}}{\frac{1}{2}{\theta^{(\ell)}_{t}}^{-1}e_i} + \inprod{\dot{\theta}^{(\ell)}_{t} }{\frac{1}{4}{\theta^{(\ell)}_{t}}^{-1}e_i(2\theta^{(\ell)}_{t} x-\eta^{(\ell)}_{t})\transp {\theta^{(\ell)}_{t}}^{-1}}\\
	      &=\inprod{\dot{\eta}_{t}^{(\ell)} {}}{\frac{1}{2}{\theta^{(\ell)}_{t}}^{-1}e_i}+\inprod{\dot{\theta}^{(\ell)}_{t} }{\frac{1}{2}{\theta^{(\ell)}_{t}}^{-1}e_ix\transp - \frac{1}{4}{\theta^{(\ell)}_{t}}^{-1}e_i{\eta^{(\ell)}_{t}}\transp{\theta^{(\ell)}_{t}}^{-1}} \\
	      &=\frac{1}{2}\dot{\eta}_{t}^{(\ell)} {}\transp({\theta^{(\ell)}_{t}}^{-1}e_i) + \frac{1}{2}\tr\left(\dot{\theta}^{(\ell)}_{t} ({\theta^{(\ell)}_{t}}^{-1}e_ix\transp)\transp\right)-\frac{1}{4}\tr\left(\dot{\theta}^{(\ell)}_{t} ({\theta^{(\ell)}_{t}}^{-1}e_i{\eta^{(\ell)}_{t}}\transp{\theta^{(\ell)}_{t}}^{-1})\transp\right) \\
	   %   &=\frac{1}{2}\dot{\eta}_{t}^{(\ell)} {}\transp({\theta^{(\ell)}_{t}}^{-1}e_i) + \frac{1}{2}\tr\left(\dot{\theta}^{(\ell)}_{t} x e_i\transp {\theta^{(\ell)}_{t}}^{-1}\right)-\frac{1}{4}\tr\left(\dot{\theta}^{(\ell)}_{t} {\theta^{(\ell)}_{t}}^{-1}\eta^{(\ell)}_{t} e_i\transp{\theta^{(\ell)}_{t}}^{-1}\right) \\
	     % &=\frac{1}{2}\dot{\eta}_{t}^{(\ell)} {}\transp({\theta^{(\ell)}_{t}}^{-1}e_i) + \frac{1}{2}\tr\left( {\theta^{(\ell)}_{t}}^{-1}\dot{\theta}^{(\ell)}_{t} x e_i\transp\right)-\frac{1}{4}\tr\left({\theta^{(\ell)}_{t}}^{-1}\dot{\theta}^{(\ell)}_{t} {\theta^{(\ell)}_{t}}^{-1}\eta^{(\ell)}_{t} e_i\transp\right) \\
	      &=\frac{1}{2}\dot{\eta}_{t}^{(\ell)} {}\transp({\theta^{(\ell)}_{t}}^{-1}e_i)+\frac{1}{2}e_i\transp  {\theta^{(\ell)}_{t}}^{-1}\dot{\theta}^{(\ell)}_{t} x -\frac{1}{4}e_i\transp {\theta^{(\ell)}_{t}}^{-1}\dot{\theta}^{(\ell)}_{t} {\theta^{(\ell)}_{t}}^{-1}\eta^{(\ell)}_{t}. 
	    \end{align*}
	    Therefore, \vspace{-2mm}
%	    \begin{equation*}
%	    \begin{aligned}
%	     \gamma_i(x,t) &= -\frac{1}{p(x,\Theta^{(1)}_t,\hdots,\Theta^{(k)}_t)}\sum_{\ell=1}^k \nu_\ell p_\ell(x,\Theta^{(\ell)}_{t})\left( \frac{1}{2}\dot{\eta}\transp {\theta^{(\ell)}_{t}}^{-1}e_i + \frac{1}{2}e_i\transp {\theta^{(\ell)}_{t}}^{-1}\dot{\theta}^{(\ell)}_{t} x\right.\\
%	     &\hspace{8mm} \left.-\frac{1}{4}e_i\transp {\theta^{(\ell)}_{t}}^{-1}\dot{\theta}^{(\ell)}_{t} {\theta^{(\ell)}_{t}}^{-1}\eta_{t}^{(\ell)} \right) 
%	     \end{aligned}
%	    \end{equation*}
%	    and
	     \begin{align*}
	   \gamma(x,t) &= -\frac{1}{p(x,\Theta^{(1)}_t,\hdots,\Theta^{(k)}_t)}\sum_{\ell=1}^k \nu_\ell p_\ell(x,\Theta^{(\ell)}_{t})\\
	   &\hspace{8mm}\left( \frac{1}{2}{\theta^{(\ell)}_{t}}^{-1}\dot{\eta}_{t}^{(\ell)} + \frac{1}{2}{\theta^{(\ell)}_{t}}^{-1}\dot{\theta}^{(\ell)}_{t} x-\frac{1}{4}{\theta^{(\ell)}_{t}}^{-1}\dot{\theta}^{(\ell)}_{t} {\theta^{(\ell)}_{t}}^{-1}\eta_{t}^{(\ell)} \right). 
	  \end{align*}
	  Furthermore, \vspace{-2mm}
	    \begin{align*}
	   &\frac{\partial}{\partial x_j} p(x,\Theta^{(1)}_t,\hdots,\Theta^{(k)}_t) = \sum_{\ell=1}^k \nu_\ell \inprod{\Theta^{(\ell)}_{t}}{\frac{\partial c^{(\ell)}(x)}{\partial x_j}}\exp[\inprod{\Theta^{(\ell)}_{t}}{c^{(\ell)}(x)}-\psi_\ell(\Theta^{(\ell)}_{t})]\\
	   &= \sum_{\ell=1}^k \nu_\ell \inprod{\Theta^{(\ell)}_{t}}{\left( e_j,e_jx\transp + xe_j\transp \right)}\exp\left(\inprod{\Theta^{(\ell)}_{t}}{c^{(\ell)}(x)}-\psi_\ell(\Theta^{(\ell)}_{t})\right)\\
	     &= \sum_{\ell=1}^k \nu_\ell \left({\eta^{(\ell)}_{t}}\transp e_j + 2e_j\transp \theta^{(\ell)}_{t} x\right)\exp\left(\inprod{\Theta^{(\ell)}_{t}}{c^{(\ell)}(x)}-\psi_\ell(\Theta^{(\ell)}_{t})\right),
	  \end{align*}
	  and therefore \vspace{-2mm}
%	  \begin{align*}
%	   g_i(x,t)  
%	   %&=\frac{1}{2}\sum_{j=1}^n \frac{\partial}{\partial x_j}a_{ij}(x)\\ 
%	   %&+ \frac{1}{2}\sum_{j=1}^na_{ij}(x) \frac{\sum_{\ell=1}^k \nu_\ell \left(\eta^{(\ell)}_{t} e_j + 2e_j\transp \theta^{(\ell)}_{t} x\right)\exp\left(\inprod{\Theta^{(\ell)}_{t}}{c^{(\ell)}(x)}-\psi_\ell(\Theta^{(\ell)}_{t})\right)}{p(x,\Theta^{(1)}_t,\hdots,\Theta^{(k)}_t)} \\
%	   &= \frac{1}{2}\diverg{a_{i\cdot}}+ \frac{1}{2}\sum_{j=1}^na_{ij}(x) \frac{\sum_{\ell=1}^k \nu_\ell p_\ell(x,\Theta^{(\ell)}_{t}) \left(\eta^{(\ell)}_{t} e_j + 2e_j\transp \theta^{(\ell)}_{t} x\right)}{p(x,\Theta^{(1)}_t,\hdots,\Theta^{(k)}_t)} 
%	  \end{align*}
%	  leading to
	  \begin{align*}
	   g(x,t) &= \frac{1}{2}\diverg{a(x)} + \frac{1}{2} \frac{\sum_{\ell=1}^k \nu_\ell p_\ell(x,\Theta^{(\ell)}_{t}) a(x) \left(\eta^{(\ell)}_{t}+ 2\theta^{(\ell)}_{t} x\right)}{p(x,\Theta^{(1)}_t,\hdots,\Theta^{(k)}_t)}.
	   \end{align*} 
Note that our choice of $\varphi_i^{(\ell)}$ satisfy \eqref{eq:cond_varphi_thm_multi_brigo} as shown in Lemma~\ref{lemma:choose_varphi}, which then completes the proof.
	 \qquad\endproof

%%%%%%%%%%%%%%%%%%%%%%%%%%%%%%%%%%%%%%%%%%%%%%%%%%%%%%%%%%%%%%%%%%%%%%%%%%%%%%%%%%%%%%%%%%%%%%%%%%%%%%%%%%%%%%%%%%%%%%%%%%%%%%%%%%%%%%%%%%%%%%%%%%%%%%%%%%%%%%%%%%%%%%%%%%%%%%%%%%%%%%%%%%%%%
	  \section{Proof of Theorem~\ref{thm:mean_variance} } \label{app:proof_thm_mS}
	  \begin{lemma} \label{lemma:eq-for-variance}
	  For an SDE of the form \eqref{eq:approx_SDE} the mean $m_{t}$ and covariance matrix $S_{t}$ of $X_{t}$ satisfy
	 \begin{equation}
	  \begin{aligned}
	  \drv m_{t} &= \Expec{u(X_{t},t)}\drv t, \\
	 \drv S_{t} &= \left( \Expec{X_{t}u(X_{t},t)\transp}+\Expec{u(X_{t},t)X_{t}\transp}+\Expec{\sigma(X_{t})\sigma(X_{t})\transp}\right. \\
		   &\hspace{8mm} \left. -m_{t}\Expec{u(X_{t},t)}\transp-\Expec{u(X_{t},t)}m_{t}\transp \right) \drv t.
	  \end{aligned}
	 \end{equation}
	\end{lemma}
	\indent\textit{Proof.}
	 The equation for the mean is trivial. For the variance let $Y_{t}:=X_{t}X_{t}\transp$. According to It\^o's Lemma \cite{ref:Oksendal-03}
	   $\drv Y_{t} = X_{t} (u(X_{t},t)\drv t + \sigma(X_{t}) \drv B_{t})\transp + (u(X_{t},t)\drv t + \sigma(X_{t}) \drv B_{t})X_{t}\transp +\sigma(X_{t})\sigma(X_{t})\transp \drv t,$
	 and similarly
	   $\drv m^2_{t} = \left( m_{t}\Expec{u(X_{t},t)}\transp +\Expec{u(X_{t},t)}m_{t}\transp \right)\drv t.$
	 Hence,
	 \begin{equation*}
	  \begin{aligned}
	   \drv S(t) &= \Expec{\drv Y(t)}-\drv m_{t}^2 = \left( \Expec{X_{t}u(X_{t},t)\transp}+\Expec{u(X_{t},t)X_{t}\transp}+\Expec{\sigma(X_{t})\sigma(X_{t})\transp}\right. \\
		   &\hspace{38mm} \left. -m_{t}\Expec{u(X_{t},t)}\transp-\Expec{u(X_{t},t)}m_{t}\transp \right) \drv t. \qquad\endproof
	  \end{aligned}
	 \end{equation*}

	 \begin{lemma} \label{lemma:multi_mix:m,S}
	   Mean $m_{t}$ and variance $S_{t}$ satisfy
	   \begin{equation*}
	     m_{t} = \sum_{\ell=1}^k \nu_\ell m^{(\ell)}_{t}, \quad 
	     S_{t} = \sum_{\ell=1}^k \nu_\ell S^{(\ell)}_{t} + \sum_{\ell=1}^k\nu_\ell m^{(\ell)}_{t} {m^{(\ell)}_{t}}\transp - \left( \sum_{\ell =1}^k \nu_\ell  m^{(\ell)}_{t}  \right)\left( \sum_{\ell =1}^k \nu_\ell  m^{(\ell)}_{t}  \right)\transp.
	   \end{equation*}
	  \end{lemma}
	 \indent \textit{Proof.}
	  The statement for the mean is straightforward. For the variance,
	   \begin{align*}
	      S_{t} 
	      %&= \CExpec{XX\transp}{p}-m_{t}m_{t}\transp \\
	      &= \sum_{\ell =1}^k \nu_\ell  \CExpec{XX\transp}{p_\ell }- \left( \sum_{\ell =1}^k \nu_\ell  m^{(\ell)}_{t}  \right)\left( \sum_{\ell =1}^k \nu_\ell  m^{(\ell)}_{t}  \right)\transp \\
	      &= \sum_{\ell =1}^k \nu_\ell  \left( \CExpec{XX\transp}{p_\ell }-m^{(\ell)}_{t} {m^{(\ell)}_{t}}^{-1} \right) + \sum_{\ell =1}^k \nu_\ell  m^{(\ell)}_{t} {m^{(\ell)}_{t}}^{-1} - \left( \sum_{\ell =1}^k \nu_\ell  m^{(\ell)}_{t}  \right)\left( \sum_{\ell =1}^k \nu_\ell  m^{(\ell)}_{t}  \right)\transp\endproof 
	      %&=\sum_{\ell=1}^k \nu_\ell S^{(\ell)}_{t} + \sum_{\ell=1}^k\nu_\ell m^{(\ell)}_{t} {m^{(\ell)}_{t}}\transp - \left( \sum_{\ell =1}^k \nu_\ell  m^{(\ell)}_{t}  \right)\left( \sum_{\ell =1}^k \nu_\ell  m^{(\ell)}_{t}  \right)\transp.\qquad\endproof
	   \end{align*}

	 \indent \textit{Proof of Theorem \ref{thm:mean_variance}}
	   Consider a drift function $u(x,t)$ given by \eqref{eq:ansatz_multi_mix}. In view of Lemma \ref{lemma:eq-for-variance}
	  \begin{align*}
	   &\frac{\drv m_{t}}{\drv t} 
	   %= \CExpec{\frac{1}{2}\diverg{a(X)}}{p} + \CExpec{\frac{\sum_{\ell=1}^k \nu_\ell p_\ell(x,\Theta^{(\ell)}_{t})A^{(\ell)}_{t}}{p(x,\Theta^{(1)}_t,\hdots,\Theta^{(k)}_t)}}{p} + \CExpec{\frac{\sum_{\ell=1}^k \nu_\ell p_\ell(x,\Theta^{(\ell)}_{t})B^{(\ell)}_{t} X}{p(x,\Theta^{(1)}_t,\hdots,\Theta^{(k)}_t)}}{p}\\ 
	   %&\hspace{8mm}+ \CExpec{\frac{\sum_{\ell=1}^k \nu_\ell p_\ell(x,\Theta^{(\ell)}_{t})a(X)C^{(\ell)}_{t}}{p(x,\Theta^{(1)}_t,\hdots,\Theta^{(k)}_t)}}{p} + \CExpec{\frac{\sum_{\ell=1}^k \nu_\ell p_\ell(x,\Theta^{(\ell)}_{t})a(X)D^{(\ell)}_{t} X}{p(x,\Theta^{(1)}_t,\hdots,\Theta^{(k)}_t)}}{p} \\
	   &= \sum_{\ell=1}^k \nu_\ell\bigg( \CExpec{\frac{1}{2}\diverg{a(X)}}{p_\ell} + A^{(\ell)}_{t} + B^{(\ell)}_{t} m^{(\ell)}_{t} + \CExpec{a(X)}{p_\ell}C^{(\ell)}_{t} + \CExpec{a(X)D^{(\ell)}_{t} X}{p_\ell}\bigg),
	  \end{align*}
	  which can be simplified according to Lemma \ref{lemma:multi_mix:m,S}, such that
	  \begin{equation} \label{e:proof:thm:dm}
	  \begin{aligned}
	   &\sum_{\ell=1}^k \nu_\ell \frac{\drv m^{(\ell)}_{t}}{\drv t} = \sum_{\ell=1}^k \nu_\ell\bigg( \CExpec{\frac{1}{2}\diverg{a(X)}}{p_\ell} + A^{(\ell)}_{t} \\
	   &\hspace{30mm} + B^{(\ell)}_{t} m^{(\ell)}_{t} + \CExpec{a(X)}{p_\ell}C^{(\ell)}_{t} + \CExpec{a(X)D^{(\ell)}_{t} X}{p_\ell}\bigg).
	  \end{aligned}
	  \end{equation}
	  Note that \eqref{e:proof:thm:dm} has to hold for all $\nu_\ell\geq 0$ such that $\sum_{\ell=1}^k \nu_\ell=1$. 
	  Therefore
	  \begin{equation} \label{e:proof:thm:m}
	  \begin{aligned}
	   \frac{\drv m^{(\ell)}_{t}}{\drv t} &= \CExpec{\frac{1}{2}\diverg{a(X)}}{p_\ell} + A^{(\ell)}_{t} + B^{(\ell)}_{t} m^{(\ell)}_{t} + \CExpec{a(X)}{p_\ell}C^{(\ell)}_{t} + \CExpec{a(X)D^{(\ell)}_{t} X}{p_\ell}.
	  \end{aligned}
	  \end{equation}
	 For the variance, we have according to Lemma \ref{lemma:multi_mix:m,S}
	\begin{align*}
	  \frac{\drv S_{t}}{\drv t} 
	  %&= \sum_{\ell=1}^k \nu_\ell \frac{\drv S^{(\ell)}_{t}}{\drv t} +  \sum_{\ell=1}^k \nu_\ell \left( \frac{\drv m^{(\ell)}_{t}}{\drv t}{m^{(\ell)}_{t}}\transp + m^{(\ell)}_{t} \left( \frac{\drv m^{(\ell)}_{t}}{\drv t} \right)\transp \right)  \\
	    %&\hspace{8mm} - \left( \sum_{\ell =1}^k \nu_\ell  \frac{\drv m^{(\ell)}_{t}}{\drv t}  \right)m_{t}\transp - m_{t}\left( \sum_{\ell =1}^k \nu_\ell  \frac{\drv m^{(\ell)}_{t}}{\drv t}  \right)\transp \\
	    &=\sum_{\ell=1}^k \nu_\ell \left( \frac{\drv S^{(\ell)}_{t}}{\drv t} +  \frac{\drv m^{(\ell)}_{t}}{\drv t}{m^{(\ell)}_{t}}\transp + m^{(\ell)}_{t} \left( \frac{\drv m^{(\ell)}_{t}}{\drv t} \right)\transp   -  \frac{\drv m^{(\ell)}_{t}}{\drv t}  m_{t}\transp - m_{t}\left(\frac{\drv m^{(\ell)}_{t}}{\drv t}  \right)\transp \right).
	\end{align*}
	 This implies
	 \begin{align*}
	  \sum_{\ell=1}^k \nu_\ell \frac{\drv S^{(\ell)}_{t}}{\drv t} = \frac{\drv S}{\drv t} + \sum_{\ell=1}^k \nu_\ell \left( \frac{\drv m^{(\ell)}_{t}}{\drv t}(m_{t}\transp-{m^{(\ell)}_{t}}\transp) + (m_{t}-m^{(\ell)}_{t})\left( \frac{\drv m^{(\ell)}_{t}}{\drv t}\right)\transp \right),
	 \end{align*}
	 where $\frac{\drv S_{t}}{\drv t}$ is given according to Lemma \ref{lemma:eq-for-variance}, by \vspace{-2mm}
	 \begin{align*}
	 \frac{\drv S_{t}}{\drv t} =  \Expec{X_{t}u(X_{t},t)\transp}+\Expec{u(X_{t},t)X_{t}\transp}+\Expec{\sigma(X_{t})\sigma(X_{t})\transp}  -m_{t}\left( \frac{\drv m_{t}}{\drv t} \right)\transp-\frac{\drv m_{t}}{\drv t}m_{t}\transp.   
	 \end{align*}
	  Therefore, \vspace{-2mm}
	 \begin{equation} \label{eq:dS_ell_first}
	 \begin{aligned}
	   \sum_{\ell=1}^k \nu_\ell \frac{\drv S^{(\ell)}_{t}}{\drv t} &= \Expec{X_{t}u(X_{t},t)\transp}+\Expec{u(X_{t},t)X_{t}\transp}+\Expec{\sigma(X_{t})\sigma(X_{t})\transp} \\
	   &\hspace{8mm} - \sum_{\ell=1}^k \nu_\ell \left( \frac{\drv m^{(\ell)}_{t}}{\drv t}{m^{(\ell)}_{t}}\transp + m^{(\ell)}_{t}\left( \frac{\drv m^{(\ell)}_{t}}{\drv t}\right)\transp \right).
	 \end{aligned} \vspace{-2mm}
	 \end{equation}
	 Recall that $\frac{\drv m^{(\ell)}_{t}}{\drv t}$ is given by \eqref{e:proof:thm:m}. Next, we compute \vspace{-2mm}
	  \begin{align*}
	   \Expec{Xu(X,t)\transp} 
	   %&= \sum_{\ell=1}^k \nu_\ell\bigg(  \CExpec{\frac{1}{2}X\diverg{a(X)}\transp}{p_\ell} +\CExpec{X {A^{(\ell)}_{t}}\transp}{p_\ell} + \CExpec{XX\transp {B^{(\ell)}_{t}}\transp}{p_\ell}\\ 
	   %&\hspace{8mm}+ \CExpec{X{C^{(\ell)}_{t}}\transp a(X)}{p_\ell} + \CExpec{XX\transp D^{(\ell)}_{t} a(X)}{p_\ell} \bigg) \\
	   &= \sum_{\ell=1}^k \nu_\ell\bigg(  \CExpec{\frac{1}{2}X\diverg{a(X)}\transp}{p_\ell} +m^{(\ell)}_{t} {A^{(\ell)}_{t}}\transp +  (m^{(\ell)}_{t} {m^{(\ell)}_{t}}\transp + S^{(\ell)}_{t}){B^{(\ell)}_{t}}\transp \\ 
	   &\hspace{8mm}+ \CExpec{X{C^{(\ell)}_{t}}\transp a(X)}{p_\ell} + \CExpec{XX\transp D^{(\ell)}_{t} a(X)}{p_\ell} \bigg), \\
	    \Expec{u(X,t)X\transp} 
	    %&=  \sum_{\ell=1}^k \nu_\ell\bigg(  \CExpec{\frac{1}{2}\diverg{a(X)}X\transp}{p_\ell} +\CExpec{A^{(\ell)}_{t} X\transp}{p_\ell} + \CExpec{ B^{(\ell)}_{t} XX\transp}{p_\ell}\\ 
	   %&\hspace{8mm}+ \CExpec{a(X)C^{(\ell)}_{t} X\transp}{p_\ell} + \CExpec{a(X) D^{(\ell)}_{t} XX\transp}{p_\ell} \bigg)\\
	    &= \sum_{\ell=1}^k \nu_\ell\bigg(  \CExpec{\frac{1}{2}\diverg{a(X)}X\transp}{p_\ell} +A^{(\ell)}_{t} {m^{(\ell)}_{t}}^{-1} +  B^{(\ell)}_{t}(m^{(\ell)}_{t} {m^{(\ell)}_{t}}\transp + S^{(\ell)}_{t})\\ 
	   &\hspace{8mm}+ \CExpec{a(X)C^{(\ell)}_{t} X\transp}{p_\ell} + \CExpec{a(X) D^{(\ell)}_{t} XX\transp}{p_\ell} \bigg), \\
	   \CExpec{\sigma(X)\sigma(X)\transp}{p} &= \CExpec{a(X)}{p} =\sum_{\ell=1}^k \nu_\ell \CExpec{a(X)}{p_\ell},  
	   \end{align*} 
	  such that by evaluating \eqref{eq:dS_ell_first} and by recalling that it has to hold for all convex combinations, we get the assertion \eqref{e:thm:var}. \qquad\endproof
%	   \begin{align*}
%	  &\frac{\drv S^{(\ell)}_{t}}{\drv t} = \frac{1}{2}\CExpec{X\diverg{a(X)}\transp}{p_\ell}+ \frac{1}{2}\CExpec{\diverg{a(X)} X\transp}{p_\ell} - \frac{1}{2}m^{(\ell)}_{t} \CExpec{\diverg{a(X)}}{p_\ell}\transp\\ 
%	     &\hspace{8mm}- \frac{1}{2}\CExpec{\diverg{a(X)}}{p_\ell}{m^{(\ell)}_{t}}\transp + \CExpec{a(X)}{p_\ell} + S^{(\ell)}_{t} {B^{(\ell)}_{t}}\transp + B^{(\ell)}_{t} S^{(\ell)}_{t} + \CExpec{X{C^{(\ell)}_{t}}\transp a(X)}{p_\ell}\\
%	     &\hspace{8mm}+ \CExpec{a(X)C^{(\ell)}_{t} X\transp}{p_\ell} - m^{(\ell)}_{t} {C^{(\ell)}_{t}}\transp \CExpec{a(X)}{p_\ell} - \CExpec{a(X)}{p_\ell}C^{(\ell)}_{t} {m^{(\ell)}_{t}}\transp\\
%	     &\hspace{8mm}+\CExpec{XX\transp D^{(\ell)}_{t} a(X)}{p_\ell} + \CExpec{a(X) D^{(\ell)}_{t} XX\transp}{p_\ell} - m^{(\ell)}_{t} \CExpec{X\transp D^{(\ell)}_{t} a(X)}{p_\ell}\\ 
%	     &\hspace{8mm}-\CExpec{a(X)D^{(\ell)}_{t} X}{p_\ell}{m^{(\ell)}_{t}}\transp.  \qquad\endproof
%	 \end{align*}  

\end{appendix}

\vskip 0.2in
\bibliography{ref}

\begin{thebibliography}{36}
\providecommand{\natexlab}[1]{#1}
\providecommand{\url}[1]{\texttt{#1}}
\expandafter\ifx\csname urlstyle\endcsname\relax
  \providecommand{\doi}[1]{doi: #1}\else
  \providecommand{\doi}{doi: \begingroup \urlstyle{rm}\Url}\fi

\bibitem[Andrieu et~al.(2010)Andrieu, Doucet, and Holenstein]{ref:Andrieu-10}
Christophe Andrieu, Arnaud Doucet, and Roman Holenstein.
\newblock Particle {M}arkov chain {M}onte {C}arlo methods.
\newblock \emph{J. R. STAT. SOC. Series B}, 72\penalty0 (3):\penalty0 269--342,
  2010.
\newblock ISSN 1369-7412.
\newblock URL \url{http://dx.doi.org/10.1111/j.1467-9868.2009.00736.x}.

\bibitem[Archambeau and Opper(2011)]{ref:Archambeau-11}
C\'edric Archambeau and Manfred Opper.
\newblock Approximate inference for continuous-time {M}arkov processes.
\newblock In \emph{Bayesian Time Series Models}, pages 125--€--140. Cambridge
  University Press, 2011.
\newblock ISBN 9780521196765.

\bibitem[Archambeau et~al.(2007)Archambeau, Cornford, Opper, and
  Shawe-Taylor]{ref:Archambeau-07}
C\'edric Archambeau, Dan Cornford, Manfred Opper, and John Shawe-Taylor.
\newblock Gaussian process approximations of stochastic differential equations.
\newblock In \emph{Gaussian Processes in Practice}, volume~1 of \emph{JMLR
  Proceedings}, pages 1--16, 2007.

\bibitem[Archambeau et~al.(2008)Archambeau, Opper, Shen, Cornford, and
  Shawe-Taylor]{ref:Archambeau-08}
C\'edric Archambeau, Manfred Opper, Yuan Shen, Dan Cornford, and John
  Shawe-Taylor.
\newblock Variational inference for diffusion processes.
\newblock In \emph{Advances in Neural Information Processing Systems 20}, pages
  17--24. MIT Press, Cambridge, MA, 2008.

\bibitem[Bain and Crisan(2009)]{ref:Bain-09}
Alan Bain and Dan Crisan.
\newblock \emph{Fundamentals of stochastic filtering}.
\newblock Stochastic Modelling and Applied Probability. Springer, New York,
  2009.
\newblock ISBN 978-0-387-76895-3.

\bibitem[Bernstein(2009)]{ref:Ber-09}
D.~S. Bernstein.
\newblock \emph{Matrix {M}athematics}.
\newblock Princeton University Press, 2 edition, 2009.

\bibitem[Brigo(2000)]{ref:Brigo-00}
Damiano Brigo.
\newblock On {S}{D}{E}s with marginal laws evolving in finite-dimensional
  exponential families.
\newblock \emph{Statistics and Probability Letters}, 49\penalty0 (2):\penalty0
  127--134, 2000.

\bibitem[Capp{\'e} et~al.(2005)Capp{\'e}, Moulines, and
  Ryd{\'e}n]{ref:Cappe-05}
Olivier Capp{\'e}, Eric Moulines, and Tobias Ryd{\'e}n.
\newblock \emph{Inference in hidden {M}arkov models}.
\newblock Springer Series in Statistics. Springer, New York, 2005.
\newblock ISBN 978-0387-40264-2; 0-387-40264-0.

\bibitem[Clarke(2013)]{ref:Clarke-13}
Francis Clarke.
\newblock \emph{Functional analysis, calculus of variations and optimal
  control}, volume 264 of \emph{Graduate Texts in Mathematics}.
\newblock Springer, London, 2013.
\newblock ISBN 978-1-4471-4819-7; 978-1-4471-4820-3.
\newblock URL \url{http://dx.doi.org/10.1007/978-1-4471-4820-3}.

\bibitem[Cox et~al.(1985)Cox, Ingersoll, and Ross]{ref:Cox-85}
John~C. Cox, Jonathan~E. Ingersoll, Jr., and Stephen~A. Ross.
\newblock A theory of the term structure of interest rates.
\newblock \emph{Econometrica}, 53\penalty0 (2):\penalty0 385--407, 1985.
\newblock ISSN 0012-9682.
\newblock URL \url{http://dx.doi.org/10.2307/1911242}.

\bibitem[Crisan and Lyons(1999)]{ref:Crisan-99}
D.~Crisan and T.~Lyons.
\newblock A particle approximation of the solution of the
  {K}ushner-{S}tratonovitch equation.
\newblock \emph{Probab. Theory Related Fields}, 115\penalty0 (4):\penalty0
  549--578, 1999.
\newblock ISSN 0178-8051.
\newblock URL \url{http://dx.doi.org/10.1007/s004400050249}.

\bibitem[Cseke et~al.(2013)Cseke, Opper, and Sanguinetti]{ref:Cseke-13}
Botond Cseke, Manfred Opper, and Guido Sanguinetti.
\newblock Approximate inference in latent gaussian-markov models from
  continuous time observations.
\newblock In C.~J.~C. Burges, L.~Bottou, M.~Welling, Z.~Ghahramani, and K.~Q.
  Weinberger, editors, \emph{Advances in Neural Information Processing Systems
  26}, pages 971--979. Curran Associates, Inc., 2013.

\bibitem[Csisz{\'a}r(1975)]{ref:Csiszar-75}
I.~Csisz{\'a}r.
\newblock {$I$}-divergence geometry of probability distributions and
  minimization problems.
\newblock \emph{Ann. Probability}, 3:\penalty0 146--158, 1975.

\bibitem[Del~Moral et~al.(2001)Del~Moral, Jacod, and Protter]{ref:DelMoral-01}
Pierre Del~Moral, Jean Jacod, and Philip Protter.
\newblock The {M}onte-{C}arlo method for filtering with discrete-time
  observations.
\newblock \emph{Probab. Theory Related Fields}, 120\penalty0 (3):\penalty0
  346--368, 2001.
\newblock ISSN 0178-8051.
\newblock URL \url{http://dx.doi.org/10.1007/PL00008786}.

\bibitem[Dmitruk and Kaganovich(2008)]{ref:Dmitruk-08}
A.V. Dmitruk and A.M. Kaganovich.
\newblock The hybrid maximum principle is a consequence of {P}ontryagin maximum
  principle.
\newblock \emph{Systems and Control Letters}, 57\penalty0 (11):\penalty0 964 --
  970, 2008.
\newblock ISSN 0167-6911.
\newblock URL
  \url{http://www.sciencedirect.com/science/article/pii/S016769110800100X}.

\bibitem[{Eyink}(2000)]{ref:Eyink-00}
G.~L. {Eyink}.
\newblock {A Variational Formulation of Optimal Nonlinear Estimation}.
\newblock \emph{ArXiv Physics e-prints}, November 2000.
\newblock URL \url{http://arxiv.org/abs/physics/0011049v2}.

\bibitem[Garcia and Palacios(2001)]{ref:Garcia-01}
Nancy~Lopes Garcia and Jos{\'e}~Luis Palacios.
\newblock On inverse moments of nonnegative random variables.
\newblock \emph{Statist. Probab. Lett.}, 53\penalty0 (3):\penalty0 235--239,
  2001.
\newblock ISSN 0167-7152.
\newblock URL \url{http://dx.doi.org/10.1016/S0167-7152(01)00008-6}.

\bibitem[Harel et~al.(2015)Harel, Meir, and Opper]{ref:Yuval-15}
Yuval Harel, Ron Meir, and Manfred Opper.
\newblock A tractable approximation to optimal point process filtering:
  Application to neural encoding.
\newblock In C.~Cortes, N.~D. Lawrence, D.~D. Lee, M.~Sugiyama, and R.~Garnett,
  editors, \emph{Advances in Neural Information Processing Systems 28}, pages
  1603--1611. Curran Associates, Inc., 2015.

\bibitem[Kallenberg(2002)]{ref:Kallenberg-02}
Olav Kallenberg.
\newblock \emph{Foundations of modern probability}.
\newblock Probability and its Applications (New York). Springer-Verlag, New
  York, second edition, 2002.
\newblock ISBN 0-387-95313-2.
\newblock URL \url{http://dx.doi.org/10.1007/978-1-4757-4015-8}.

\bibitem[Kushner(1967)]{ref:Kushner-67}
H.~J. Kushner.
\newblock Dynamical equations for optimal nonlinear filtering.
\newblock \emph{J. Differential Equations}, 3:\penalty0 179--190, 1967.
\newblock ISSN 0022-0396.

\bibitem[Kushner and Dupuis(2001)]{ref:Kushner-01}
Harold~J. Kushner and Paul Dupuis.
\newblock \emph{Numerical methods for stochastic control problems in continuous
  time}, volume~24 of \emph{Applications of Mathematics (New York)}.
\newblock Springer-Verlag, New York, second edition, 2001.
\newblock ISBN 0-387-95139-3.
\newblock URL \url{http://dx.doi.org/10.1007/978-1-4613-0007-6}.
\newblock Stochastic Modelling and Applied Probability.

\bibitem[Lasserre(2010)]{ref:Las-10}
J.~B. Lasserre.
\newblock \emph{Moments, {P}ositive {P}olynomials and {T}heir {A}pplications},
  volume~1 of \emph{Imperial College Press Optimization Series}.
\newblock Imperial College Press, London, 2010.

\bibitem[Lasserre(2001)]{ref:Lasserre-01}
Jean~B. Lasserre.
\newblock Global optimization with polynomials and the problem of moments.
\newblock \emph{SIAM Journal on Optimization}, 11\penalty0 (3):\penalty0
  796--817, 2001.
\newblock URL \url{http://dx.doi.org/10.1137/S1052623400366802}.

\bibitem[Mitter and Newton(2003)]{ref:Mitter-03}
Sanjoy~K. Mitter and Nigel~J. Newton.
\newblock A variational approach to nonlinear estimation.
\newblock \emph{SIAM J. Control Optim.}, 42\penalty0 (5):\penalty0 1813--1833,
  2003.
\newblock ISSN 0363-0129.
\newblock URL \url{http://dx.doi.org/10.1137/S0363012901393894}.

\bibitem[{\O}ksendal(2003)]{ref:Oksendal-03}
Bernt {\O}ksendal.
\newblock \emph{Stochastic differential equations}.
\newblock Universitext. Springer, 6. ed. edition, 2003.
\newblock ISBN 3-540-04758-1.

\bibitem[Pag{\`e}s and Pham(2005)]{ref:pages-05}
Gilles Pag{\`e}s and Huy{\^e}n Pham.
\newblock Optimal quantization methods for nonlinear filtering with
  discrete-time observations.
\newblock \emph{Bernoulli}, 11\penalty0 (5):\penalty0 893--932, 2005.
\newblock ISSN 1350-7265.
\newblock URL \url{http://dx.doi.org/10.3150/bj/1130077599}.

\bibitem[Pardoux(1981/82)]{ref:Pardoux-81}
E.~Pardoux.
\newblock \'{E}quations du filtrage non lin\'eaire, de la pr\'ediction et du
  lissage.
\newblock \emph{Stochastics}, 6\penalty0 (3-4):\penalty0 193--231, 1981/82.
\newblock ISSN 0090-9491.
\newblock URL \url{http://dx.doi.org/10.1080/17442508208833204}.

\bibitem[Pinski et~al.(2015)Pinski, Simpson, Stuart, and Weber]{ref:Pinski-15}
F.~J. Pinski, G.~Simpson, A.~M. Stuart, and H.~Weber.
\newblock Kullback-leibler approximation for probability measures on infinite
  dimensional spaces.
\newblock \emph{SIAM Journal on Mathematical Analysis}, 47\penalty0
  (6):\penalty0 4091--4122, 2015.
\newblock \doi{10.1137/140962802}.
\newblock URL \url{http://dx.doi.org/10.1137/140962802}.

\bibitem[Shiryaev(1999)]{ref:Shiryaev-99}
Albert~N. Shiryaev.
\newblock \emph{Essentials of stochastic finance}, volume~3 of \emph{Advanced
  Series on Statistical Science \& Applied Probability}.
\newblock World Scientific Publishing Co., Inc., River Edge, NJ, 1999.
\newblock ISBN 981-02-3605-0.
\newblock URL \url{http://dx.doi.org/10.1142/9789812385192}.
\newblock Facts, models, theory, Translated from the Russian manuscript by N.
  Kruzhilin.

\bibitem[Stoer and Bulirsch(2002)]{ref:Stoer-02}
J.~Stoer and R.~Bulirsch.
\newblock \emph{Introduction to numerical analysis}, volume~12 of \emph{Texts
  in Applied Mathematics}.
\newblock Springer-Verlag, New York, third edition, 2002.
\newblock ISBN 0-387-95452-X.
\newblock URL \url{http://dx.doi.org/10.1007/978-0-387-21738-3}.

\bibitem[Stratonovich(1960)]{ref:Stratonovich-60}
R.~L. Stratonovich.
\newblock Conditional {M}arkov processes.
\newblock \emph{Theory of Probability \& Its Applications}, 5\penalty0
  (2):\penalty0 156--178, 1960.
\newblock URL \url{http://dx.doi.org/10.1137/1105015}.

\bibitem[{Sutter} et~al.(2015){Sutter}, {Ganguly}, and
  {Koeppl}]{ref:SutterVar-15}
Tobias {Sutter}, Arnab {Ganguly}, and Heinz {Koeppl}.
\newblock {A variational approach to path estimation and parameter inference of
  hidden diffusion processes}.
\newblock \emph{ArXiv e-prints}, August 2015.

\bibitem[van Handel(2007)]{ref:vanHandel-07}
R.~van Handel.
\newblock \emph{Filtering, Stability, and Robustness}.
\newblock PhD thesis, Caltech, 2007.

\bibitem[Vrettas et~al.(2015)Vrettas, Opper, and Cornford]{ref:Opper-15}
Michail~D. Vrettas, Manfred Opper, and Dan Cornford.
\newblock Variational mean-field algorithm for efficient inference in large
  systems of stochastic differential equations.
\newblock \emph{Phys. Rev. E}, 91:\penalty0 012148, Jan 2015.
\newblock \doi{10.1103/PhysRevE.91.012148}.
\newblock URL \url{http://link.aps.org/doi/10.1103/PhysRevE.91.012148}.

\bibitem[Wilkinson(2006)]{ref:Wilkinson-06}
Darren~James Wilkinson.
\newblock \emph{Stochastic modelling for systems biology}.
\newblock Chapman \& Hall/CRC Mathematical and Computational Biology Series.
  Chapman \& Hall/CRC, Boca Raton, FL, 2006.
\newblock ISBN 978-1-58488-540-5; 1-58488-540-8.

\bibitem[Zakai(1969)]{ref:Zakai-69}
Moshe Zakai.
\newblock On the optimal filtering of diffusion processes.
\newblock \emph{Zeitschrift f\"ur Wahrscheinlichkeitstheorie und Verwandte
  Gebiete}, 11\penalty0 (3):\penalty0 230--243, 1969.
\newblock ISSN 0044-3719.
\newblock URL \url{http://dx.doi.org/10.1007/BF00536382}.

\end{thebibliography}

\end{document}